\documentclass[12pt]{amsart}
\usepackage{latexsym,amscd,amssymb,epsfig, enumitem, amsmath, tikz, bm, curves} 
\usepackage{booktabs}  % professionally typeset tables
\usepackage{amsfonts}
\usepackage{amstext}
\usepackage{amsthm}
\usepackage{textcomp}  % better copyright sign, among other things
\usepackage{xcolor}
\usepackage{amsmath}
\usepackage{tikz}
\usepackage{bm}
\usepackage{amsthm}
\usepackage{lscape} 
\usepackage{xargs} 
\usepackage[ruled, algosection]{algorithm2e}
\usepackage[normalem]{ulem}
\usepackage{centernot}
\usetikzlibrary{arrows.meta}
\usetikzlibrary{calc}
\usetikzlibrary{decorations.pathreplacing}
\usetikzlibrary{arrows.meta,calc, shapes.geometric}
\usepackage[colorinlistoftodos]{todonotes}

\theoremstyle{plain}
  \newtheorem{thm}{Theorem}[section]
  \newtheorem{prop}[thm]{Proposition}
  \newtheorem{lem}[thm]{Lemma}
  \newtheorem{cor}[thm]{Corollary}

\theoremstyle{definition}
  \newtheorem{defn}[thm]{Definition}
  \newtheorem{ex}[thm]{Example}
  
  \newtheorem{qstn}[thm]{Question}
\theoremstyle{remark}
  \newtheorem{rmk}[thm]{Remark}

\makeatletter
\let\c@thm\c@algocf
\makeatother

\numberwithin{equation}{section}

\def\reals{{\mathbb R}}

\begin{document}

\title[Generalized recursive atom ordering]{Generalized recursive atom ordering and equivalence  to CL-shellability}

\author{Patricia Hersh}
\address{Department of Mathematics, University of Oregon, Eugene, OR, 97403}
\email{plhersh@uoregon.edu}

\author{Grace Stadnyk}
\address{Department of Mathematics, Furman University, Greenville, SC 29613} 
\email{grace.stadnyk@furman.edu}

\thanks{P. Hersh  is supported by NSF grant DMS-1953931.}

\begin{abstract}
Bj\"orner and Wachs introduced 
 CL-shellability as a technique for studying the topological structure of order complexes of partially ordered sets.  They also introduced the 
 notion of recursive atom ordering, and   
 they proved  that a finite bounded poset   is CL-shellable if and only if it admits a recursive atom ordering. 

In this paper, a generalization  of  the notion of recursive atom ordering is introduced.  A finite bounded poset is proven to  admit  
such a generalized recursive atom ordering if and only if it admits  a traditional   recursive atom ordering.
This is also proven equivalent to admitting  a CC-shelling  
(a type of shelling   introduced by Kozlov) with a  further property  called 
 self-consistency.   Thus, CL-shellability  is proven  equivalent to self-consistent CC-shellability.   
 As an application,  
 the uncrossing posets, namely  the face posets for  stratified spaces of planar electrical networks, are proven to be dual CL-shellable.   
\end{abstract}

\maketitle 

\section{Introduction} 
%\commentph{Perhaps we should for now just keep ``generalized recursive atom ordering'' but say in the response to the referee that we are somewhat hesitant to change the name given that the paper is already at the arxiv with this in the title and that we have already given talks with this name, but that if the referee feels strongly we would be willing to change to ``relaxed recursive atom ordering'' -- this gives them another chance to reconsider but also allows them to insist on this change if they really feel strongly (or for the editor to intervene).  This would also allow us to send back to the journal without that further piece of work to do right now.  Just an idea...}
This paper introduces a new  tool % \sout{technique} 
for studying the topological structure of 
order complexes
of finite  partially ordered sets (posets). % \sout{ namely we introduce 
%generalized recursive atom orderings. } \gs{
This tool, called  generalized recursive atom ordering, %  \sout{The objects that these}
%\commentph{ generalizes }
 %\commentph{ encompasses }
  is a relaxation of  % \commentph{If we change ``generalized'' to ``relaxed'' then we might want to change ``is a relaxation of...'' to ``encompasses''}
%That's fine if we stick with generalized recursive atom ordering but too repetitive here if we change the name to relaxed recursive atom ordering}  
%\sout{,} 
the  %\gs{
fundamental and widely used  recursive atom ordering technique  introduced by % \sout{of}
 Bj\"orner and Wachs % \sout{(introduced in \cite{bw}) 
 %introduced}
  in \cite{bw}. 
  % \sout{are one of the most fundamental and widely used methods available for studying the topological structure of order complexes (also known as nerves) of posets. }
Any recursive atom ordering  (a notion that is reviewed  in Section \ref{bg-section}) of  %\sout{any} 
a finite bounded poset  gives rise to a lexicographic shelling for the poset, 
thereby % \commentplh{
yielding the result %}  \commentplh{implying} % \sout{proving} 
that  the order complex of the  poset  % \sout{(a   simplicial complex whose $i$-dimensional  faces are the chains of  $i+1$ comparable poset elements) }
 is either homotopy equivalent to a wedge of spheres or contractible. 
 % \commentph{Moreover, this shelling gives a way to count how many $i$-dimensional spheres there are in this wedge of spheres for each $i\ge 0$. }
 %  \commentph{ambiguous next bit} \sout{, doing so in a way that 
%enumerates  the spheres of each dimension.    }

We establish  a number of fundamental properties of these generalized recursive atom orderings  %\commentph{ Change to (RRAOs)?} 
(GRAOs), including the property  that any generalized recursive atom ordering may be transformed into a traditional recursive atom ordering (RAO)  by a process we introduce called the atom reordering process.    Since GRAOs are 
easier to construct than RAOs, this may give a useful new pathway to proving  a poset is CL-shellable. 
These  generalized recursive atom orderings further  allow us to prove  that several different forms of % \commentph{
lexicographic shellability  %\commentph{(in the not necessarily graded case)}
  are all equivalent to each other, by which 
we mean  that  a finite bounded poset admits any one of these  types of lexicographic shelling if and only if it  admits  each of the others.
%}.\sout{ every   other one of these types of lexicographic shelling\sout{;} }
One might expect this to 
imply the stronger statement that any instance of any one of these types of lexicographic shelling is also an instance of any other of these types of lexicographic shelling, but this 
%\sout{stronger statement} 
is not always true. 
%\sout{ (for instance it fails for ``self-consistent CC-shellability'' and CL-shellability by virtue of not every generalized recursive atom ordering being
%a recursive atom ordering).} 
%\commentph{Here is my stab at how to fix this: ``
%\commentph{This next sentence is only true if we regard GRAOs and RAOs as chain-atom orders.
%, but it seems too technical to mention that here. 
% I added some clarification  about this point  at the very start of Section 3, in the remark just before the definition of GRAO -- at which point we have the definition of chain-atom ordering at our disposal.}
  %Probably we should leave this as it is, but I am thinking about it, in case there is a nice way to avoid this issue.  The issue is that we actually prove that any atom ordering which may be extended to a GRAO may also be extended to an RAO.}
    For instance, one may deduce that not every ``self consistent CC-shelling'' is a CL-shelling from the fact 
    that not every generalized recursive atom ordering is a recursive atom ordering;  
    this latter fact is 
    stated  more precisely in Remark ~\ref{more-general}, and  % \commentph{(when we regard these as so-called chain-atom orderings, not just as orderings on atoms)}; 
an example demonstrating this latter fact appears in Figure 2.
%} \gs{For instance, while any poset that is (self-consistently?) CC-shellable in a self-consistent way is also a poset that is CL-shellable, the shelling orders that arise (are constructed?) from these classifications are not necessarily the same.
% }
These equivalence results  clarify 
the hierarchy of different  techniques for proving  that a finite bounded poset  %\sout{to be} 
is lexicographically shellable. 
 Figure \ref{implications-figure} gives a schematic of  many of the   implications proven in this paper,   
 and it points  readers  (by way of  the labels on the implication arrows) 
 to  %\sout{the result number} 
 where each result is  proven in this paper or elsewhere in the literature.

\begin{figure}\label{implications-figure}
\begin{center}
\begin{tikzpicture}
\draw  (1, 0) node [align=center,minimum size=3cm]  {P admits \\ GRAO}; 
\draw  (1, -5) node [align=center,minimum size=3cm] {P admits \\ RAO}; 
\draw (-4, 5) node[align=center,minimum size=3cm] {P is CC-shellable \\ in a \\self-consistent manner};
\draw (1, 5) node[align=center,minimum size=3cm] {P is CC-shellable \\via labeling  \\ with UE property}; 
\draw (-4, 2) node[align=center,minimum size=3cm] {P is \\CC-shellable}; 
\draw (-4, -0.5) node[align=center,minimum size=3cm] {P is \\ TCL-shellable}; 
\draw (-4, -3) node[align=center,minimum size=3cm] {P is \\ shellable}; 
\draw (5, 5) node[align=center,minimum size=3cm] {P is TCL-shellable in \\a self-consistent \\manner}; 
\draw (-9, 2) node[align=center,minimum size=3cm] {P is \\ EC-shellable}; 
\draw (-9, -5) node[align=center,minimum size=3cm] {P is \\ EL-shellable}; 
\draw (-4, -5) node[align=center,minimum size=3cm] {P is \\ CL-shellable}; 
\draw [double, thick, <->](1, -0.5)--node [midway, right]{\ref{cccl}}(1,-4); 
\draw [thick, double, ->](1, 0.5)-- node [midway, right] {\ref{altgraothencc}}(1,4); 
%\draw [double, thick, <-](-0.2, 0.5)-- node [near end, right, align=center,minimum size=2cm] {\ref{CCimpliesTopolCL} then \ref{TopolCLthengrao}
\draw [thick, double, ->] (-4, 6) to [out=20, in=160] node [midway, below]{\ref{CCimpliesTopolCL}} (5, 6) ; 
%; \\ \ref{grao2thenCC}

\draw [double, thick, <-](1.2, 0.5)--node [midway, right=0.5em]{%\ref{grao2thenCC} then \ref{CCimpliesTopolCL}; \\ 
\ref{TopolCLthengrao}}(4.5,4); 
\draw [thick, double, ->] (-4, 4) -- (-4,2.5); 
\draw [thick, double, ->] (-4, -1.25) -- node [midway, right]{\ref{tclthenshell}} (-4,-2.5); 
\draw [thick, double, ->] (-4, 1.5)--node [midway, right]{\ref{CCimpliesTopolCL}}(-4,0.25); 
\draw [thick, double, ->] (-7.75, 2)--node [align=center, minimum size=2cm, midway, above=-1.5em]{\cite{kozlov}}(-5.25, 2); 
\draw [thick, double,->] (-7.75, -5)--node [align=center, minimum size=2cm, midway, below=-1.5em] {\cite[2.3]{bw}}(-5.25, -5); 
\draw [double, thick, <->] (-2.75, -5)--node [align=center, minimum size=2cm, midway, below=-1.5em] {\cite[3.2]{bw}}(0,-5); 
\draw [thick, double,->] (-0.4, 5)--node [midway, above] {\ref{UEimpliesSC}} (-2.7, 5); 
\draw [thick, double, ->] (-4.75, -4.5) to [out=150, in=210] node [midway, left]{\ref{RAOimpliesCC}} (-4.75, 1.5) ; 

\end{tikzpicture}
\end{center}
\caption{Implications and where they are proven.}
\end{figure}
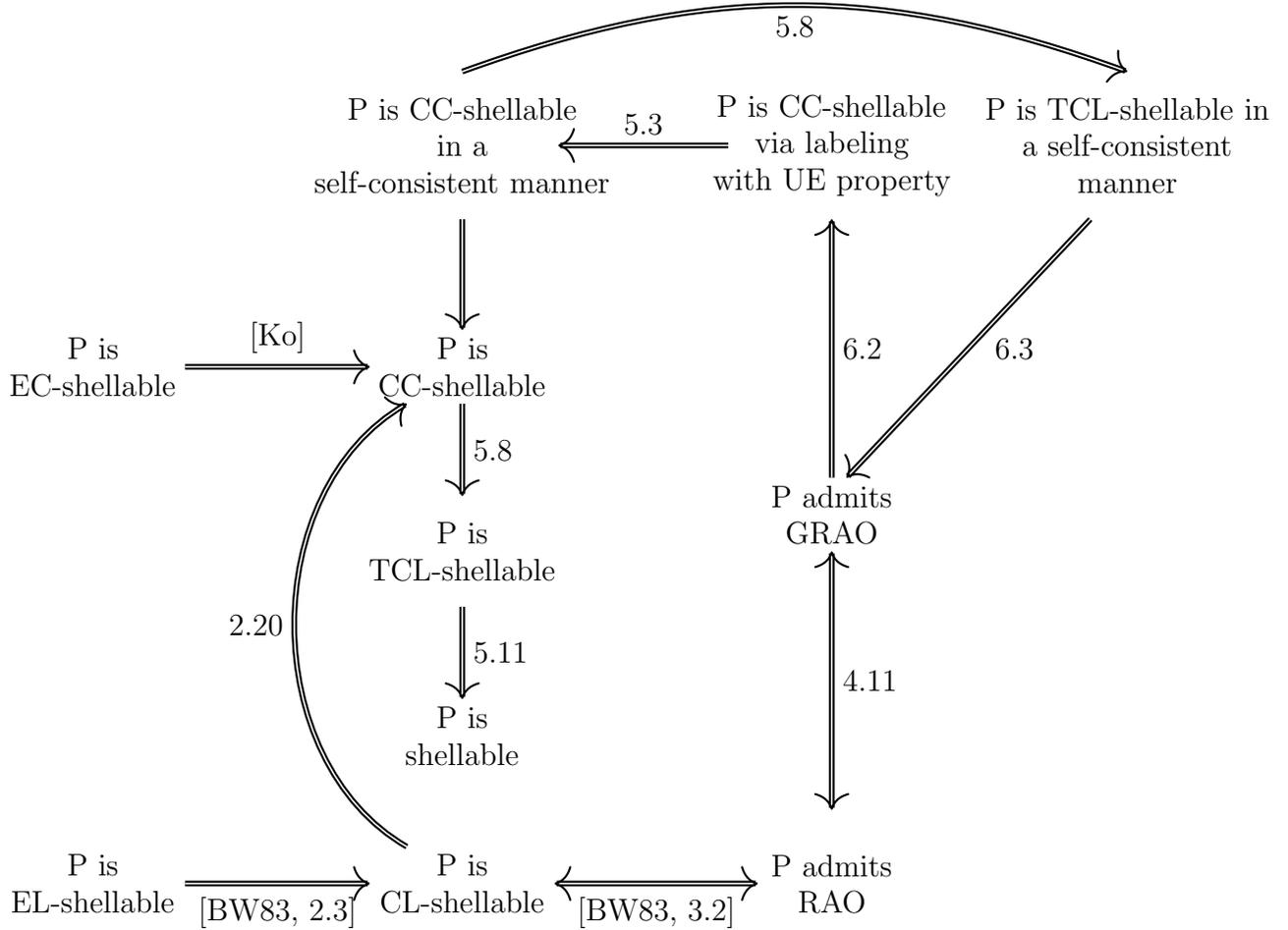

The importance of the  notion of recursive atom ordering  stems from the fact  that a finite bounded poset admits a recursive atom ordering if and only if it is CL-shellable. %}
EL-shellability  and 
CL-shellability are  the original  and  %\commentph{as well as being  the}   \commentplh{and continue to be the predominant}
 %\gs{and  the 
 predominant techniques  % \sout{(and still  the predominant ones)}
  for proving  that  posets  %\commentph{``to be'' is not incorrect grammatically and to me seems easier to parse here, since it is more consistent with the subjunctive tense we are in here}
   are 
%\sout{lexicographically} 
 shellable.
EL-shellability was first introduced by Bj\"orner in \cite{bjornercm}.  %\sout{with this} \commentph{Start a new sentence here for CL-shellability?}   \gs{and was} 
It  was generalized to the notion of   CL-shellability 
by Bj\"orner and Wachs in \cite{bw} % \commentph{timing unclear}
 when  they observed that their idea for a way to shell Bruhat order
%\sout{  construction of a} %  \sout{lexicographic} 
%\sout{ shelling for  Bruhat order}  \sout{in \cite{BW82}} % \sout{which}
 did not meet the requirements of an 
EL-shelling but nonetheless gave rise to a shelling % \sout{of the order complex}  
of a very similar flavor  to an EL-shelling. 
%\commentplh{I think it's important here to say ``reformulation'' so it is clear that we have an if and only if.  I also think it's important to discuss that this is useful for posets with recursive structure -- this paragraph motivates what we plan to do.  I don't think what was previously written here was too technical, but rather helps the reader know when one would want to use RAO (and likewise when GRAO may be called for).} \gs{
Recursive atom orderings were  % \sout{also} 
 introduced  in \cite{bw} as an % \commentph{
 alternative approach %} \sout{entirely different pathway} 
 to proving 
 CL-shellability.
 %  \sout{A reformulation of CL-shellability   
%as a recursive property that  a poset might  have, namely a  ``recursive atom ordering",  was also introduced in \cite{bw}.} \sout{This gave  an entirely different pathway to proving 
 %CL-shellability, one  that is  well suited to handling   posets with recursive structure inside of  them.  } 
  %\commentph{ Maybe add back  something like ``
  For families of  posets with inherent recursive structure, such as the partition lattice, 
  %(where intervals in these posets are themselves products of  smaller  posets in the same family), 
  recursive atom ordering can often be 
  the easiest and  most natural way to prove these posets are shellable.  
  % % are intrinsically recursive in nature and as such 
 % are particularly well-suited to posets with recursive structure to them; for instance, every upper interval in the partition lattice is isomorphic to a smaller partition lattice which makes it relatively easy to construct a recursive atom ordering on the partition lattice.''} 
 Both CL-shellability and the  related notion 
 of recursive atom ordering were extended to the non-graded case in \cite{non-pure1} and \cite{non-pure2}.
 %, \commentph{the setting we will work in throughout this paper as well}. 
 %\commentgs{Move this to later?/background} \sout{We will work in the non-graded  %(or more precisely, not necessarily graded)   
% case throughout this paper.  }
In \cite{kozlov}, EL-shellability and CL-shellability were generalized to the notions of EC-shellability and CC-shellability, which allowed more flexibility in constructing labelings that could be used to establish lexicographic shellability. Some important examples of posets that have been proven to be lexicographically shellable (or dual lexicographically shellable), in some cases 
 by way of a recursive  atom ordering,  
  include  Bruhat order (\cite{BW82}), posets with exponential structures (\cite{Sa}), supersolvable lattices (\cite{bjornercm}), geometric lattices  (\cite{bjornercm}), geometric semilattices (\cite{WW}), various posets from finite group theory (\cite{Sh01}, \cite{Wo}) and combinatorial commutative algebra (\cite{PRS})),  intersection posets of $k$-equal subspace arrangements (\cite{non-pure1}),  and face posets of  shellable 
 $d$-complexes (\cite{Bj84}).

\begin{figure} 
\begin{tikzpicture}[scale=0.8]
\draw [fill] (4,0) circle [radius=0.05]; 
\draw [fill] (1,2) circle [radius=0.05]; 
\node [left] at (1, 2) {\small \textbf{}}; 
\draw [fill] (4,2) circle [radius=0.05]; 
\node [below right] at (4, 2) {\footnotesize \textbf{a}}; 
\draw [fill] (7,2) circle [radius=0.05]; 
\draw [fill] (0,4) circle [radius=0.05]; 
\draw [fill] (2,4) circle [radius=0.05]; 
\draw [fill] (4,4) circle [radius=0.05]; 
\draw [fill] (6,4) circle [radius=0.05]; 
\draw [fill] (8,4) circle [radius=0.05]; 
\draw [fill] (1,6) circle [radius=0.05]; 
\draw [fill] (4,6) circle [radius=0.05]; 
\draw [fill] (7,6) circle [radius=0.05]; 
\draw [fill] (4,8) circle [radius=0.05]; 
%\node [above right] at (4,8){\footnotesize $\hat{1}$};

\draw (4,0)--node [midway,  circle, fill=white, inner sep=0.5pt,minimum size=1pt]{\scriptsize 1}(1,2)--node [midway, circle, fill=white, inner sep=0.5pt,minimum size=1pt]{\scriptsize 1}(0,4)--node [midway, circle, fill=white, inner sep=0.5pt,minimum size=1pt]{\scriptsize 1}(1,6)--(4,8); 

\draw (1,2)--node [circle, fill=white, midway, inner sep=0.5pt,minimum size=1pt]{\scriptsize 2}(2,4)--node [ near start, circle, fill=white, inner sep=0.5pt,minimum size=1pt ]{\scriptsize 1}(1,6); 
\draw (1,2)--node [near start, circle, fill=white, inner sep=0.5pt,minimum size=1pt]{\scriptsize 3}(4,4)--node [near start, circle, fill=white, inner sep=0.5pt,minimum size=1pt]{\scriptsize \textcolor{red} 3}(1,6);

\draw (4,0)--node [midway, circle, fill=white, inner sep=0.5pt,minimum size=1pt ]{\scriptsize 2}(4,2)--node [ near start, circle, fill=white, inner sep=0.5pt,minimum size=1pt,  ]{\scriptsize 1}(2,4)--node [  near start, circle, fill=white, inner sep=0.5pt,minimum size=1pt ]{\scriptsize 2}(4,6)--(4,8); 
\draw (4,2)--node [midway, circle, fill=white, inner sep=0.5pt,minimum size=1pt]{\scriptsize \textcolor{red} 3}(4,4)--node [midway, circle, fill=white, inner sep=1pt,minimum size=1pt ]{\scriptsize 1}(4,6); 
\draw (4,2)--node [ near start, circle, fill=white, inner sep=0.5pt,minimum size=1pt]{\scriptsize \textcolor{red} 2}(6,4)--node [ near start,  circle, fill=white, inner sep=0.5pt,minimum size=1pt]{\scriptsize 1}(4,6);
\draw (4,0)--node [midway, circle, fill=white, inner sep=0.5pt,minimum size=1pt ]{\scriptsize 3}(7, 2)--node [near start, circle, fill=white, inner sep=0.5pt,minimum size=1pt ]{\scriptsize 1}(4,4)--node [near start, circle, fill=white, inner sep=0.5pt,minimum size=1pt ]{\scriptsize \textcolor{red} 2}(7, 6)--(4,8); 
\draw (7, 2)--node [midway, circle, fill=white, inner sep=0.5pt,minimum size=1pt]{\scriptsize \textcolor{red} 3}(6, 4)--node [ near start, circle, fill=white, inner sep=0.5pt,minimum size=1pt ]{\scriptsize  2}(7,6); 
\draw (7, 2)--node [midway, circle, fill=white, inner sep=0.5pt,minimum size=1pt ]{\scriptsize \textcolor{red} 2}(8,4)--node [midway, circle, fill=white, inner sep=0.5pt,minimum size=1pt ]{\scriptsize 1}(7,6); 

\draw (7,6)--node [midway, circle, fill=white,inner sep=0.5pt,minimum size=1pt] {\scriptsize 1} (4,8);
\draw (4,6)--node [midway, circle, fill=white,inner sep=1pt,minimum size=1pt] {\scriptsize 1}(4,8);
\draw (1,6)--node [midway, circle, fill=white,inner sep=0.5pt,minimum size=1pt] {\scriptsize 1}(4,8);

%\draw [fill, color=red] (4,2) circle [radius=0.1]; 
\end{tikzpicture}
\hskip 0.5in
\begin{tikzpicture}[scale=0.8]

\draw (4,0)--node [midway, circle, fill=white, inner sep=0.5pt,minimum size=1pt ]{\scriptsize 1}(1,2)--node [midway, circle, fill=white, inner sep=0.5pt,minimum size=1pt]{\scriptsize 1}(0,4)--node [midway, circle, fill=white, inner sep=0.5pt,minimum size=1pt]{\scriptsize 1}(1,6)--(4,8); 
\draw (1,2)--node [midway, circle, fill=white, inner sep=0.5pt,minimum size=1pt ]{\scriptsize 2}(2,4)--node [near start, circle, fill=white, inner sep=0.5pt,minimum size=1pt ]{\scriptsize 1}(1,6); 
\draw (1,2)--node [near start, circle, fill=white, inner sep=0.5pt,minimum size=1pt]{\scriptsize 3}(4,4)--node [near start, circle, fill=white, inner sep=0.5pt,minimum size=1pt ]{\scriptsize \textcolor{red} 2}(1,6);

\draw (4,0)--node [midway, circle, fill=white, inner sep=0.5pt,minimum size=1pt ]{\scriptsize 2}(4,2)--node [near start, circle, fill=white, inner sep=0.5pt,minimum size=1pt ]{\scriptsize 1}(2,4)--node [near start, circle, fill=white, inner sep=0.5pt,minimum size=1pt]{\scriptsize 2}(4,6)--(4,8); 
\draw (4,2)--node [midway, circle, fill=white, inner sep=0.5pt,minimum size=1pt]{\scriptsize \textcolor{red} 2}(4,4)--node [midway, circle, fill=white, inner sep=1pt,minimum size=1pt ]{\scriptsize 1}(4,6); 
\draw (4,2)--node [near start, circle, fill=white, inner sep=0.5pt,minimum size=1pt ]{\scriptsize \textcolor{red} 3}(6,4)--node [near start, circle, fill=white, inner sep=0.5pt,minimum size=1pt]{\scriptsize 1}(4,6);
\draw (4,0)--node [midway, circle, fill=white, inner sep=0.5pt,minimum size=1pt]{\scriptsize 3}(7, 2)--node [near start, circle, fill=white, inner sep=0.5pt,minimum size=1pt ]{\scriptsize 1}(4,4)--node [near start, circle, fill=white, inner sep=0.5pt,minimum size=1pt ]{\scriptsize \textcolor{red} 3}(7, 6)--(4,8); 
\draw (7, 2)--node [midway, circle, fill=white, inner sep=0.5pt,minimum size=1pt ]{\scriptsize \textcolor{red}  2}(6, 4)--node [near start, circle, fill=white, inner sep=0.5pt,minimum size=1pt ]{\scriptsize 2}(7,6); 
\draw (7, 2)--node [midway, circle, fill=white, inner sep=0.5pt,minimum size=1pt ]{\scriptsize \textcolor{red}  3}(8,4)--node [midway, circle, fill=white, inner sep=0.5pt,minimum size=1pt ]{\scriptsize 1}(7,6); 

\draw (7,6)--node [midway, circle, fill=white, inner sep=0.5pt,minimum size=1pt] {\scriptsize 1}(4,8);
\draw (4,6)--node [midway, circle, fill=white, inner sep=1pt,minimum size=1pt] {\scriptsize 1}(4,8);
\draw (1,6)--node [midway, circle, fill=white, inner sep=0.5pt,minimum size=1pt] {\scriptsize 1}(4,8);

\draw [fill] (4,0) circle [radius=0.05]; 
\draw [fill] (1,2) circle [radius=0.05]; 
\node [left] at (1, 2) {\small \textbf{}}; 
\draw [fill] (7,2) circle [radius=0.05]; 
%\node[star, fill=green, draw=black , inner sep=0.05cm,minimum size=0.05cm ] at (4,4){}; 
\draw [fill] (4,2) circle [radius=0.05]; 
\node [below right] at (4, 2) {\footnotesize \textbf{a}}; 
\draw [fill, ] (0,4) circle [radius=0.05]; 
\draw [fill] (2,4) circle [radius=0.05]; 
\draw [fill] (4,4) circle [radius=0.05]; 
\draw [fill] (6,4) circle [radius=0.05]; 
\draw [fill] (8,4) circle [radius=0.05]; 
\draw [fill] (1,6) circle [radius=0.05]; 
%\node [diamond, fill=blue, draw=black, inner sep=0.05cm, minimum size=0.05cm] at (1,6) {};
\draw [fill] (4,6) circle [radius=0.05]; 
%\node [diamond, fill=blue, draw=black, inner sep=0.05cm, minimum size=0.05cm] at (4,6) {};
%\node [left] at (4, 6){$F$}; 
\draw [fill] (7,6) circle [radius=0.05];
%\node [right] at (7, 6){$G$}; 
\draw [fill] (4,8) circle [radius=0.05];
%\node [above right] at (4,8){\footnotesize $\hat{1}$};
%\draw [fill, color=red] (4,2) circle [radius=0.1]; 

\end{tikzpicture}
\caption{Poset with GRAO that is not RAO (left) and with  RAO (right), with edge labels indicating the atom ordering.} 
\label{graotorao}
\end{figure}
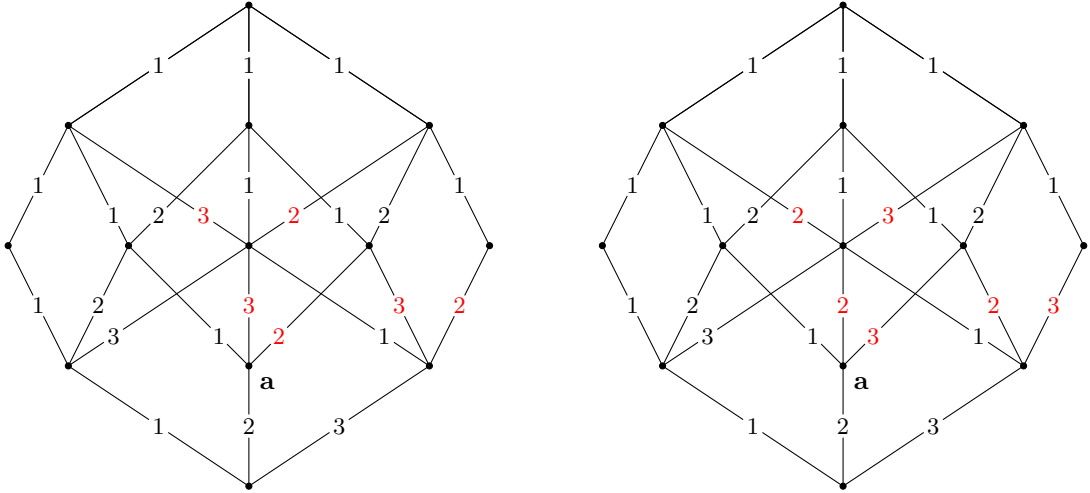

%\commentph{Is there any place to point out here that we are not assuming our posets are graded?}  

%\commentph{Less sure whether to get rid of this paragraph or just make it gentler.}  
%\commentph{We should split the next sentence into two sentences.  Something like: 
In Section \ref{bg-section}, we review background, including %reviewing 
some key ideas of Bj\"orner and Wachs that we will build upon later in the paper.  
%We establish some fundamental properties... in Sections \ref{graosection} and \ref{GRAOiffRAO}, respectively.
%After providing necessary background for our results in Section \ref{bg-section}, w
We establish  some fundamental  properties of GRAOs and of the atom 
reordering process  
in Sections \ref{graosection} and \ref{GRAOiffRAO}, respectively. 
%%%%%%%%%%THE BELOW WAS MOVED FROM BEFORE THE SENTENCE ABOVE%%%%%%%%%%%%%%%%
%\commentplh{Delete this paragraph?  I'm less sure it's a good idea to get rid of this paragraph, though at least we do already mention the reordering process in the second paragraph of the introduction, which seems important to mention in the introduction.} 
%\commentph{Something is wrong with the next sentence -- it does not make sense as now written.}  \gs{
In particular, Section \ref{GRAOiffRAO} presents   %\commentph{includes?} 
one of the main results of the paper, Theorem \ref{cccl};  % \commentph{
this result %} %  which \sout{we prove}
 %\gs{
 establishes that a finite bounded 
poset admits a recursive atom ordering (RAO) if and only if it admits a generalized recursive atom ordering (GRAO).
%While every RAO is a GRAO, as is proven in Lemma \ref{raothengrao}, the converse is not true, as  Figure 2 demonstrates.  
%\commentph{The next sentence needs some transition language to relate it better to the prior sentence.  Something like ``
Figure 2 gives an example highlighting the subtlety of Theorem \ref{cccl}.
% \commentph{should we add red to highlight where things change?} 
 It shows a small example of a poset
%...''}  Figure 2 % \ref{graotorao} 
%shows   a small example of a poset 
$P$ endowed  with a GRAO that is not an RAO (on the left in the figure) and  the same poset endowed with an RAO (on the right in the figure).  
%\commentgs{Probably not the best place for the below remark now that the intro is much shorter. Still questionable whether we want to include the remark at all. Maybe put it where we introduce RAO/GRAO} 
%\gs{\begin{rmk}\label{atomorderlabel}
%In all figures showing RAOs or GRAOs in this paper, we follow the convention that if the cover relation $u \lessdot v$ is labeled by $k$, this means that $v$ is the $k$th atom of $u$ in the RAO or GRAO.  
%\end{rmk}}
% gives an example showing this.  
A key ingredient to  the proof of Theorem \ref{cccl} is an algorithm %we provide  %\commentplh{I changed Definition to Algorithm, but maybe move up to here?} 
%an algorithm 
that  transforms  any GRAO into an RAO on the same  poset $P$,  a procedure we call the ``atom reordering process''  (see Algorithm \ref{reorder-def}).
% \commentplh{Delete remainder of this paragraph?  Or move with Figure 2?}  
 %This   proceeds from bottom to top in $P$, reordering the atoms of each rooted interval $[u,\hat{1}]_r$ in turn, doing so in a way that moves the elements of $F_r(u)$ ahead of the elements of $G_r(u)$ for each $u$ and each $r$ while preserving the  relative order  of the elements within  $F_r(u)$ and within $G_r(u)$.    
  This procedure transforms the GRAO shown on the left in Figure 2 
 into the RAO shown  on the right in Figure 2. 
 
%\commentph{Delete remainder of this paragraph?} 
%before proving Theorem \ref{cccl}.  In Lemma \ref{GRAOrestricts}, we show  that the restriction of a GRAO  to any 
%rooted interval is itself a GRAO.  In Theorem \ref{switch-thm}, we prove that any  local move on a GRAO switching the order of two consecutive atoms in a rooted interval
%$[u,\hat{1}]_r$ where neither is the first atom, subject to closely  related further  conditions, 
%yields a new chain-atom ordering that is itself a GRAO.   
%We prove in Lemma \ref{alwaysfirst}  that whichever atom comes first in any given rooted interval of a GRAO  still comes first at the conclusion of the atom reordering process.  
%We use Lemma \ref{alwaysfirst} to  deduce  in  Lemma \ref{equiv-of-chain-atom-orderings}  that the  atom reordering process may  be carried out by  a series of  local moves of the type arising in  Theorem 
%\ref{switch-thm} above, implying that applying  the atom reordering process  to a GRAO  yields   a GRAO.   Equipped with these results, 
% Theorem \ref{cccl} shows that applying the reordering process to a  GRAO yields not only a GRAO but an RAO.
%
%
%\commentph{Keep or omit?}  
%\commentplh{The edits to this sentence make it ambiguous and I'm afraid will cause many readers to miss interpret.}  
Sections \ref{UE-section} and \ref{CLiffCCsection} establish a link that is not necessarily an equivalence  % \commentph{(but not an equivalence)} 
 between finite bounded posets that admit a GRAO and those that are CC-shellable. 
This link  %\sout{Our work 
%relating  Kozlov's CC-shellability to the property of admitting a GRAO also} 
 led us  to introduce  
a variation on the notion of CC-shellability
in Section \ref{topol-CL-section}, namely TCL-shellability  (see Definition \ref{topol-CL}).  
%It is easy to see that every CL-labeling is a topological CL-labeling, though it is not true that every 
%CL-labeling is a CC-labeling.  We also  prove that CC-shellability implies topological CL-shellability.  
%This mild variation on  Kozlov's notion of  
%CC-shellability (and similarly on  Kozlov's  closely related  edge-labeling condition called  EC-shellability)  seems like it could be   a worthwhile  gentle 
%update to his  theory. 
%
%\commentplh{There is some new text in the next paragraph, trying to make things flow and emphasize key points, though this could be a bit repetitive right now with things said early in this section.}

%\commentph{Very short paragraph?}
%\commentplh{This next sentence  is now a run-on sentence.  I agree the sentences did not read very well before, but I have not come up with a good way yet to rewrite them.  It may be  a necessary evil to have  to repeat the name of the theorem after splitting into two sentences.}
Many of the results proven throughout the paper are  tied together  %\sout{in one place} 
 in  Theorem \ref{TFAE-theorem},  %\sout{.  %, appears just below.  
%
%Theorem \ref{TFAE-theorem}} %is  % collects together  %in one place
 %many of the implications proven throughout the   paper in a way that
a result which   shows   that  several different notions of lexicographic shellability are 
%\sout{all} 
equivalent to each other.  % \commentph{
In this result, we specify %/list out}
%More precisely, we \sout{discuss} \commentph{enumerate}  
 eight different versions of lexicographic shellability and 
show % \sout{ in Theorem \ref{TFAE-theorem} } 
that a finite bounded poset admits any  one of these types 
of lexicographic shelling if and only if it admits all of the others.  
%As a word of caution, it still does happen that one type of lexicographic shelling may be strictly easier to construct than another.   
%(along  with two of the main  results from \cite{bw}).  
%\commentph{Delete statement of  the next theorem from introduction, since we need to state it where we prove it to make the proof readable.}
%The statement of Theorem \ref{TFAE-theorem} is as follows: 
%
%
%\begin{thm} %\label{TFAE-theorem}
% Let $P$ be a finite, bounded poset. Then the following are equivalent:
%\begin{enumerate} % [label=\alph*)]
%\item $P$ admits a recursive atom ordering
%\item $P$ admits a generalized recursive atom ordering
%\item $P$ admits a CL-labeling
%\item $P$ admits a CL-labeling with the UE property
%\item $P$ admits a self-consistent CC-labeling. 
%\item $P$ admits a CC-labeling with the UE property
%\item $P$ admits a self-consistent topological CL-labeling
%\item $P$ admits a topological CL-labeling with the UE property
%\end{enumerate}
%Moreover, all of these implications are proven constructively.  That is, for each implication  either it is shown  how to 
%construct the latter type of object from the former  or else  the   former type of object is proven also to be the latter type of object.
%\end{thm} 

In Section \ref{CC-UE},  we apply our results % regarding techniques for proving lexicographic shellability 
 to
deduce  that a class of posets previously % \gs{
shown to be %} \sout{proven} 
  CC-shellable in \cite{elect}  is  in fact CL-shellable.  That is, we prove that the dual posets to the uncrossing orders (conjectured to be  lexicographically shellable by  Lam in \cite{lam}) are  CL-shellable.   These uncrossing orders arise naturally as face posets of  stratified spaces of planar electrical networks (see e.g. \cite{K-elect}, \cite{lam}, and  references therein).
  % \commentplh{We could fix the next sentence by changing CW posets to so-called CW posets -- I think this sentence is important in that it helps motivate our application, and it's very hard to get a paper accepted without it including an impressive application.   Without this sentence, I worry editors may think ``who cares that another class of posets is shellable''.   I also don't think this is what the referee was calling too technical.} 
   %\commentgs{Too technical:} \commentph{
   The fact that they are shellable posets combines with Lam's result from \cite{lam} that they are Eulerian posets to imply that they are  CW posets, i.e. face posets of 
regular CW complexes with finitely many cells.  Thus, the shellability of uncrossing orders  provides 
an important step in understanding the topological structure of these spaces of planar electrical networks. %}
%,  by virtue of being Eulerian (proven in \cite{lam})  and shellable.  }
%\commentph{Delete remainder of this paragraph?} 
% To deduce dual CL-shellability  of uncrossing orders, we check  that the EC-labeling (a special type of CC-labeling) from \cite{elect} for the dual uncrossing orders has the UE property.  Equipped with this, we may use our upcoming  results  (1) that  the UE property implies self-consistency, (2) that  any finite bounded poset with a self-consistent CC-labeling has a GRAO, and (3) that any such  GRAO may be transformed into an RAO.  Thus, we  show that the dual uncrossing posets have an RAO.  But  \cite{bw} proved that having an  
%RAO implies  CL-shellability, giving the final  step  in  the proof  
%that dual uncrossing orders are CL-shellable.  

The paper concludes with  further results,  observations,  and  open questions  in Section \ref{Consequences-section}. 

\section{Background}\label{bg-section} 

%For this section, we l
Let $P$ denote a partially ordered set  (poset).  All posets throughout this paper are assumed to be  finite and bounded  but are not assumed to be  graded.   % posets.
%, whether or not that is indicated in each place where posets arise.   
%They are also assumed to be bounded, namely to have a unique minimal element and a unique maximal element.  We do not assume posets are graded.   
%
For background on posets,  poset topology, and shellability  %associated simplicial complexes called order complexes 
beyond what appears below, we refer readers to \cite{St12},  \cite{Wachs}, \cite{Bj-Top-Meth}, \cite{Stanley-green},  and  \cite{Zi}.

A \textbf{cover relation} $u\lessdot v$ in a poset $P$ is an order relation $u<v$ with the further requirement that there does not exist any $z\in P$ with $u<z<v$.  In this case, we say that $v$ \textbf{covers} $u$. % or that $u$ is \textbf{covered by} $v$.
%
%
%\begin{defn} 

A poset $P$ is \textbf{bounded} if it has both a unique least element (often denoted $\hat{0}$) and a unique greatest element (often denoted $\hat{1}$).
%, namely elements $\hat{0}$ and $\hat{1}$ such that $\hat{0}\le u \le \hat{1}$ for all $u\in P$.
%\end{defn} 
A \textbf{closed interval}, denoted $[u,v]$, in a poset $P$ is the subposet consisting of all elements $z\in P$ such that $u\le z \le v$.  
%Closed intervals are always bounded with minimal element $u$ and maximal element $v$.  

The \textbf{atoms}  of a bounded poset $P$  (resp. a  closed interval $[u,v]$)  
are those elements $a$  in $P$  (resp. $[u,v]$)  
that cover  $\hat{0}$ (resp. $u$).  
Likewise the \textbf{coatoms} of a bounded poset $P$  (resp. a closed interval $[u,v]$)  
are the elements in $P$  (resp. $[u,v]$)  
covered by  $\hat{1}$ (resp. $v$).  

%\begin{defn} A poset is \textbf{graded} if it is bounded and pure. 
%\end{defn}

%\begin{defn}
A \textbf{chain} in a poset $P$  is a totally ordered subset $u_1 < \cdots < u_r$  of $P$. The \textbf{length} of a chain  $u_1 < \cdots < u_r $ is the number  $r-1$ of order relations in the chain.   The {\bf length} of a poset $P$, % (for instance, in the next definition), 
is the length of the longest chain in $P$.  
 A chain is \textbf{maximal} in $P$  if no additional elements of $P$  may be inserted in it.  A chain $u_1 < u_2 < \cdots < u_k$ is \textbf{saturated} in $P$  if it is a maximal  chain of $[u_1,u_k]$. We will also make the convention of   sometimes referring to the maximal chains of a closed interval $[u_1,u_k]$ as the saturated chains of $[u_1,u_k]$. %is one less than the number of elements in the chain. 
%\end{defn}
%
%\begin{defn}
If all maximal chains in a poset  $P$ are of the same length, the poset is said to be \textbf{graded}. % \textbf{pure}. 
%\end{defn}

\begin{defn} 
The \textbf{order complex} of  a poset $P$, denoted $\Delta (P)$,  is the abstract simplicial complex whose  $k-$faces are  the chains of length $k$ of $P$. In particular, this means the vertices are the chains consisting of single elements of $P$.  Note that a face $\sigma$ in $\Delta (P)$ is contained in another face $\tau$ in $\Delta (P)$  if and only if  the chain corresponding to $\sigma$ is contained in the chain corresponding to $\tau$.
\label{ordercomplex}
\end{defn}

When we say that a poset $P$ has  a topological property (such as shellability or homotopy equivalence to a wedge of spheres), we mean that  %a simplicial complex derived from the poset, called
$\Delta(P)$ %, % (see Definition \ref{ordercomplex}), 
has this topological property.   %The order complex $\Delta (P)$ of a poset $P$  and in particular t

The \textbf{dual} of a poset $P$, denoted $P^*$,  has  the same elements as $P$ with $u \le v$ in $P^*$ if and only if $v\le u$ in $P$.  %We will soon utilize 
%\commentph{Removed text: We will soon use the fact that  $\Delta (P) = \Delta (P^*)$  % as abstract  simplicial complexes  
 %%%will soon be useful to us in that it 
%%%.  % due to having the same faces as each other.  W
%%%Specifically, we will  use the fact
%and specifically its consequence  that any  ``lexicographic shelling'' for $P^*$   will give  a shelling for $\Delta (P)$.}
%%%make use of the consequent fact % ce of this, namely 
%%% that $P$ is shellable if and only if $P^*$ is shellable.  
%\commentplh{New text: 
%\commentph{Should the next bit be a remark so that we can reference it later?  For instance in the section on dual uncrossing orders.}
\begin{rmk}
Virtually everything in this paper has a dual version for the simple reason that a poset and its dual have the same order complex.  This fact  enables  any  poset theoretic technique  to be applied to the dual poset to derive the same consequence regarding the order complex.  There are indeed posets where this is a helpful thing to do (e.g. the uncrossing orders as discussed in Section ~\ref{CC-UE}).  We leave it to the interested reader to fill in the dual versions of our results.
%}
\end{rmk}

\begin{defn}
Given %any regular CW complex $K$ such as for example 
any simplicial complex $K$, its \textbf{face poset} $P(K)$ consists of  the faces % cells 
of $K$ %(called faces in the case of a simplicial complex)  
with order relation 
$\sigma \le \tau $ if and only if the set of vertices in $\sigma $ is a subset of the set of vertices in $\tau $.  The \textbf{closure} of a face $\tau $, denoted $\overline{\tau }$,  is the set of faces $\sigma $ such that $\sigma \le \tau $ in $P(K)$.  
%$\sigma $ is contained in the closure of $\tau $; additionally,  a unique minimal element $\hat{0}$ is adjoined to to $P(K)$ to  represent the empty face.  In the  special 
%case of simplicial complexes,  the order relation  on faces % my be equivalently described as
%is exactly  the order  by containment of the sets of vertices in the faces.  
%ordered by inclusion with a $\hat{0}$ adjoined.

 The \textbf{augmented face poset} $\hat{P}(K)$ is $P(K)$ with a maximal element  $\hat{1}$ adjoined if $P(K)$ does not already have a unique maximal element, and $\hat{P}(K) = P(K)$ otherwise. 
\end{defn}

\begin{defn}\label{shelling-def}
 A simplicial complex is \textbf{shellable} if there is a total order $F_1,\dots ,F_k$ on its maximal faces (known as \textbf{facets}) such that 
$\overline{F_j} \cap (\cup_{i<j} \overline{F_i})$ is a pure, codimension one subcomplex of $\overline{F_j}$ for each $j\ge 2$.  Such a facet ordering is known as a \textbf{shelling}.
\end{defn}

% If \todo[color=red!35]{more here?} $K$ is a regular CW complex, then $\Delta(P(K)-\hat{0}) \cong K$. Thus the incidence relations of cells as given in $P(K)$ determines the topology of $K$. This is not true of all CW complexes. See \cite{bjorner} for more details.
\begin{defn} A poset $P$ is said to be \textbf{shellable} if its order complex $\Delta(P)$ is shellable. 
\end{defn}

%\commentph{What you propose is confusing in two ways.  First:  ``These'' in the second sentence could refer to the labelings, and second, readers may 
%wonder in the first sentence what
%lexicographic shellings are.  Here is another attempt that hopefully achieves your goal of breaking up the sentence while resolving my above concerns:}
 Next we review various types of edge and chain-edge labelings of a finite poset $P$.  These labelings will induce lexicographic shellings for $\Delta (P)$, namely shellings 
obtained by taking the facets of $\Delta (P)$ in order according to the lexicographic (i.e. dictionary) order of the label sequences of the corresponding maximal
chains of $P$, breaking ties in any manner.
%\sout{Next we review various types of edge and chain-edge labelings of a finite  poset $P$  % relevant to this paper 
% that will induce lexicographic shellings for $\Delta (P)$.  These % \commentph{Lexicographic shellings} \sout{These}  
% are shellings obtained by taking the facets of $\Delta(P)$ in order according to the lexicographic (i.e. dictionary) order of the label sequences of the corresponding maximal chains of $P$. }
 % \sout{i.e. shellings obtained by ordering the facets of $P$ by ordering  label sequences on the corresponding maximal chains of $P$
 %lexicographically (namely in dictionary order).}
  Let $E(P)$ be the set of edges in the Hasse diagram of  a finite  poset $P$, that is, the pairs $x,y\in P$ such that $x\lessdot y$.  An {\bf edge labeling}  of $P$ is
a map $\lambda : E(P) \rightarrow Q$ for $Q$ a poset. Quite often,  $Q$ is the integers with their usual order.   %Given an edge-labeling of $P$, w

%\commentph{here we define ascent as a strict ascent.  we should check that we do not talk about weakly ascending chains later and/or use $\le $ later for ascent.}

We say that $x\lessdot y\lessdot z$ is an \textbf{ascent}  with respect to the edge labeling  $\lambda  $ if
$\lambda (x,y)  <_Q \lambda (y,z)$.  Any $x\lessdot y\lessdot z$ that is not an ascent with respect to $\lambda $ is said to be a \textbf{descent}.
 A maximal chain in a finite poset $P$ (or more generally in a closed interval $[u,v]$ in $P$)  
is an {\bf ascending chain} if it consists entirely of ascents.

\begin{defn}[\cite{bjornercm}, \cite{non-pure2}]
An edge labeling of a finite, bounded poset $P$  is called an \textbf{EL-labeling} (for edge lexicographical labeling) if for every interval $[x, y] \in P$ the following conditions are both met:
\begin{enumerate}[label=(\roman*)]
\item There is a unique ascending maximal chain $c$ in $[x, y]$.
\item The label sequence associated to $c$ lexicographically precedes the label sequences associated to every  other maximal chain in $[x, y]$. 
\end{enumerate}  
A finite, bounded poset admitting  such a labeling  is  \textbf{EL-shellable}. 
\end{defn}

%\commentgs{Could also turn the following into a definition for "lexicographic order of the maximal chains of $P$"} 
%\gs{
\begin{rmk}
In this paper, when we say that we take the lexicographic order on maximal chains of $P$, we mean that we take the maximal chains in order according to the lexicographic order on their label sequences.
\end{rmk}

%In defining EL-labelings, 
Bj\"orner  first introduced EL-labelings and EL-shellability in the graded case in \cite{bjornercm}. This was generalized by Bj\"orner and Wachs to the not necessarily 
graded case  when they introduced the notion on nonpure shellability in \cite{non-pure1} and \cite{non-pure2}.
%originally  included as a hypothesis  that the posets should  be graded; this requirement was removed later  in work of  Bj\"orner and Wachs when they introduced the more general  notion of nonpure shellability in \cite{non-pure2}.    %The usage of the  term EL-shellable above for a poset admitting an EL-labeling  is justified by t
The following  fundamental result  from \cite{bjornercm} %(which was generalized to the not necessarily graded 
%case in \cite{non-pure1}) 
explains the usage of the term EL-shellablility. 

%\commentph{I just added ``finite'' and  ``bounded'' as hypotheses below  to match what Bj\"orner wrote -- if you check Bj\"orner's paper, he defines "graded" to mean  finite, bounded and with all maximal chains having the same length as each other, which is not how we define graded, so this is an effort to reconcile what we are saying.}

% is a fundamental result in topological combinatorics. % and poset topology. 
\begin{thm} [%{Theorem~2.3}, 
\cite{bjornercm}, Theorem~2.3]
If $P$ is a finite, bounded, graded 
poset with an EL-labeling, then the lexicographic order of the maximal chains of $P$ is a shelling order for the corresponding facets of $\Delta(P)$. 
\end{thm}

Now we turn to % the notion of CL-shellability, 
%a \commentplh{New: 
a generalization %}  %relaxation 
of EL-shellability due to Bj\"orner and Wachs (see \cite{bw})  
in which  edge labelings are replaced by more general chain-edge labelings (defined next).
In this context, we replace $E(P)$ by the set $E^*(P)$ defined as follows:  
%For a graded poset $P$ of length $n$,let  $E^*(P)$ be the set of edges of maximal chains in the Hasse diagram of $P$ i.e. 
$$E^*(P)=\{(c, x, y) : \mbox{$c$ is a maximal chain}; x, y \in c; x \lessdot y\} .$$   
% Given a labeling of $E^*(P)$ with elements of a poset $Q$, we say that a maximal chain in  a closed interval $[u,v]$ within $P$ is \textbf{ascending}  with respect to the edge-labeling if the sequence $(\sigma_1,\sigma_2,\dots ,\sigma_r)$ of labels has $\sigma_i  <_Q \sigma_{i+1}$ for $i=1,2,\dots ,r-1$.

%\commentgs{The wording for this definition is extremely similar to the definition given by Bjorner and Wachs in ``On Lexicographically Shellable Posets" Is this a concern?} 
\begin{defn} Let $P$  be a finite bounded poset and let $Q$ be any  poset.  A  \textbf{chain-edge labeling} (or \textbf{CE-labeling}) of $P$ is a map $\lambda: E^*(P) \rightarrow Q$  that satisfies the following condition: If
two maximal chains coincide along their first $d$ edges, then they have the same labels as each other on these edges.  
%their labels also coincide along these edges.
 In other words, if $c$ is a maximal  chain $\hat{0}=x_0 \lessdot x_1 \lessdot \ldots \lessdot x_n=\hat{1}$ and $c'$ is a maximal 
 chain $\hat{0}=x_0' \lessdot x_1' \lessdot \ldots \lessdot x_n'=\hat{1}$ where $x_i=x_i'$ for $i=0,1,2,\dots ,d$, then  $\lambda (c, x_{i-1}, x_i)=\lambda (c', x_{i-1}', x_i')$  for $i=1,2,\dots ,d$. % if we also have  $x_i=x_i'$ for $i=0, 1, \ldots , d$.
\end{defn} 

 %\commentph{Add here (or after CL-labeling below) something along the lines of the following: 
 To see  a naturally arising  example of a chain-edge labeling that is not an edge labeling, we refer readers to the chain-edge labeling due to Bj\"orner and Wachs for  the dual poset to Bruhat order.  This labeling % which is actually a 
% CL-labeling   (a notion defined shortly) 
 is reviewed  just prior to Proposition \ref{Bruhat-UE}.  Bj\"orner and Wachs proved in \cite{bw} that it  is a CL-labeling, a notion defined shortly. 
 
%\commentph{Note that we use the word relaxation to say that CL-shellability is a relaxation of EL-shellability.  Perhaps we should say ``generalization''.}
%\commentgs{Agreed, especially if we plan to change GRAOs to relaxed recursive atom orderings}
\begin{defn}
If $[x, y]$ is an interval in $P$ and $r$ is a saturated chain from $\hat{0}$ to $x$, then the pair $([x, y], r)$ is called a \textbf{rooted interval} with  $r$ as the {\bf root} of this rooted interval.  %Such a rooted interval
This   is denoted  by $[x, y]_r$.   
\end{defn}

%??? Tricia changed the phrasing next because this `association' is not bijective --- many saturated chains in general may have the same label sequence.  I think there is probably a better word such as ``assigned to it'' rather than ``associated''.  Also, we should get rid of $\sigma $ and just use $\lambda $.  One should always avoid using more symbols than necessary

Given a chain-edge labeling $\lambda$ of a finite bounded poset  $P$,  $\lambda $ associates to each maximal chain of $P$  a label sequence  as follows. If $m=(\hat{0}=x_0 \lessdot x_1 \lessdot \ldots \lessdot x_n=\hat{1})$, then the associated label sequence is 
$$\lambda(m):= (\lambda(m, x_0, x_1), \lambda(m, x_1, x_2), \ldots \lambda(m, x_{n-1}, x_n)).$$ By definition of a chain-edge labeling, any two  maximal chains both containing the same root $r$ from $\hat{0}$ to $x$ and the same saturated chain $c$ in the interval $[x, y]$ will both have the same label sequence assigned to $c$. We will denote this label sequence by $\lambda_r(c)$ and  the individual labels  comprising it as $\lambda_r(x_i,x_{i+1})$ for $i=0,\dots ,n-1$.  
%\commentph{Is this next what we actually do much in this paper?} \commentgs{``Sometimes we will write..."?} Typically  
Sometimes we will  write $\lambda (c)$ for the label sequence and  $\lambda (x_i,x_{i+1})$ for the label on the edge $x_i \lessdot x_{i+1}$ %$i=0,1,\dots ,n-1$   
if the choice of root $r$ is clear from context.   

%\commentph{here again we use strict ascents}

Given a label sequence $(\lambda_1,\dots ,\lambda_r)$, we say that a pair of consecutive labels $\lambda_i,\lambda_{i+1}$ comprises an {\bf ascent} if and only if $\lambda_i < \lambda_{i+1}$.  The pair $\lambda_i, \lambda_{i+1}$  comprises a {\bf descent} otherwise.  

\begin{defn}
A maximal chain $c$ in a rooted interval $[x, y]_r$ is \textbf{ascending} with respect to a %n edge-labeling or 
chain-edge labeling $\lambda  $  if the label sequence $\lambda_r(c) = (\lambda_1,\lambda_2,\dots ,\lambda_r)$  has $\lambda_i < \lambda_{i+1}$ for $i=1,2,\dots ,r-1$.   
\end{defn}

\begin{defn}[\cite{bw}] A CE-labeling $\lambda$ of a finite bounded poset  
$P$ is called a \textbf{CL-labeling} (for chain-lexicographical labeling) if for every rooted interval $[x, y]_r$ in $P$, 
\begin{enumerate}[label=(\roman*)]
\item there is a unique ascending chain $c$ in $[x, y]_r$ and
\item the label sequence $\lambda_r(c)$ lexicographically precedes the label sequence for  every other maximal chain in $[x, y]_r$.
\end{enumerate}
If a % graded
finite bounded  poset $P$ admits a CL-labeling, then $P$ is said to be \textbf{CL-shellable}. 
\end{defn}

%\commentph{Be careful to give the right reference next!}
Bj\"orner and Wachs proved in \cite{BW82} (resp. \cite{non-pure1}) that  whenever a finite bounded  poset $P$ that is graded (resp. is not necessarily graded) admits a CL-labeling, then any linear extension of the lexicographic order on its maximal chains  given by  the CL-labeling is a  shelling order on the corresponding  facets of  $\Delta(P)$. %corresponding to these maximal chains.
%; recall that a linear extension of a partial order is any total order that is consistent with that partial order.
 
% \commentplh{Removed dual statement here.}
% Since $\Delta(P)=\Delta(P^*)$,  
%a CL-labeling for $P^*$ also gives  a shelling for $\Delta(P)$. We call a  CL-labeling of $P^*$ a \textbf{dual CL-labeling} of $P$, and  we call a poset $P$ such that $P^*$ admits a CL-labeling  \textbf{dual CL-shellable}. 

One of the primary  techniques for proving that a finite bounded poset is CL-shellable is to  construct  a recursive atom ordering (see Definition \ref{rao}). The  notion of recursive atom ordering  (RAO)  was introduced by Bj{\"o}rner and Wachs in \cite{bw}.  They extended it to posets that are not necessarily graded 
in \cite{non-pure2}.   %Recursive atom orderings reflect the recursive nature of a shelling. %  (a viewpoint  that is utilized e.g.  in  Section \ref{d-complexes}).   
%\commentph{Tricia made a first pass at moving the notion of chain-atom ordering here, in the next few lines.   It's possible this should go even earlier, near the definition of CL-labeling, but I have my doubts about that.}
Before defining recursive atom ordering, we lay the groundwork with a notion we call chain-atom ordering that will encompass all recursive atom orderings and all  generalized recursive atom orderings (defined later) as special cases.  

%\commentph{%I would be naturally inclined to think of 
%$\Omega ([u,\hat{1}]_r)$ as a chain-atom ordering on this rooted interval, not as just an ordering of its atoms.  However, I am using different notation for such a full chain-atom ordering when I need to use that, as explained below.  So that helps.  
%We should make sure we are doing things in a way that does not create confusion, though there may be an important place for both of these notions.}

\begin{defn}\label{chain-atom-def}
A \textbf{chain-atom ordering} $\Omega$ of a finite bounded poset $P$  is a choice of ordering on the atoms of each rooted interval  
$[u,\hat{1}]_r$  of $P$. For the rooted interval $[u, \hat{1}]_r$, we will denote this ordering of atoms as $\Omega([u, \hat{1}]_r)$. 
%\commentgs{I introduced notation here for use in the algorithm}
%   \commentph{Add: 
%\commentph{I felt like I had to change this next notation adding a bar, not just because the referee suggested it but because the notation was too confusing some 
%places without the bar.  There are still more places I need to change this in the paper.}
   On the other hand, we denote the restriction of the full chain-atom ordering 
    $\Omega $ to a rooted interval $[u,\hat{1}]_r$ as $\Omega|_{[u,\hat{1}]_r}$, and more generally we denote 
   the restriction of $\Omega $  to  $[u,v]_r$ as 
$\Omega|_{[u,v]_r}$.  We sometimes call $\Omega|_{[u,v]_r}$ the chain-atom ordering on $[u,v]_r$ induced by $\Omega $.
\end{defn}

Whenever  a finite bounded poset has a recursive atom ordering, defined next, this by definition
guarantees the existence of an especially well-behaved type of chain-atom ordering.  We will often refer to these especially nice   chain-atom orderings themselves as recursive atom orderings.  A key place where we will do this is when we introduce  a relaxation of the notion of recursive atom ordering in Section \ref{graosection}. 
%\gs{We will often refer to these types of chain-atom orderings themselves as recursive atom-orderings.}
%\commentph{I  really don't like the repetition of the phrase well-behaved here.  Maybe you can find a different word, but I think it's clear from context that we are referring to the ones in the prior sentence.} 

\begin{defn} A finite bounded  poset $P$ admits a \textbf{recursive atom ordering} if $P$ has length 1 or if the atoms of $P$ can be ordered $a_1, a_2, \ldots a_t$ such that: 
\begin{enumerate}[label=(\roman*)]
\item 
\begin{enumerate}
\item For all $j=1, \ldots, t$, $[a_j, \hat{1}]$ admits a recursive atom ordering.
\item For $j \neq 1$, the atoms that come first in this recursive atom ordering  for $[a_j,\hat{1}]$ are those that are greater than some atom $a_k$ of $P$ for $k < j$. 
\end{enumerate}
\item For all $i < j$ and $y \in P$ satisfying $ y > a_i $ and $y >  a_j$, there exists  $k < j$ and  $z\in P$ such that $a_j \lessdot z$ and $a_k < z \leq y$. 
\end{enumerate}

\label{rao}
\end{defn}

%???  should this next theorem and the commentary before it say "bounded" or just "finite"?

The following theorem of Bj{\"o}rner and Wachs from \cite{bw}  (which they extended to the not necessarily graded case in \cite{non-pure2}) established  a very useful relationship between the existence  of  a recursive atom ordering and of a CL-labeling for  any finite bounded poset $P$. 

%\commentph{The version stated next is the version of \cite{non-pure2}.  Perhaps we should add the second citation below?}

\begin{thm}[\cite{bw}, Theorem 3.2; \cite{non-pure2}, Theorem 5.11] \label{CL-iff-RAO}
A finite  bounded poset $P$ admits a recursive atom ordering if and only if $P$ is CL-shellable. 
\end{thm}

%\commentph{The referee wants us to give the next definition in Section 2 in the generality of chain-atom orderings rather than recursive atom orderings, then delete the generalization later.  Tricia thinks we should do this, so she is suggesting here language.}

%A key  ingredient in their proof % \commentph{omit the parenthetical remark: (that we will  also use and generalize later)}   %we will need here (and later in the paper) that was introduced in \cite{bw} 
% is the pair of sets $F_r(u)$ and $G_r(u)$ \commentph{old:, defined as follows.} which they defined for recursive atom orderings and we will define now more generally for chain-atom orderings.  

%\commentph{This next definition  seems pretty technical for this spot in the paper.  I am somewhat
%inclined to move it to later in this section, namely to right before we use it (first in the proof of Proposition  2.18 and then in the proof sketch of Theorem 2.19).  In the sentence or two just above, we could point readers to this definition and to the proof sketch as well at the end of this section.  The main argument for doing this definition now is if it helps with CC-shellability or other upcoming notions, but I don't that's really the case.}  

%\commentph{Be careful where to put the next paragraph and the exact phrasing, also watching for repetitiveness.}

%\commentph{Perhaps add: 
We sketch one direction of the proof of Theorem \ref{CL-iff-RAO} shortly,   %namely the implication that every recursive atom ordering gives rise to a CL-labeling,
 %later in this section  in Theorem \ref{one-direction-of-BW}, 
since many of the  ideas in this argument  will be used in other proofs later in the paper.   A key ingredient is the pair of sets $F_r(u)$ and $G_r(u)$ defined next.
% below in 
%Definition \ref{F-u}. 
% Readers may also find it useful to read the proof of Theorem \ref{RAOimpliesCC} as a warm-up for results later in the paper.  
%This proof sketch relies  heavily on the notions of $F_r(u)$ and $G_r(u)$ that we recall in Definition \ref{F-u}.
%
%\commentph{Perhaps move definition ahead of proof sketch?}
%
%\commentph{Be consistent throughout the paper where we put the subscript $r$ e.g. in $F_r(u)$.}
%
%\commentph{Be consistent with superscript notation somewhere, e.g.  $F^{\Lambda }$.}
%
Bj\"orner and Wachs introduced these sets $F_r(u)$ and $G_r(u)$ for recursive atom orderings, but we find it convenient to define them more
generally. %for chain-atom orderings. 

\begin{defn}\label{F-u}
Consider any rooted interval $[u,\hat{1}]_r$ in a finite bounded poset $P$.  Let $u^-$ be the element of $r$ covered by $u$ and
let  $r^-$ be the root for $[u^-,\hat{1}]$ obtained by omitting $u$ from $r$ and otherwise preserving $r$. 
Let $\Lambda $ be  either % either %  any entity which includes within its data 
a total order on the atoms of $[u^-,\hat{1}]_{r^-}$ % or a chain-atom ordering
 or  any richer structure,  such as a chain-atom ordering,  which specifies  such  a total order.
 % on the atoms of
%$[u^-,\hat{1}]_{r^-}$.}
% for instance $\Lambda $ could be an ordering of the atoms of $[u^-,\hat{1}]_{r^-}$  or it could be a chain-atom ordering.} % or any other such collection of atom orderings.}
%Let 
%\commentph{awkward to use $\Lambda $ for two 
%different things.} 
%$\Lambda$  be  \commentph{either a chain-atom ordering or a portion of a chain-atom ordering which includes within its data}  an ordering of the atoms of $[u^-, \hat{1}]_{r^-}$. %\commentph{or a partial chain-atom order or a  chain-atom ordering which includes within it such an ordering of the atoms of 
%$[u^-,\hat{1}]_{r^-}$}. \commentph{Let $\Lambda $ be a specification of an ordering on the atoms of ..., for example $\Lambda $ could be  a chain-atom ordering or a partial chain-atom ordering.} 

Define \textbf{ $F^{\Lambda }_r(u)$} to be the set of atoms $a$ of  $ [u,\hat{1}]_r$ such that 
%\commentph{It would remove ambiguity to replace the next inequality with 
$a>_P a'$
 %$a >a'$
  for some atom $a'$ in $[u^-,\hat{1}]_{r^-}$ that comes earlier than 
%\commentph{I added the phrase ``the atom'' to make this easier for readers to understand.}  
%the atom 
$u$  %\commentph{or skip saying atom and saying where}   of $[u^-,\hat{1}]_{r^-}$ 
in $\Lambda $. % \commentph{old: RAO.}
 Define 
\textbf{$G^{\Lambda }_r(u)$} to be the set of all atoms of $[u,\hat{1}]_r$ that are not contained in $F^{\Lambda }_r(u)$. 

 % \commentgs{ Give F and G an order here, or just in the reordering algorithm? Leaning towards the latter}\commentph{Tricia agrees not to give $F$ and $G$ an order here.  It would make various lemmas she proves later much more difficult to prove in a coherent way -- making the notation extremely challenging.}  
 
% \commentph{Add: More generally, we define $F_r(u,v)$ and $G_r(u,v)$...}
 
% \begin{defn}\label{F-and-G}
%Consider  any chain-atom ordering (see Definition \ref{chain-atom-def}) of a finite poset $P$ and suppose
%Suppose that  $[u,v]_r $ is a rooted interval of $P$, where $r$ is a 
%saturated chain $\hat{0}\lessdot u_1\lessdot \cdots \lessdot u_k \lessdot u$.  
%Let $r^-$ denote the saturated chain $\hat{0}\lessdot u_1\lessdot \cdots \lessdot u_k$.  

%\commentgs{Have we deliberately chosen not to match exactly the language below with the language above? E.g. why u' and not a'?} 
Given any $v\in P$ satisfying $u<v$, define $F^{\Lambda }_r(u,v)$ to be the set of atoms $a$ of $[u, v]_r$ such that $a >_P a'$ for some $a'$ that covers $u^-$ and comes earlier than $u$ in % \commentph{the restriction of $\Lambda $...} the given chain-atom ordering of the atoms of 
 $\Lambda|_{[u^-, v]_{r^-}}$. Define $G^{\Lambda }_r(u,v)$ to be the set of atoms of $[u, v]_r$ that are not contained in $F^{\Lambda }_r(u,v)$.   

  Sometimes we will denote these sets  simply by 
$F_r(u,v)$, $G_r(u,v)$, $F_r(u)$ and $G_r(u)$  when the choice of $\Lambda $ is clear from context.  
\end{defn}

\begin{rmk}
By definition,  % for $F_r(u)$ and $G_r(u)$ (see Definition \ref{F-u})  
we have $F^{\Lambda }_r(u) = F^{\Lambda }_r(u,\hat{1})$ and  %that
 $G^{\Lambda }_r(u) = G^{\Lambda }_r(u,\hat{1})$.
\end{rmk}

{\bf Proof sketch of how an RAO yields a CL-labeling:} 
%\begin{proof}
  Given an RAO $a_1,\dots ,a_n$ for $P$, start by labeling each  cover relation of the form $\hat{0}\lessdot a_i$ with the integer $i$.  
  By induction on the length of the longest saturated chain in $P$, 
 one  may assume that 
 each rooted interval $[a_i,\hat{1}]$  has its own RAO that induces a CL-labeling for $[a_i,\hat{1}]$.   The plan is to modify this CL-labeling for  $[a_i,\hat{1}]$ for each $i$ so that the modified labels may be taken together  with the   labels % $1,2,\dots ,n$ on the cover relations 
 $\lambda ( \hat{0} , a_i) = i$ for $i=1,2,\dots ,n$ to give a CL-labeling $\lambda $  for all of $P$.  
 It will suffice to  describe how to modify the labels on the  cover relations upward from $a_i$ to the atoms of $[a_i,\hat{1}]$, then to apply this same label modification process inductively to handle the rooted intervals higher in the poset.
 
 Letting $x_1,\dots ,x_u$ be our RAO for 
 $[a_i,\hat{1}]$ which is guaranteed to exist by  the definition of an RAO for $P$, Bj\"orner and Wachs made the key observation that there always exists  some $j\in \{ 1,\dots ,u\} $ such that % \commentph{We still need to make our notation for $F$ and $G$ in the rest of this proof sketch consistent with elsewhere in our paper.}
 $x_1,\dots ,x_j\in F_{\hat{0}\lessdot a_i} (a_i)$ while $x_{j+1},\dots ,x_u\in G_{\hat{0}\lessdot a_i} (a_i)$. %  (see Definition \ref{F-u}).    
 For example, one may label  each cover relation $a_i\lessdot x_{i'}$ for 
 $i'\le j$ with the label % \commentph{These next labels  were written in a confusing way, 
% so I corrected this part.  Bj\"orner and Wachs do not actually calculate these labels, so I also changed
 %the language to reflect that.}  
  $ \lambda (a_i,x_{i'}) = \lambda (\hat{0},a_i) - (j-i') - 1$, and one may label  each cover relation $a_i\lessdot x_{i'}$ for $i'\ge j+1$ with the label 
 $\lambda (a_i,x_{i'}) = \lambda (\hat{0},a_i) + (i' - j) $.     %Bj\"orner and Wachs show that t
 The resulting labeling has the following two properties:
% \commentph{I think we should put  $\ge $ in item 1 if we define ascending chains to be strictly ascending.  Also should double check how we do the shifting above, in case it ever gives equality}
 \begin{enumerate}
 \item
$ \lambda (\hat{0},a_i) \ge   \lambda (a_i,x_{i'}) $ if and only if $x_{i'} \in F_{\hat{0}\lessdot a_i}(a_i)$
 \item
 $\lambda (a_i,x_{i'}) < \lambda (a_i,x_{i''})$ if and only if $i'<i''$.
 \end{enumerate}
 In other words, the shifting of label values preserves the relative order of the labels on cover relations upward from a fixed element with the same fixed choice of root below, but at the same time it creates ascents and descents exactly where they  are needed in order to have a CL-labeling.
% \end{proof}
%\\
%{\bf End of proof sketch.}
%
See Theorem 3.2 in  \cite{bw} for further details of this proof.
% \commentph{Should we add something like:  Much of the remainder of the
%argument is similar to our proof of 
%Proposition \ref{RAOimpliesCC} above, since one may use the sets $F_r(u)$ and $G_r(u)$ to show  
%that the ascents (resp. descents)  in the chain-labeling we have just described are exactly the topological  ascents (resp. topological descents)  in this same chain-labeling.}\commentplh{Be careful in this possible
%addition whether to use $F_r(u,v)$ or $F_r(u,\hat{1})$ and likewise for $G$.}   

\noindent {\bf End of proof sketch.}

% \commentph{Remember to change to $F^{\Lambda }(u,v)_r$ and $G^{\Lambda }(u,v)_r$ assuming we do that...}
%\end{defn}

%\commentplh{Removed dual statement here.}
%Bj\"orner and Wachs also  introduced the following dual notion to recursive atom ordering: % known as recursive coatom ordering.

%\begin{defn} A finite bounded  poset $P$ admits a \textbf{recursive coatom ordering} if its dual poset $P^*$ admits a recursive atom ordering.
%\end{defn}

We next review the notions of  CC-shellability  and EC-shellability from \cite{kozlov}, doing so  using the language of topological ascents and descents (defined next) that was  
introduced in \cite{hersh}.  % where this notion was rediscovered.  
%
%\begin{defn} 
Let $\lambda$ be an edge labeling  on the cover relations of a poset $P$ by elements of some poset $Q$.
%Consider elements  $u \lessdot v \lessdot w$ in $P$. 
 A \textbf{topological ascent} occurs in $P$ at $u\lessdot v\lessdot w$  whenever the ordered pair of labels $(\lambda (u, v)$, $\lambda (v, w))$ is lexicographically earlier than all other ordered sequences  of labels on all other saturated chains from $u$ to $w$. If $(\lambda(u, v), \lambda(v, w))$ is not a topological ascent, then it is a \textbf{topological descent}. 
  
 For $\lambda $ a chain-edge labeling, we define topological ascents and descents in the same way, but now with respect to a choice of root.  That is, we have a \textbf{topological ascent} at $u\lessdot v\lessdot w$ with respect to root $r$ if the ordered pair $(\lambda (u,v), \lambda (v,w))$ is lexicographically smaller than all other label sequences  on saturated chains from $u$ to $w$ with  this same root $r$.   We have a \textbf{topological descent} at $u\lessdot v\lessdot w$ 
 % \commentph{Add: 
 with respect to  root $r$  otherwise.  
 
 If $c$ is a saturated chain  from  an element $u$ to an element  $v$ consisting entirely of topological ascents (in either an edge labeling or a chain-edge labeling), then we say that $c$ is a \textbf{topologically ascending} chain from $u$ to $v$.   If $c$ is a saturated chain from $u$ to $v$ consisting entirely of topological descents, then it is said to be \textbf{topologically descending}.

The next definition is rephrased from how it appears in \cite{kozlov}, using the language of topological ascents and descents, but is entirely equivalent to his definition.
%??? be careful that Kozlov never uses the terms topological ascent or descent but rather something equivalent to this

\begin{defn}[\cite{kozlov}]An \textbf{EC-labeling} of  a finite bounded poset $P$ is an edge labeling $\lambda:E(P) \rightarrow Q$ on the cover relations of $P$ with labels belonging to a poset $Q$, subject to the requirements  that (1) every interval $[x,y]$ has a unique saturated chain consisting entirely of topological ascents, and (2) the label sequences for the saturated chains of $P$ (and hence of each interval $[x,y]$) are all distinct with no label sequence being a prefix of any other label sequence.  

If $P$ admits an EC-labeling, then $P$ is said to be \textbf{EC-shellable}. 
\label{ec}
\end{defn}

Just as EL-shellability was generalized to CL-shellability, the notion of EC-shellability was likewise generalized by Kozlov
from edge labelings to chain-edge labelings.  His definition from \cite{kozlov} may be rephrased  as follows.

\begin{defn}[\cite{kozlov}] A \textbf{CC-labeling} of a finite bounded poset $P$ is a chain-edge labeling $\lambda: E^*(P) \rightarrow Q$ with labels belonging to a poset $Q$, subject to the requirements  that (1)  every rooted interval $[u,v]_r$ has a unique saturated chain consisting entirely of topological ascents, and  (2) the label sequences for the saturated chains of $[u,v]_r$ are all distinct with no label sequence being a prefix of another label sequence.

 If a finite bounded poset $P$ admits a CC-labeling, then $P$ is said to be \textbf{CC-shellable}. 
\label{cc}
\end{defn}

%\commentplh{Removed dual statement here.}
%\begin{defn} 
%%%Let $P$ be a finite  bounded  poset. % and $P^*$ its dual poset.
% If  the dual poset  $P^*$ to a finite bounded poset $P$  admits an EC-labeling (resp. CC-labeling), then 
% $P$ is said to be \textbf{dual EC-shellable} (resp. \textbf{dual CC-shellable}). 
%\end{defn} 

%As with CL-shellable posets, a 
Kozlov proved  in Theorem 3.8 of  \cite{kozlov} that the order complex of any finite bounded  poset  $P$ admitting a CC-labeling  is 
shellable, doing so  by taking any linear extension of the lexicographic
order on the  label sequences  for the maximal chains of $P$ as the shelling order for the 
corresponding  facets of $\Delta (P)$. 

Kozlov noted in \cite{kozlov} that any finite bounded poset that is CL-shellable is also CC-shellable, leaving the proof to the reader.  We include  a proof of this result as
Proposition \ref{RAOimpliesCC}.  Our proof  below closely follows parts of 
the proof of Bj\"orner and Wachs that any recursive atom ordering induces a CL-labeling  (see Theorem 3.2 in \cite{bw}).
%, filling in more details than appear in the proof sketch of that result earlier in this section.   %Again we will also use the sets $F_r(u)$ and $G_r(u)$ which we review in Definition \ref{F-u}.
This proof  is included because it  
introduces several additional   important  ideas that will be used in various ways  in other proofs 
later in the paper.

\begin{prop}\label{RAOimpliesCC}
Any recursive atom ordering $a_1,\dots ,a_t$ of a finite bounded poset $P$ gives rise to a CC-labeling of $P$.  Thus, 
any finite bounded poset that is CL-shellable is CC-shellable.
\end{prop}

\begin{proof}
%\commentgs{Can Tricia please review?}\commentph{Done: notice my comments below, which hopefully you agree with.}
%\commentph{The reason I am taking the perspective of a chain-atom ordering is so the labeling is well-defined as we go up all the way in the poset, since we will need all these labels in our proof.} 
%\commentph{Add the phrase: 
Given a recursive atom ordering $\Lambda $,  we %}  \sout{We}
 start by describing the desired CC-labeling $\lambda $ %
 %\commentph{Add: 
 derived from $\Lambda $.  %\commentph{Add: 
Letting $a_1,\dots ,a_t$ denote the atom ordering of $P$ given by $\Lambda $,
%\sout{First,}  
we assign the label $\lambda (\hat{0}, a_i) = i$ to each cover relation of the form $\hat{0}\lessdot a_i$.  Now, for any $u\in P$ and any root $r$ leading up to $u$, consider  any % \commentph{Add: the} 
  recursive atom ordering $a'_{1},\dots ,a'_{t'}$ for $[u,\hat{1}]_r$ 
%\commentph{Add: given by $\Lambda $}
of the type guaranteed to exist recursively in the definition of RAO.    Assign the labels $\lambda_r(u,a_i') = i$ for $i=1,2,\dots ,t'$.

By construction, this labeling has two of  the requisite properties of a CC-labeling, namely  the requirements  that 1) no two saturated chains of $[u, \hat{1}]_r$ have the same label sequence and 2) no two saturated chains of $[u, \hat{1}]_r$ have the property that the label sequence of one is a prefix of the label sequence of the other.  
%One may use the reasoning from the proof of Bj\"orner and Wachs that any RAO gives rise to a chain-edge labeling in which each rooted interval has a unique ascending chain and that this label sequence is lexicographically smallest  to show likewise that for 
%labeling $\lambda $ that we have constructed  (which is a close relative of the labeling of Bj\"orner and Wachs) that 
Now we verify that each rooted interval $[u, v]_r$ in $P$ has a unique topologically ascending maximal chain and that  this chain is lexicographically first. We do so by induction on the length of the longest chain in $[u, v]_r$.  This hypothesis clearly holds for intervals of length two, the base case. % Specifically, we use the observation from \cite{bw} that 
Let $c$ which is  given by $u \lessdot u_1 \lessdot \ldots \lessdot u_k = v$ be the lexicographically first maximal chain in $[u, v]_r$. Because it is lexicographically first, this forces $c$ to be  topologically ascending. 

%\commentph{I added paragraph break here, since people need to pause while reading, so can't handle such long paragraphs.}  
Suppose that there is another topologically ascending maximal chain $c'$  in $[u,v]_r$.  Let $c'$ be  given by $u \lessdot u_1' \lessdot u_2' \lessdot \ldots \lessdot u_{k'}' = v$.   As $u \lessdot u_1' \lessdot u_2'$ must then be  a topological ascent, it is the lexicographically earliest maximal  chain from $u$ to $u_2'$. 
%\commentph{I need to think more carefully  about whether the next sentence really follows from this directly.  What is bothering me is that we can't just go freely back and forth between the chain-atom ordering and the chain-labeling, so we need to be careful to specify which  chain-atom ordering we are talking about in speaking of $G_{r\cup u}(u_1')$.} 
 This implies $u_2' \in G^{\Lambda }_{r \cup u_1'}(u_1')$. %  with respect to the RAO  $\Lambda $. 
 As $c$ is lexicographically earlier than $c'$,
 % \commentph{why can't we have that  $\lambda (u,u_1) = \lambda (u,u_1')$?  I edited the next bit to address this.} 
 we may conclude that $\lambda(u, u_1) \le   \lambda(u, u_1')$.  But our choice of chain-edge labeling based on an RAO ensures that 
$\lambda (u,u_1)\ne \lambda (u,u_1')$, implying  $\lambda (u,u_1) < \lambda (u,u_1') $  from which we deduce  
 that $u_1$ comes before $u_1'$ in the  % \commentph{Add: 
 given %}
   RAO of $[u, \hat{1}]_r$. By (ii) of Definition \ref{rao}, there exists an atom $a$ of 
 %\commentph{typo: $u$ needs to be an interval, I believe
  $[u,v]_r$ and an element $z \in P$ such that $a$ comes before $u_1'$ in the RAO,  $u_1' \lessdot z$ and $a < z \leq v$. If $z = u_2'$, then $u_2' \in F^{\Lambda }_{r \cup u_1'}(u_1')$,  %\commentph{Be careful throughout this proof where we put subscript for the root -- be consistent with elsewhere in paper.}, 
  contradicting our earlier claim that $u_2' \in G^{\Lambda }_{r \cup u_1'}(u_1')$. If $z \neq u_2'$, then $u_1' \lessdot u_2' \lessdot \ldots u_{k'}'=v$ is not lexicographically first in $[u_1', v]_{r \cup u_1'}$ and, by the induction hypothesis, not topologically ascending. This gives  a contradiction to the claim that $c'$ consists entirely of topological ascents. Thus, we confirm that the chain-edge labeling $\lambda $ has a unique topologically ascending chain in $[u,v]_r$, completing our proof that this is a CC-labeling.  

Combining the above argument  with the result of \cite{bw} that a finite bounded poset is CL-shellable if and only if it admits a recursive atom ordering shows that every  CL-shellable poset is CC-shellable.
\end{proof}

\section{Generalized recursive atom ordering} %: definition and fundamental properties}
 %and equivalence  %of existence 
%to self-consistent  CC-shellability }
\label{graosection}

In this section we introduce a  generalization of  
the notion of recursive atom ordering and prove  %the following
several   fundamental properties of these  generalized recursive atom orderings.  

%\begin{defn}\label{chain-atom-def}
%A \textbf{chain-atom ordering} of a finite bounded poset $P$  is a choice of ordering on the atoms of each rooted interval  
%$[u,\hat{1}]_r$  of $P$.  
%\end{defn}

% \commentph{Tricia commented out the definition of recursive atom ordering just above, but has not yet done anything about the next remark that should perhaps be deleted, or else moved and likely shortened.}

%\commentph{The next remark incorrectly speaks of things being graded.}

%\commentph{I re-read the definition of Bj\"orner and Wachs of recursive atom ordering and now believe (due to their  usage of the word ``admits'' rather than "is") that it is very ambiguous whether they are defining an RAO to be a  chain-atom ordering or just an atom ordering.  As a reminder, the referee asked us to remove the next remark saying it would likely only cause confusion, but we decided not to remove it but just to shorten it.  Probably this is the right call, but I wasn't sure now, so put this remark. }

\begin{rmk}\label{chain-atom-vs-atom}
%\sout{A recursive atom ordering is an ordering of just the atoms of a finite % \commentph{graded} 
%poset $P$, but it is an atom ordering for which their exists one or more extensions 
%to chain-atom orderings of $P$  of the type that is guaranteed to exist by the various requirements in the definition of  an RAO.
%% %in a manner that satisfies the requirements of the chain atom orderings guaranteed to exist for  $[a_i,\hat{1}]$ for each atom of $P$.  
%  It is often useful to abuse notation and  think of a recursive atom ordering as such a chain-atom ordering rather than as  an ordering of just the  atoms.  }
%
The generalized recursive atom orderings (GRAOs)  we are about to introduce 
will  %\sout{likewise} 
 be  defined  as 
  orderings on the atoms of a finite  %\commentph{graded} 
poset $P$ that are extendible to chain-atom orderings of $P$  meeting certain requirements.
%of the type  guaranteed to exist by the requirements for a GRAO.
%\commentplh{Move this next part to after the definition of GRAO: 
 %In upcoming proofs, i
 It will often be convenient to 
 think of a generalized recursive atom ordering as a % type of 
 chain-atom ordering of this type. % rather than as an ordering of just the atoms.
 \end{rmk}
 
% \commentph{This next remark also seems like a very good place to explain that when we say there are more GRAOs than RAOs, we are thinking of them as chain-atom orderings, not just as atom orderings.  I made a first attempt at making this point here.}

%??? should we change ``this GRAO'' in item (ii) into ''such a GRAO'' to allow for the possibility that there is a choice?

%\commentph{Keep:?}  
%Having   made these critical   remarks, %no
Now we are ready to give the  main new definition of the paper: %that  of generalized recursive atom ordering:

\begin{defn} A  finite bounded poset $P$ admits a \textbf{generalized recursive atom ordering} (GRAO) if the length of $P$ % (i.e. the length of the longest chain in $P$) 
 is 1 or if the length of $P$ is greater than 1 and there is an ordering $a_1, a_2, \ldots a_t$ on the atoms of $P$ satisfying: 
\begin{enumerate}[label=(\roman*)] 
\item \begin{enumerate}
\item For $1 \leq j \leq t$, $[a_j, \hat{1}]$ admits a  GRAO. %generalized recursive atom ordering 
\item For any  atom $a_j$ and any $x,w\in P$ satisfying % $a_j, x$, and $w$ where  
$a_j \lessdot x \lessdot w$, the following property holds when the chain-atom ordering given by the  
GRAO from  (i)(a) is restricted to $[a_j, w]$: either the first atom of $[a_j, w]$ is above an
atom  $a_i$ with $i < j$, or no atom of $[a_j, w]$ is above any atom  $a_i$ with $i < j$. 
\end{enumerate} 
\item % For all $i < j$, i
%\gs{Phrasing for RAO (better to use our own phrasing both to avoid plagiarism and to be more in the style we have written this definition): For all $i < j$ and $y \in P$ satisfying $ y > a_i $ and $y >  a_j$, there exists  $k < j$ and  $z\in P$ such that $a_j \lessdot z$ and $a_k < z \leq y$.} 
For  any   $y\in P$ and any  atoms  $a_i,a_j $ satisfying $a_i < y$ and $a_j < y$ with $i<j$,  %,  for $i<j$,
there  exists 
% an atom $a_k$ with   %some
 %$k < j$ and 
 an element 
$z\in P$  with  $z\le y$ and an atom $a_k$ with $k<j$  such that  $a_j \lessdot z$ and  $a_k < z$. 
%and  $a_j \lessdot z \leq y$. %  for some $k<j$. 
\end{enumerate}
\label{grao}

\end{defn}

\begin{rmk}\label{more-general}
 %  \sout{That is, many of our results regarding GRAOs are proven by taking the atom ordering of any GRAO and proving the result  for every chain-atom ordering of the type guaranteed to exist as an extension of the atom ordering of  the given GRAO.}
%\commentph{Add something like: 
When we say that not all recursive atom orderings are generalized recursive atom orderings  and when we say that  GRAOs are  strictly more general than RAOs, we are regarding  RAOs and GRAOs  as being types of  chain-atom orderings. 
%  rather than  being types of  atom ordering that are extendible to such chain-atom orderings.
% and observing that there are many more chain-atom orderings meeting the requi
\end{rmk}

\begin{rmk}
Conditions (i)(a) and (ii) in the definition of GRAO are exactly equivalent to corresponding statements within the definition of recursive atom ordering.  However,  
condition (i)(b)  above  is considerably less restrictive than (i)(b) of the definition of RAO.  
\end{rmk}

One of the main points of this relaxation of the notion of recursive atom ordering 
is % that it is designed to 
that it gives considerably more flexibility than a recursive atom ordering  in how we may  %choosing an atom 
order %ing %in terms of how to order 
those atoms in a rooted interval $[u,\hat{1}]_r$ that are not the 
first atom of that interval.    For this reason,  the following rephrasing of condition (ii) also seems quite useful to note:

\begin{rmk}\label{2-rephrased}
Condition (ii) in Definition ~\ref{grao} is logically equivalent to the following statement: if $a_i<y$ and $a_j<y$ for $i<j$, 
then there exists  $z\in P$ such that $a_j\lessdot z\le y$ where   $a_j$ is not the earliest atom in $[\hat{0},z]$.    
\end{rmk}

%\commentph{Do we explain/mention our usage of red anywhere?  If we use a color (which seems a good idea), we should say somewhere what it is indicating.}  

\begin{ex}
The chain-atom ordering  given on the left in Figure \ref{graoex} is a generalized recursive atom ordering but is not a (traditional) recursive atom ordering. In particular, the atom of $[a, \hat{1}]$ labeled 3 in the poset on the left causes condition  (i)(b) of Definition \ref{rao} to fail. The chain-atom ordering given on the right is a recursive atom ordering.    We highlight in color  the place where the two chain-atom orderings differ.
\end{ex}  
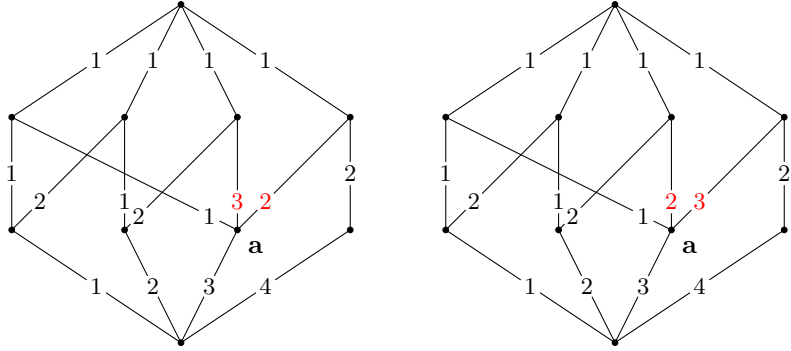
\begin{figure}
\begin{center}
\begin{tikzpicture}[scale=0.75]
\node at (3,0){};
\draw [fill] (3,0) circle [radius=0.05]; 
\node at (0,2){};
\draw [fill] (0,2) circle [radius=0.05];
\node at (2,2){};
\draw [fill] (2,2) circle [radius=0.05];
\node [below right] at (4,2){\footnotesize \textbf{a}};
\draw [fill] (4,2) circle [radius=0.05];
\node at (6,2){};
\draw [fill] (6,2) circle [radius=0.05];
\node at (0,4){};
\draw [fill] (0,4) circle [radius=0.05];
\node at (2,4){};
\draw [fill] (2,4) circle [radius=0.05];
\node at (4,4){};
\draw [fill] (4,4) circle [radius=0.05];
\node at (6,4){};
\draw [fill] (6,4) circle [radius=0.05];
\node at (3,6){};
\draw [fill] (3,6) circle [radius=0.05];
\draw (3,0)--node [midway, circle, fill=white, inner sep=0.5pt,minimum size=1pt] {\scriptsize 1}(0,2)--node [midway, circle, fill=white, inner sep=0.5pt,minimum size=1pt] {\scriptsize 1}(0,4)--node [midway, circle, fill=white, inner sep=0.5pt,minimum size=1pt] {\scriptsize 1}(3,6);
\draw (0,2)--node [near start, circle, fill=white, inner sep=0.5pt,minimum size=1pt] {\scriptsize 2}(2,4)--node [midway, circle, fill=white, inner sep=0.5pt,minimum size=1pt] {\scriptsize 1}(3,6);
\draw (3,0)--node [midway, circle, fill=white, inner sep=0.5pt,minimum size=1pt] {\scriptsize 2}(2,2)--node [near start, circle, fill=white, inner sep=0.5pt,minimum size=1pt] {\scriptsize 1}(2,4);
\draw (2,2)--node [very near start, circle, fill=white, inner sep=0.5pt,minimum size=1pt] {\scriptsize 2}(4,4)--node [midway, circle, fill=white, inner sep=0.5pt,minimum size=1pt] {\scriptsize 1}(3,6);
\draw (3,0)--node [midway, circle, fill=white, inner sep=0.5pt,minimum size=1pt] {\scriptsize 3}(4,2)--node [very near start, circle, fill=white, inner sep=0.5pt,minimum size=1pt] {\scriptsize 1}(0,4);
\draw (4,2)--node [near start, circle, fill=white, inner sep=0.5pt,minimum size=1pt] {\scriptsize \textcolor{red} 3}(4,4);
\draw (4,2)--node [near start, circle, fill=white, inner sep=0.5pt,minimum size=1pt] {\scriptsize \textcolor{red} 2}(6,4)--node [midway, circle, fill=white, inner sep=0.5pt,minimum size=1pt] {\scriptsize 1}(3,6);
\draw (3,0)--node [midway, circle, fill=white, inner sep=0.5pt,minimum size=1pt] {\scriptsize 4}(6,2)--node [midway, circle, fill=white, inner sep=0.5pt,minimum size=1pt] {\scriptsize 2}(6,4);
\end{tikzpicture}
\hskip 2em
\begin{tikzpicture}[scale=0.75]
\node at (3,0){};
\draw [fill] (3,0) circle [radius=0.05]; 
\node at (0,2){};
\draw [fill] (0,2) circle [radius=0.05];
\node at (2,2){};
\draw [fill] (2,2) circle [radius= 0.05];
\node [below right] at (4,2){\footnotesize \textbf{a}};
\draw [fill] (4,2) circle [radius=0.05];
\node at (6,2){};
\draw [fill] (6,2) circle [radius=0.05];
\node at (0,4){};
\draw [fill] (0,4) circle [radius=0.05];
\node at (2,4){};
\draw [fill] (2,4) circle [radius=0.05];
\node at (4,4){};
\draw [fill] (4,4) circle [radius=0.05];
\node at (6,4){};
\draw [fill] (6,4) circle [radius=0.05];
\node at (3,6){};
\draw [fill] (3,6) circle [radius=0.05];
\draw (3,0)--node [midway, circle, fill=white, inner sep=0.5pt,minimum size=1pt] {\scriptsize 1}(0,2)--node [midway, circle, fill=white, inner sep=0.5pt,minimum size=1pt] {\scriptsize 1}(0,4)--node [midway, circle, fill=white, inner sep=0.5pt,minimum size=1pt] {\scriptsize 1}(3,6);
\draw (0,2)--node [near start, circle, fill=white, inner sep=0.5pt,minimum size=1pt] {\scriptsize 2}(2,4)--node [midway, circle, fill=white, inner sep=0.5pt,minimum size=1pt] {\scriptsize 1}(3,6);
\draw (3,0)--node [midway, circle, fill=white, inner sep=0.5pt,minimum size=1pt] {\scriptsize 2}(2,2)--node [near start, circle, fill=white, inner sep=0.5pt,minimum size=1pt] {\scriptsize 1}(2,4);
\draw (2,2)--node [very near start, circle, fill=white, inner sep=0.5pt,minimum size=1pt] {\scriptsize 2}(4,4)--node [midway, circle, fill=white, inner sep=0.5pt,minimum size=1pt] {\scriptsize 1}(3,6);
\draw (3,0)--node [midway, circle, fill=white, inner sep=0.5pt,minimum size=1pt] {\scriptsize 3}(4,2)--node [very near start, circle, fill=white, inner sep=0.5pt,minimum size=1pt] {\scriptsize 1}(0,4);
\draw (4,2)--node [near start, circle, fill=white, inner sep=0.5pt,minimum size=1pt] {\scriptsize \textcolor{red} 2}(4,4);
\draw (4,2)--node [near start, circle, fill=white, inner sep=0.5pt,minimum size=1pt] {\scriptsize \textcolor{red} 3}(6,4)--node [midway, circle, fill=white, inner sep=0.5pt,minimum size=1pt] {\scriptsize 1}(3,6);
\draw (3,0)--node [midway, circle, fill=white, inner sep=0.5pt,minimum size=1pt] {\scriptsize 4}(6,2)--node [midway, circle, fill=white, inner sep=0.5pt,minimum size=1pt] {\scriptsize 2}(6,4);
\end{tikzpicture}
\caption{GRAO that is not RAO (left) and RAO for same poset (right).}
\label{graoex}
\end{center}
\end{figure}
%
%
%
%
%\commentplh{Removed dual statement here.}
%Naturally, it  also makes sense to establish the following dual notion:
%
%\begin{defn}  
%A  \textbf{generalized recursive coatom ordering} (GRCO) of a finite bounded poset  $P$ is 
%a GRAO of  the dual poset $P^*$.
%\end{defn}
%
%
\begin{rmk}\label{restricted-GRAO}
Implicit in the definition of GRAO is that a GRAO for a finite bounded poset $P$ induces a GRAO for $[u,\hat{1}]_r$ for each $u\in P$ and each root $r$.
Lemma \ref{GRAOrestricts}  will show that it  also induces a GRAO for each $[u,v]_r$.  This  result will allow us henceforth to  refer 
to the  restriction of any  GRAO for a finite bounded poset to the  rooted interval  % the atoms of
 $[u, v]_r$ as the  GRAO for $[u, v]_r$ induced by the  GRAO for $P$.
 %, as we indeed do in Lemmas \ref{ib-preserved} and \ref{2-flippable} as well as  in Theorem \ref{switch-thm}. 
\end{rmk}

\begin{lem}\label{GRAOrestricts}
Let $P$ be a finite bounded poset  with $\Gamma $   a GRAO for $P$. Then for any $u<v$ and any root $r$ for $[u,v]$, restricting % to $[u,v]_r$ 
the GRAO for $[u,\hat{1}]_r$  induced by $\Gamma $  to $[u,v]_r$  %that is guaranteed to exist by Definition \ref{grao}  
yields    a GRAO for $[u,v]_r$.
\end{lem}

\begin{proof}
Our proof will be by induction on the length of the longest saturated chain from $u$ to $\hat{1}$.  We can use maximum length 2 for our base case, obtaining this case by using the fact that  every possible chain-atom ordering on a finite bounded poset with maximum chain length of 1 or 2 is a GRAO.  
We then obtain condition (i)(a) by induction as follows.  Given any $a_i$ satisfying $u\lessdot a_i \le v$, the length of the longest saturated chain from $u$ to $\hat{1}$ is strictly larger than the length of the longest saturated chain from $a_i$ to $\hat{1}$.  
%So we may assume the lemma holds for $[a_i,v]_{r\cup a_i}$ in order to deduce the lemma for $[u,v]_r$.   
Thus, our inductive hypothesis gives us that the GRAO for $[a_i,\hat{1}]_{r\cup a_i}$ %guaranteed to exist as a result of the GRAO for $P$ 
will restrict to a GRAO for $[a_i,v]_{r\cup a_i}$.   But it is immediate from the definition of GRAO that the properties of a GRAO given in conditions (i)(b) and (ii) from the definition of GRAO will restrict from $[u,\hat{1}]_r$ to $[u,v]_r$.  Thus, the GRAO for $P$ restricts just as desired.
\end{proof}

Next we show  how the statement about cover relations in  condition (i)(b) in the definition of GRAO can be strengthened to a corresponding 
statement about all order 
relations.  

%\commentph{I added back the phrase ``with atom ordering'' in the first sentence for two reasons.  For one thing,  otherwise we would not have addressed one of the referee's complaints -- the complaint about how confusing it is to have the phrase GRAO immediately followed by symbols.  More importantly, I also felt that this statement inherently assumes we are talking about a chain-atom ordering, not just an atom ordering, which made the phrasing without ``with atom ordering'' problematic.}

%\commentph{Should we add roots to the intervals in the statement and the  proof of the next lemma? -- agreed and done.}

%\commentph{Should we use $\Lambda $ and $\Lambda| $ etc in the next lemma.}

%\commentph{It's annoying reading all the subscript $a_i$'s in the next lemma, and it is probably distracting to readers.  We should probably suppress them, perhaps saying we are doing so.}

\begin{lem}
 Let $P$ be  a finite bounded poset, and let $\Lambda $ be a GRAO for $P$ % \gs{that admits  a GRAO}  
 with atom ordering  
$a_1, a_2, \ldots a_t$.   
For each $\hat{0} \lessdot a_j < v$, restricting  $\Lambda|_{[a_j,\hat{1}]}$ 
%\gs{restricting % \commentph{which one?  Let's use $\Lambda $ to fix this.} 
 % the GRAO for $[a_j,\hat{1}]$}  
   to  %$[a_j,v]_
 %$a_j \gtrdot \hat{0}$ and each $v > a_j$ 
%the induced  GRAO for %of  $[a_j, \hat{1}]$ restricted to
 $[a_j, v]$ %_{\hat{0}\lessdot a_j}$ 
 %induced by the given GRAO 
 yields a GRAO, denoted $\Lambda|_{[a_j,v]}$,  for $[a_j,v]$ %_{\hat{0}\lessdot a_j}$ 
  with 
 the following property: either (a) the first atom of $[a_j, v]$  %_{\hat{0}\lessdot a_j}$
  is greater than some  atom $a_i$ satisfying $i < j$ or  (b) no atom of 
 $[a_j, v]$ %_{\hat{0}\lessdot a_j}$
  is greater than  any atom 
 $a_i$ satisfying $i < j$. 
\label{graoibequiv}
\end{lem}

\begin{proof}
By Lemma \ref{GRAOrestricts}, $\Lambda|_{[a_j,v]}$ %_{\hat{0}\lessdot a_j}}$
 is guaranteed to be a GRAO.  Let $x$ be the first atom in % the GRAO for
$[a_j,v]$ %_{\hat{0}\lessdot a_j}$ 
with respect to $\Lambda $.
%\gs{Let  $x$ be the first atom in the GRAO for $[a_j, v]_{\hat{0}\lessdot a_j}$
 %\commentph{Add something like: 
% that is guaranteed to be a GRAO by Lemma \ref{GRAOrestricts}.   }
Suppose % $x$ is not greater than any $a_i$ with $i < j$ but
 that some other atom $y$  of  $[a_j, v]$ %_{\hat{0}\lessdot a_j}$ 
  is greater than some  atom $a_i$  of $P$ with $i<j$. Since $x$ and $y$ are both less than  $v$ with $x$ coming before $y$ in $\Lambda|_{[a_j,\hat{1}]}$, %_{\hat{0}\lessdot a_j}}$, 
%   \gs{the GRAO of $[a_j, \hat{1}]_{\hat{0}\lessdot a_j}$},
    Definition \ref{grao} (ii)  guarantees the  existence of  elements  $w_1$ and  $y_1$, where  $y_1$ comes before $y$ in  $\Lambda|_{[a_j,\hat{1}]}$
    %_{\hat{0}\lessdot a_j}}$ 
  %\gs{ the GRAO of $[a_j, \hat{1}]_{\hat{0}\lessdot a_j}$} 
  and where 
 %\commentph{I think $w_1 < v$ needs to be $w_1\le v$ next}   
 $y \lessdot w_1 \le  v$ and $y_1 < w_1$. Since $y$ covers  $a_i$ with $i < j$, Definition \ref{grao} (i)(b) ensures that the first atom in $[a_j, w_1]$ %_{\hat{0}\lessdot a_j}$ 
 must be greater than  some atom $a_{i'}$  satisfying $i' < j$.  We may take $y_1$ to be  the first atom in $[a_j,w_1]$, %_{\hat{0}\lessdot a_j}$,
  i.e.,  we may  replace $y_1$ by this first  atom if it is not already this atom. 
If $y_1$ also equals $x$, then $x$ is greater than  some atom that comes before $a_j$ in $\Lambda $,
 %\gs{ the GRAO of $P$}, 
 namely $a_{i'}$, as desired. 

If, on the other hand, $y_1 \neq x$, then we repeat this process as many times as necessary to get the desired result that $x$ is greater than some atom that comes before $a_j$ in  $\Lambda $,
%\gs{ the GRAO of $P$}, 
doing so as follows.   %, this time using $y_1$ in place of $y$: 
Since $x$ and $y_1$ are both less than  $v$,  we may use Definition \ref{grao} (ii) to deduce the existence of  $y_2 $ and  $w_2$ where $y_2$ comes before $y_1$ in 
$\Lambda|_{[a_j,\hat{1}]}$ %_{\hat{0}\lessdot a_j}}$  %\gs{the GRAO of $[a_j, \hat{1}]_{\hat{0}\lessdot a_j}$ } 
and where % $a_j \lessdot 
$y_2 < w_2$ and $y_1 \lessdot w_2\le v$.  We may take 
$y_2$ to be the first atom in  $[a_j, w_2]$. %_{\hat{0} \lessdot a_j}$. 
But then Definition \ref{grao} (i)(b)  guarantees that  $y_2$  must be greater than 
some atom  $a_{i''}$ where  $i'' < j$. If $y_2=x$, this yields the desired result. % shows contradicts the fact 
%that $x$ is greater than  some atom that comes before $a_j$ in the GRAO of $P$. 
If $y_2 \neq x$, then we repeat this process again obtaining elements $y_3$ and $w_3$ in place of 
$y_2$ and $w_2$. 

 Continuing in this manner, we obtain a sequence $y_1,y_2,y_3,\dots  $ of atoms of $[a_j, v]$ %_{\hat{0}\lessdot a_j}$ 
 where for each $i \ge 2$ we have that:
 \begin{enumerate}
 \item  
 $y_i$  comes earlier in $\Lambda|_{[a_j,v]}$ % _{\hat{0}\lessdot a_j}}$ 
  %\gs{the GRAO of $[a_j, v]_{\hat{0}\lessdot a_j}$} 
   than $y_{i-1}$, and   
 \item
 %(ii) 
 $y_i$ is greater than an atom  of $P$ that comes before $a_j$ in $\Lambda $. % \gs{ the GRAO of $P$}.
 \end{enumerate}
Note that  (1) implies each $y_i$ is distinct, allowing us to deduce from % distinctness of the $y_i$'s, 
 finiteness of 
%Since there is a finite set of atoms of 
$[a_j, v]$ that the sequence $y_1,y_2,y_3,\dots $  must terminate at some $y_n$. 
% process must terminate, say, at $y_n$. %implying that some  such  $y_n$ is  the earliest atom in the GRAO of $[a_j, v]$.  
This terminal element  $y_n$  must satisfy  $y_n = x$ since we have shown that  otherwise there 
would be an atom $y_{n+1}$ of $[a_j,v]$ coming still earlier than $y_n$. %_{\hat{0}\lessdot a_j}$.    %for this $y_n$, 
Thus we obtain  the desired result  % contradiction to  the fact
 that $x$ is greater than  some atom that comes before $a_j$ in $\Lambda $. % \gs{ the GRAO of $P$}.
 % by virtue of $y_n$ having this property.
\end{proof}

%\commentph{I just edited the next bit more in light of your questions/comments.} 
%\commentph{Add superscripts to  notation next?}
%\gs{Is it the sets or the lemma that is playing a key role?}
The sets $F_r(u,v)$ and $G_r(u,v)$ appearing in the next lemma
 %, defined next,  of the sets $F_r(u)$ and $G_r(u)$ introduced by Bj\"orner and Wachs in  \cite{bw} 
will  play a key role  later in the paper %,   for instance 
in allowing us to transform any generalized recursive atom ordering into a recursive atom 
ordering.   
  % (\commentph{it's confusing to say ``for instance'' again, in that it's confusing what it modifies} 
%Two  places where these sets  will play an especially  important role are   in Algorithm \ref{reorder-def} and in  
%the proof of Lemma \ref{alwaysfirst}.   
Our justification that this atom reordering process converts a GRAO to an RAO will rely heavily on the property of
the sets $F_r(u,v)$ and $G_r(u,v)$ that we justify next.  See Definition \ref{F-u} for the definitions of the sets $F_r(u,v)$ and $G_r(u,v)$.

\begin{lem}\label{F-and-G-preserved}
Let $\Gamma $ be a chain-atom ordering  
 for a finite bounded poset $P$.  Consider  any  two   %consecutive 
  atoms $a_i,a_{i+1}$ of  a rooted interval $[u,\hat{1}]_r$ % in $P$  
   that are consecutive   %atoms 
 atoms  of $[u,\hat{1}]_r$ 
 with respect to $\Gamma $.  Suppose   the pair  of atoms  
  $a_i,a_{i+1}$   has  the further 
 property  for every   rooted interval $[u,v]_r$ % \sout{with our fixed choice of $u$ and $r$ 
  %has} 
  %\gs{
  containing both $a_i$ and $a_{i+1}$ % \sout{ $a_i,a_{i+1}\in [u,v]_r$ }
 %with  our fixed choice of $u$ and $r$  as above 
 % $a_i,a_{i+1}$ 
%have the further property 
%\commentph{I wanted this next phrase because otherwise the sentence is too difficult to parse} 
%\sout{  that this implies that} 
that  neither $a_i$ nor $a_{i+1}$ is the first atom of $[u,v]_r$  with respect to   %the restriction of  
$\Gamma|_{
% $ to  % any rooted interval  
[u,v]_r}$.
 %having $a_i,a_{i+1}\in [u,v]_r$  for our fixed choice of $u$ and $r$.
%  which  %for any  $v>u$ 
% has $a_i,a_{i+1}\in [u,v]$.  
 % both $a_i$ and $a_{i+1}$ as elements of $[u,v]$. % and having the same element $u$ and the same root $r$.
%This would be consistent with the upcoming lemmas the referee also asks about, and I think was just forgotten somehow here.  This would address the referee's concern, since it would preclude the case of $v = a_i$ or $v=a_{i+1}$.} 
 %as in the statement of Theorem  \ref{switch-thm}, and
Let $\Lambda $ be the chain-atom ordering on $P$   %$\Lambda $
obtained from $\Gamma $  by 
 switching the order of $a_i$ and $a_{i+1}$ in $[u,\hat{1}]_r$. % \sout{for our fixed $u$ and $r$}
%  and otherwise leaving $\Gamma $ unchanged. 
 %to obtain a new chain-atom ordering $\Lambda  $, then % does not impact 
% which atoms of $[a_i,\hat{1}]_{r\cup a_i}$ (resp. $[a_{i+1},\hat{1}]_{r\cup a_{i+1}}$) are
%in $F_{r\cup a_i} (a_i,v)$ (resp. $F_{r\cup a_{i+1}}(a_{i+1},v)$) and which are in $G_{r\cup a_i}(a_i,v)$ (resp. $G_{r\cup a_{i+1}}(a_{i+1},v)$) for any $v$ satisfying
%$v>a_i$ or  $v>a_{i+1}$.  More generally, 
Then  $F^{\Gamma  }_{r'}(u',v') = F^{\Lambda }_{r'}(u',v')$ and 
 %(resp.  
 $G^{\Gamma }_{r'}(u',v') = %$) given by $\Gamma  $ equals the set 
%$ (resp. $
G^{\Lambda }_{r'}(u',v')$ %) % given by $\Lambda  $  %are fixed by the swap of $a_i$ and $a_{i+1}$  
for each % rooted interval $[u',v']_{r'}$ in $P$.  
$u' <  v'$ in $ P$   and each   choice of 
 root $r'$ for $[u',v']$.  
\end{lem}

\begin{proof}
%\commentgs{Can we lay out all the cases here as in the next lemma? } \commentph{Given the nature of the referee's confusion, I think it's better to first show how we can restrict to $u'$ covering $u$ and $r' = r\cup \{ u' \} $ and then at that point give the case structure  for what remains to be proven (which is what we do).  I actually put a ton of thought into how I organized things and did things for good reasons.    I think doing this easy case before the case structure of the rest helps get readers understanding what needs to be proven, as it can be very confusing.}
Throughout this proof, let $r$ be the chain $\hat{0} = u_0 \lessdot u_1 \lessdot \cdots \lessdot u_k \lessdot u$, and let 
$r^-$ be the chain $\hat{0} = u_0 \lessdot u_1 \lessdot \cdots \lessdot u_k$ obtained from $r$ by deleting $u$. 
%\commentph{I made a correction in your comment necessitating us to define $r^-$.}
% Notice that our hypotheses preclude the possibility of having $v \in \{ a_i ,a_{i+1}\} $.  
We will  first show how to  handle  all cases  which do not  satisfy both of the conditions %have  %restrict attention to the case where 
%$u'$ covering  
$u\lessdot u'$ and $r' = r\cup \{ u'\} $.  % Then we will  outline  how we break the remainder of the proof into cases.  
% with  the root  $r'$  obtained from $r$ by adding $u'$ to it.
% the chain 
%$\hat{0} = u_0 \lessdot u_1 \lessdot \cdots \lessdot u_k \lessdot u \lessdot u'$ 
%obtained from $r$ by adding $u'$ to it.
%To verify  this claim, the 
%\commentph{I like the sentence you propose, but haven't figured out a better way yet to write this paragraph that uses  it.}
 %\gs{Notice that the sets $F_r(u, v)$ and $G_r(u,v)$  can only change  when the ordering on the atoms of $[u_k,v]_{r^-}$ changes.}
 % Thus we need to consider only the following cases: [lay out cases].}
%point is to o
Let $\hat{0}  
%\commentph{I like the sentence you propose, but haven't figured out a better way yet to rewrite this paragraph in light of it.}
 %\gs{Notice that the sets $F_r(u, v)$ and $G_r(u,v)$  can only change  when the ordering on the atoms of $[u_k,v]_{r^-}$ changes.}
 % Thus we need to consider only the following cases: [lay out cases].}
 \lessdot u_1'\lessdot \cdots \lessdot u_l'\lessdot u'$ be the root $r'$.  
%Let  $(r')^-$ denote  the root   $\hat{0}  \lessdot u_1'\lessdot \cdots \lessdot u_l'$ obtained from $r'$ by deleting $u'$. % \sout{from it}.
%Observe that unless we 
Observe that if we do not have  % of the conditions %, observe that 
 %\gs{Since $u \centernot\lessdot u'$ and $r' \neq r \cup \lbrace u' \rbrace$}   \sout{
 $u = u_l'$ and also have   $r$ equalling  $\hat{0}\lessdot u_1'\lessdot \cdots \lessdot u_l'$, % = (r')^-$,  
then % observe that % the  root $r'$ contains $u$, 
swapping  the  order of the atoms $a_i$ and $a_{i+1}$ in $[u,\hat{1}]_r$  has no impact on % the ordering of the atoms of $[u',v']_{r'}$ and also has no impact on 
the ordering of the 
%\commentph{Should not be $u^-$ and $r^-$ but rather $(u')^-$ and $(r')^-$ next}
atoms of $[u_l',v']_{\hat{0}\lessdot u_1'\lessdot \cdots \lessdot u_l'}$.
%(r')^-}$. % where $u^-$ is the element covered by $u'$ in $r'$ and where $r^-$ is the root obtained from $r'$ by deleting $u'$ from it.  
%This   implies that
Thus,  the sets
$F^{\Gamma }_{r'}(u',v')$ and $G^{\Gamma }_{r'}(u',v')$ cannot be impacted by the swap of $a_i$ and $a_{i+1}$ 
unless  $u\lessdot u'$ and $r' = r\cup \{ u'\} $.
% \sout{ $u$ is covered by $u'$ and  $r=(r')^-$.  }
%\commentgs{I am getting bogged down in notation and the above change is an attempt to help with that. It is an exact repetition of what was said at the beginning of the case.}
% This change means the reader doesn't have to work at all to keep track of what root is being discussed. At first glance, I had to think carefully about whether the root we were discussing had changed since a different representation was used.}

%What remains is to handle those situations  in which % \sout{$u'$ covers $u$} \commentgs{see comment immediately above} \gs{
%$u \lessdot u'$ and   $r' = r\cup \{ u' \} $.
We subdivide  the  task of handling those situations with $u\lessdot u'$ and $r'=r\cup \{ u'\} $
 into three  cases, namely: (a) $u' = a_i$,  (b) $u' = a_{i+1}$, and  (c) $u'\not\in \{ a_i ,a_{i+1} \} $.  
We start with the easiest of these cases, namely  (c). %}\sout{Let us now turn to  the easiest of these cases, namely  (c)}\gs{, 
%in which  $u'$ is  an atom % $a\not\in \{ a_i,a_{i+1}\} $ 
%of $[u,v']_r$  with $u'\not \in \{ a_i,a_{i+1} \} $. % and with  $r' = r\cup \{ u' \} $. %}  
% the case
%of $u'$ which covers $u$ where $u' \not\in \{ a_i,a_{i+1} \} $.
%Next consider 
%\commentph{I put this sentence in because this was a spot the referee had gotten really confused, so I was trying to make sure they would not get confused this time.  Therefore I restored it, changing the language a bit.}
%\sout{I
%\sout{It may help to keep in mind in this case  that  $u'$ is  an atom % $a\not\in \{ a_i,a_{i+1}\} $ 
%of $[u,v']_r$  with $u'\not \in \{ a_i,a_{i+1} \} $ and with  $r' = r\cup \{ u' \} $.  } \commentgs{I worry a bit about being so repetitive that the reader loses patience, though I understand we are trying to be very clear. My proposed change above is an effort to be very clear without making it seem like we are repeating the same thing in adjacent sentences.} % But for  such $a\not\in \{ a_i,a_{i+1}\} $,  
Our choice of $a_i,a_{i+1}$ as consecutive atoms of $[u,\hat{1}]_r$ with respect to $\Gamma $  implies that $a_i,a_{i+1}$  are consecutive atoms  
of $[u,v']_r$ with respect to  $\Gamma|_{[u,v']_r}$ for each $v'\in P$ %  $u < v'\le \hat{1}$ 
having $a_i,a_{i+1}\in [u,v']_r$. 
%\commentgs{This is the one of the only places we omit the $r$; it looks like a typo. Maybe it is?} \commentph{This wasn't a typo, I systemically leave it off for set containments which do not depend on choice of root.  But you have already added lots of other places, and it is not wrong to add it here too if you prefer including it.  It's just unnecessary.}  
 This  
%(again with respect to $\Gamma $) 
ensures  for such  $[u,v']_r$   we must either have both $a_i$ and $a_{i+1}$ before 
the atom $u'$ % with respect to   $\Gamma $  
or    both $a_i$ and $a_{i+1}$ after $u'$ with respect to $\Gamma|_{[u,v']_r}$.  In 
either case, swapping the order of $a_i,a_{i+1}$ does not impact which atoms of $[u,v']_r$ come earlier than $u'$.
% with respect to \gs{either?} our chain-atom ordering %\commentgs{which one?}   \commentph{we aren't talking about a single one but rather the relationship between two different ones under certain conditions, so we need something like an indefinite article here} \commentgs{upon rereading this, ``our chain-atom ordering" gives the impression (at least to me) that we are speaking of a specific one. Perhaps we can omit ``with respect to our chain-atom ordering" entirely? I think it is clear without saying this.} 
Thus, this swap does not impact  which atoms of $[u',v']_{r'}$ are in  
$F_{r'}(u', v')$.  %\commentplh{Next part of this paragraph is new.} 
  If, on the other hand,  we do not have both $a_i$ and $a_{i+1}$ in $[u,v']_r$, then swapping the order of $a_i$ and $a_{i+1}$ in $[u,\hat{1}]_r$ leaves the ordering of 
the atoms of $[u,v']_r$ unchanged.  In particular, the swap   preserves  which atoms of $[u,v']_r$  come earlier than $u'$ and which come later than $u'$.  Again this 
implies that the swap does not impact which atoms of $[u',v']_{r'}$ are in $F_{r'}(u',v')$.   
%the swap of $a_i$ and $a_{i+1}$ also does not change the ordering of the atoms of $[u^-,v']_{r^-}$.  
%, which is enough to show that 
% whether $a$ belongs to 
%$F_{r'}(u',v')$ or to $G_{r'}(u',v')$.
%In all of these cases, t
This shows that   $F^{\Gamma }_{r'}(u',v') = F^{\Lambda }_{r'}(u',v')$ and 
$G^{\Gamma }_{r'}(u',v') = G^{\Lambda }_{r'}(u',v')$ in case (c). % are each preserved under the swap of $a_i$ and $a_{i+1}$. %  in this case. 

%\commentph{%It would probably be easier for readers to follow this proof if we start with the easy cases that are in currently in the last paragraph of the proof, explaining why there are only a few cases that are not obvious.  Then turn to the  subtle cases, with readers  would now be more primed to understand what these cases are saying/doing. 
% For the easy cases, we could start by explaining why  unless $u'$ covers $u$ and both $a_i$ and $a_{i+1}$ are in $[u,v]$, we do not have both $a_i$ and $a_{i+1}$ as elements of 
%$F_{r'}(u',v')\cup G_{r'}(u',v')$, and that this implies that the swap cannot have any impact on either $F_{r'}(u',v')$ or on $G_{r'}(u',v')$.}

%\commentplh{Given the referee comment, perhaps we should point out why we cannot have $v=a_i$ or $v=a_{i+1}$?}

%\commentplh{In the next paragraph, I added an explanation that we have $a_{i+1}$ contained in $[u,a']$, since readers could be confused about whether we have indeed
%verified the hypotheses needed to be sure that $a_i$ cannot be the earliest atom of $[u,a']$.  I wonder if confusion about this part is part of what bothered the referee, in addition to the fact that the hypothesis had somehow gotten omitted that we were only considering $v$ with both $a_i$ and $a_{i+1}$ in $[u,v]$ in our claim about neither coming first.} 
Now we turn to % the  two  most subtle cases, namely  
the cases  (a) and (b),  the cases  in which $u'=a_i$ and  $u'=a_{i+1}$, respectively.
% \commentgs{We were very explicit about what case c was above before starting it, despite just having laid out the cases. It seems incongruous that we do not remind the reader of what cases a or b are here.}
%First we prove set containments in each of these cases.  Then we prove that these containments are actually equalities of sets, doing so in a way
%that causes these two cases to be intertwined with each other.  
 %) with $u'=a_i$ and with $u'=a_{i+1}$, respectively.  
%The task of proving the containments  
We handle these two cases simultaneously because the  latter parts  of the arguments for  these two cases  are intrinsically   intertwined with each other.   
%\commentgs{these cases don't seem intertwined here, rather they seem to be handled simultanteously.}  \commentph{It's very subtle, but they are intertwined and cannot be done in parallel.}
  If an interval $[u,v']_r$ does not contain both of the atoms $a_i$ and $a_{i+1}$, then swapping $a_i$ and $a_{i+1}$ has no impact on the ordering of the atoms of $[u,v']_r$, yielding the desired equalities in this case.  For the remainder of cases (a) and (b), we therefore may assume the rooted  interval $[u,v']_r$ under consideration  contains both $a_i$ and $a_{i+1}$.  
%First we  observe that  every element contained  in   $F^{\Gamma } (a_i,v')_{r\cup a_i}$   stays  in $F^{\Lambda }(a_i,v')_{r\cup a_i}$ after the 
%swap of $a_i$ and $a_{i+1}$; this claim  follows from the fact
Since swapping $a_i$ and $a_{i+1}$ 
moves  $a_i$ to after $a_{i+1}$  while  otherwise preserving   the chain-atom ordering, % of the atoms of $[u,v']_r$,  
every  atom of $[u,v']_r$  that came earlier than $a_i$ in $[u,v']_r$  before the swap still comes earlier than $a_i$  after the 
swap.  In other words, we deduce  %This implies 
the set containment 
$F^{\Gamma }_{r\cup a_i}(a_i,v') \subseteq F^{\Lambda }_{r\cup a_i}(a_i,v')$.  % producing $\Lambda $.   
%\commentplh{Maybe start by saying ``We begin by justifying that 
%$F_{r\cup a_i}(a_i,v')$ calculated before the swap is contained in $F_{r\cup a_i}(a_i,v')$ after the swap and that 
Next we show %the   more subtle set containment 
$F^{\Gamma }_{r\cup a_{i+1}}(a_{i+1},v') \subseteq F^{\Lambda }_{r\cup a_{i+1}}(a_{i+1},v')$  %splitting this into cases with 
by  % calculated after swapping $a_i$ and $a_{i+1}$.
%\commentplh{Maybe do $a_i$ first since that
%is easier, as a warm-up for $a_{i+1}$, rather than leaving $a_i$ to the reader after doing the tougher case.}
%We start with  the case of $u'=a_{i+1}$. 
% \commentplh{The next part did need a little work, or at least to be explained more completely. 
% We have $a'$ above both $a_{i+1}$ and above an earlier element which might be 
%$a_i$.  If it is $a_i$, then neither $a_i$ nor $a_{i+1}$ is the earliest atom of the interval $[u,a']$....}
proving   that  $a' \in F^{\Gamma }_{r\cup a_{i+1}}(a_{i+1},v')$ implies 
$a'\in F^{\Lambda }_{r\cup a_{i+1}}(a_{i+1},v')$.     
%Suppose $a'\in F^{\Gamma }(a_{i+1},v')_{r\cup a_{i+1}}$.  
% \commentgs{why is such an $a'$ guaranteed to exist?} \commentph{We do not need it to exist -- we are simply showing that for any such $a'$ that might exist, it must then belong also to the other set.  I will think about how to rephrase.  Perhaps changing "Consider" to "Suppose"...}
  % \sout{,  % given by $\Gamma $. 
%The $a'\in F(a_{i+1},v')_{r\cup a_{i+1}}$ assumption gives us
%which means  that } \gs{
%This implication is vacuous unless  
 %$F^{\Gamma }_{r\cup a_{i+1}}(a_{i+1},v') \neq \emptyset $,  so let us assume
 Suppose there exists some   $a'\in F^{\Gamma }_{r\cup a_{i+1}}(a_{i+1},v')$.
 Then  by definition of $F^{\Lambda }_{r\cup a_{i+1}}(a_{i+1},v')$ the rooted interval   %in which case   % for $r^-$ obtained from $r$ by eliminating $a_{i+1}$ 
%it follows that 
$[u, a']_r$ has an earlier atom than $a_{i+1}$ with respect to $\Gamma|_{[u,a']_r}$.  % \commentph{Explain next sentence better.}
% \sout{Since $u\lessdot a_{i+1}$ and $a_{i+1}\lessdot a'$,  % Our set-up ensures that 
% $a_{i+1}$ must be   an atom of $[u,a']_r$.  }
Denote by $a_1$  the earliest atom of $[u,a']_{r}$ with respect to    $\Gamma $. 
 %We now consider separately the possibilities that (1)  $a_i\in [u,a']_r$ and (2)  $a_i\not\in [u,a']_r$.  
%.\commentplh{Be careful in the next sentence whether we have 
%$a_i,a_{i+1}\in [u,a']_r$, since we just added a hypothesis about that.  It seems like we do not have $a_{i+1}\in [u,a']$, 
%since we assume $a'\in F_{r\cup a_{i+1}}(a_{i+1},v')$.}
If $a_i\in [u,a']_{r}$, then  both $a_i$ and $a_{i+1}$ are atoms of $[u,a']_r$;     %hypothesis
%our hypotheses in the statement of the lemma  about $a_i,a_{i+1}$ applied to the interval $[u,a']_{r}$ % ensure
% that $a_1\not\in  \{ a_i , a_{i+1}\} $. \commentgs{this previous sentence is difficult to process and seems vague (which hypothesis? Can we avoid making the reader go back and read through all the hypotheses again?). Can we simplify to something like: ``s
 since neither $a_i$ nor $a_{i+1}$ is allowed to be the  first atom  of any rooted interval $[u,w]_r$ containing both of them,
 %in $[u, a']_r$ due to the hypotheses of our lemma, 
 $a_1\not\in  \{ a_i , a_{i+1}\} $ in this case.  % \sout{,  % implying} \gs{. Thus} 
 %Thus,  $a'\in F^{\Lambda }(a_{i+1},v')_{r\cup a_{i+1}}$ in case (1), 
%\commentph{the first part of your next comment seems mathematically incorrect since $a' $ does not have to cover $a_1$, as is most evident when you think about the non-graded case.}  \commentgs{Apologies, I was sloppy. I should have written: $a' > a_1$... (likewise below)} 
 %\gs{,
%since  $a_1$ comes before $a_{i+1}$ in $\Lambda$  in this case and $a_1<a'$.  % (and that  $a_1 \ne a_{i+1}$). 
 On the other hand, for $a_i\not\in [u,a']_{r}$ we  must  have 
$a_1\ne a_i$ since $a_1\in [u,a']_{r}$; we may also deduce 
$a_1\ne a_{i+1}$ in this case 
 from our assumption  that $a'\in F^{\Gamma }_{r\cup a_{i+1}}(a_{i+1},v')$. % \sout{ which in particular says} \gs{. 
%This means that $a_{i+1}$ is not the earliest atom of $[u,a']_r$.  
%\commentph{This next proposed revision misleads the reader to think we are still in only the second of the two cases.   So I kept the old version}
Thus, regardless of whether $a_i \in [u,a']_r$ or not, % in either  case \gs{(a) or (b){?}} \commentph{It's not (a) or (b); it's the two scenarios we just discussed.  I'll think about how to clarify.} 
swapping 
%\sout{Swapping} 
the order of  $a_i$ and $a_{i+1}$ in  $[u,\hat{1}]_r$ % in  
%our chain-atom ordering \commentgs{can we take out ``our chain-atom ordering"? (see comment previous page regarding this)} 
does not change the fact that  $a_1$ is  the earliest atom of 
$[u,a']_{r}$.  % where   $a_1\not\in \{ a_i,a_{i+1}\} $,  %\commentph{your next comment is not mathematically correct} \commentgs{see above}  \gs{
%\commentph{Omit next as obvious? nor the fact that $a' > a_1$}.  
This means that after the swap  $a'$ % \commentph{again, this next comment of yours is not correct.}  
%\commentgs{see above} \gs{
is still above the  atom  $a_1$ which still comes earlier than $a_{i+1}$, 
implying  $a' \in F^{\Lambda }_{r\cup a_{i+1}}(a_{i+1},v')$.  
This shows    $F^{\Gamma }_{r\cup a_{i+1}}(a_{i+1},v') \subseteq 
F^{\Lambda }_{r\cup a_{i+1}}(a_{i+1},v')$. % calculated after the swap.   
%Before completing this case, we need an analogous statement for  the case of 
%$u'=a_i$.

%Using exactly the same reasoning for the case of $u'=a_{i+1}$, 

%\commentplh{In the next paragraph is where it is important we are working with chain-atom orderings rather than GRAOs, since we do not know that after the swap we still have a GRAO.  Otherwise, it might look like we are assuming what we are trying to prove.} 
 % careful in what follows that we might not have a GRAO after the first  swap, so be careful whether we need that.}

%\commentplh{This next sentence might be confusing to some readers, as written.}  
%Next we  use the pair of  containments we have just justified.  
If we swap $a_i$ with $a_{i+1}$ twice in succession, the second swap takes the chain-atom ordering $\Lambda $ and produces from it the original
chain-atom ordering $\Gamma $.  % In $\Lambda$, 
The atom $a_{i+1}$ comes earlier than $a_i$ in $[u, \hat{1}]_r$ with respect to $\Lambda $,  
but   $a_{i+1}$ and $a_i$ are still consecutive atoms of $[u,\hat{1}]_r$ with respect to $\Lambda $; 
also observe that $a_{i+1}$ and $a_{i}$   still satisfy our requirement with respect to $\Lambda $ 
that any rooted interval  $[u,v]_r$  containing 
both $a_{i+1}$ and $a_{i}$ must  have neither $a_{i+1}$ nor $a_{i}$ as its first atom. % with respect to $\Lambda $. 
%\sout{This is where the cases become intertwined, since the roles of $a_i$ and $a_{i+1}$ are reversed when we do the second swap.
% \commentgs{This next sentence seems too long, but I haven't found an obvious way to split/shorten} \commentph{I found a way to add a semi-colon to help here.} 
% That is, $a_{i+1}$ comes earlier than $a_i$ with respect to $\Lambda $, but   $a_{i+1}$ and $a_i$ are still consecutive atoms of $[u,\hat{1}]_r$ with respect to $\Lambda $;  also observe that $a_{i+1}$ and $a_{i}$   still satisfy our requirement that any rooted interval  $[u,v]_r$ that  contains both $a_{i+1}$ and $a_{i}$ with our fixed choice of $u$ and $r$ must  have neither $a_{i+1}$ nor $a_{i}$ as its first atom with respect to $\Lambda $. }
% \commentgs{Can we reword this next sentence for clarity? I prefer what was here originally, something like ``We may apply the argument from the prior paragraph twice in succession, first swapping $a_i$ and $a_{i+1}$ and then swapping them back to their original order, to obtain the set containments..."}    \commentph{I worked on wording the  next sentence better, but I think the version from a year ago hides the intertwining that is actually important to understanding the proof, so am hesitant  to go back to it.}  
Thus, we may  apply the argument from the prior paragraph to  the chain-atom ordering $\Lambda $ using the  swap of $a_{i+1}$ and $a_{i}$  in the rooted interval $[u,\hat{1}]_r$ which outputs the chain-atom ordering   $\Gamma $.   
%Since the same argument may  be applied twice in succession, first swapping $a_i$ and $a_{i+1}$ and then swapping them back to their original order, 
This yields  the set containments    %from the containments we have just obtained applied to the second of these two swaps   that 
$F^{\Lambda }_{r\cup a_{i+1}}(a_{i+1},v') \subseteq %$ % calculated after the first  swap
 %is also a subset of $
 F^{\Gamma }_{r\cup a_{i+1}}(a_{i+1},v')$  % after the second swap.
%   In other words,
%$F(a_{i+1},v')_{r\cup a_{i+1}}$ calculated after the first swap is a subset of $F(a_{i+1},v')_{r\cup a_{i+1}}$ calculated with respect to $\Gamma $,   %before the swap, 
and % likewise 
$F^{\Lambda }_{r\cup a_i}(a_i,v') %$ calculated after swapping $a_i$ and $a_{i+1}$ is a subset of $
\subseteq F^{\Gamma }_{r\cup a_i}(a_i,v')$. % calculated with respect to $\Gamma $.  
We may deduce $F^{\Gamma }_{r\cup a_i}(a_i,v') = F^{\Lambda }_{r\cup a_i}(a_i,v')$ 
from the series  of set containments $$F^{\Lambda }_{r\cup a_i}(a_i,v') \subseteq F^{\Gamma }_{r\cup a_i}(a_i,v') \subseteq F^{\Lambda }_{r\cup a_i}(a_i,v').$$
%\commentph{I disagree with crossing  out part of the next sentence, since our goal is to make the referee's/reader's  life as easy as possible so they will not give up from exhaustion, and since   it is really subtle what's happening here and how the cases intertwine.} \commentgs{I will defer to you but will mention that these inequalities are quite overwhelming symbolically. As this series of inequalities is identical to the previous except for replacing $a_i$ with $a_{i+1}$, it is personally more exhausting for me to wade through another symbol-laden line than it is to see the similarities between the previous series of inequalities}
Likewise  % \gs{Similarly} 
we deduce $F^{\Gamma }_{r\cup a_{i+1}}(a_{i+1},v') = F^{\Lambda }_{r\cup a_{i+1}}(a_{i+1},v')$  from the series of  set containments
$F^{\Lambda }_{r\cup a_{i+1}}(a_{i+1},v') \subseteq F^{\Gamma }_{r\cup a_{i+1}}(a_{i+1},v') \subseteq F^{\Lambda }_{r\cup a_{i+1}}(a_{i+1},v').$ 
%Combining pairs of  set containments we have proven  allows us to conclude that %Thus, we may conclude that 
%these containments of sets are actually set equalities. % of sets.
% with $a_i,a_{i+1}$ switched back to their original order, this containment of 
%sets is actually an equality of sets.  
%\commentph{Perhaps shorten the remainder of the proof by using notation more?} 
% In other words, 
%we have shown  that swapping $a_i$ and $a_{i+1}$ preserves the sets $F^{\Gamma }(a_{i+1},v')_{r\cup a_{i+1}}$ and $F^{\Gamma }(a_i,v')_{r\cup a_i}$.  
These set equalities  imply  the desired  set equalities $G^{\Gamma }_{r\cup a_i}(a_i,v')  = G^{\Lambda }_{r\cup a_i}(a_i,v')$ and 
$G^{\Gamma }_{r\cup a_{i+1}}(a_{i+1},v') = G^{\Lambda }_{r\cup a_{i+1}}(a_{i+1},v')$ of the complementary sets.  
% that  swapping $a_i$ and 
%$a_{i+1}$  also   preserves  the complementary sets 
%$G^{\Gamma }(a_{i+1},v')_{r\cup a_{i+1}}$ and $G^{\Gamma }(a_i,v')_{r\cup a_i}$. %, completing the proof. 
%
%
%
%??? move claim (just above) and this proof of claim below to much earlier in the proof?
%
\end{proof}

Before proving the main result of the remainder of this section, Theorem \ref{switch-thm}, we prove a pair of lemmas  that will provide two of the main  pieces of the proof of Theorem \ref{switch-thm}.

%\commentplh{The next two lemmas were  3.14, 3.15 in referee report.}

\begin{lem}\label{ib-preserved}
Let $\Gamma $ be  a GRAO  for a finite bounded poset $P$.  Suppose a pair of consecutive atoms $a_i,a_{i+1}$ in  $[u,\hat{1}]_r$ with chain-atom ordering
$\Gamma|_{[u,\hat{1}]_r}$ has 
the property for % that the earliest atom  in 
each rooted interval   $[u,w]_r$ containing both   $a_i$ and $a_{i+1}$ 
that neither $a_i$ nor $a_{i+1}$ is  the earliest atom of $[u,w]_r$ with respect to $\Gamma|_{[u,w]_r}$.  Then the chain-atom ordering
$\Lambda $ on $P$  
obtained from $\Gamma $ by swapping the order of $a_i$ and $a_{i+1}$ in $[u,\hat{1}]_r$ satisfies condition (i)(b) of a GRAO.  
\end{lem}

\begin{proof}
%\commentgs{I find these intros to the proofs very helpful.} \commentph{Good, but I didn't like your proposed stylistic edit here. The goal is not formality but to make it as palatable as possible to read this.}
We break the proof into  the following three cases based on 
the nature of  the element $u\in P$  that is covered by   $a_i$ and $a_{i+1}$: 
\begin{enumerate}
\item
%\sout{, using the following cases}\commentph{I disagree with this strikeout, because this is not grammatically correct with the strike-out and is confusing}: 
  $u = \hat{0}$
 \item
  $u$ is   an atom of $P$
 \item
   $u$  
 is  some other element  of $P$.
 \end{enumerate}
%Let $a_j$ denote an arbitrary atom of $P$. 
% Our task will be to confirm the requirement of condition (i)(b) from the definition of GRAO  
%for every $x,v\in P$ such that 
%$a_j\lessdot x\lessdot v$ 
%for the chain-atom ordering $\Lambda $.  \commentph{I just  added the next sentence to help readers.} 
%To verify that $\Lambda $ satisfies condition (i)(b) of a GRAO, i
%For each choice of $u$, % \commentph{and each choice of $a_i,a_{i+1}$ covering $u$}, 
To show that $\Lambda $ satisfies condition (i)(b), it suffices to prove  for each $u\in P$  the following property  for every atom $a_j\in P$ and every  
$v\in P$ such that there exists an $x\in P$ with  $a_j\lessdot x\lessdot v$: 
%  \sout{for some $x\in P$}:  
%\commentph{The previous phrase that I struck out was ambiguous where we fix an $x$ or consider all possible $x$, so could easily be misinterpreted}
%we must show 
either the first atom of $[a_j,v]_{\hat{0}\lessdot a_j}$ with respect to $\Lambda $
 is above an earlier atom of $P$ than $a_j$ or  no atom of $[a_j,v]_{\hat{0}\lessdot a_j}$ is above an earlier atom than $a_j$.  %\commentph{
% We will handle case (3) first, but when we get to 
 %We will  call what what must be checked  for a particular choice of such  $a_j$ and $v$ %  within $P$ 
We sometimes speak below of checking ``condition (i)(b)  for the rooted  interval $[a_j,v]_{\hat{0}\lessdot a_j}$'', by which we mean that we are checking the condition above just for that  fixed choice of $a_j$ and $v$.   In the remainder of the proof, we suppress the notation indicating our choice of root for intervals $[a_j,v]$ where $a_j$ is an atom  so as to simplify notation, since it is  clear from context for such  intervals that the root must be $\hat{0}\lessdot a_j$.  
We start with the easiest case, namely case (3) where we have   $u \ne \hat{0}$ and $u$  also not an atom of $P$.
%\commentph{It is disjarring notation to use $R$ for root.  I suggest $r'$ to avoid $r$.}
%\commentph{This paragraph was all wrong, so I just fixed it. } 
% The thing to check to verify condition (i)(b) is a statement 
%regarding atoms of $P$, not for general rooted intervals}.
%Finally consider the case where $u$ is neither $\hat{0} $ nor % at  $u\ne\hat{0}$ where $u$ is also not
% an atom of $P$. 
In this case, 
swapping the order of the  atoms  $a_i$ and $a_{i+1}$ in  the rooted interval $[u,\hat{1}]_r$  cannot impact the ordering of the atoms of $P$.  
% $[u,w]_R$ 
%\commentph{It is misleading to speak or roots other than $r$, so I changed that.  We only need to consider $r$ here.} % 
%for any choice of root $R$, 
The swap also cannot  impact the ordering  on the atoms of $[a_j,v]$ for any atom $a_j$ of $P$ and any $v\in P$.
% such that there exists $x\in P$ with 
%$a_j\lessdot x\lessdot v$.  
% which atoms of $[u,w]_R$ are greater than an earlier atom than $u$ in $[u^-,w]_{R^-}$ where $u^-$ is the element covered by $u$ in $R$ and where $R^-$ is the root obtained from $R$ by removing $u$.  
%The preservation of these orderings implies 
%that this swap of  consecutive atoms   $a_i,a_{i+1}$ of  $[u, \hat{1}]_r$ % in this case % where $u$  is neither $\hat{0}$ nor an atom   of $P$ 
% cannot % in any way
 % impact whether our chain-atom ordering for  $P$  \commentgs{which one?} \commentph{we aren't talking about a single chain-atom ordering here but rather the relationship between two different ones, namely $\Gamma $ and $\Lambda $.   This is really a general statement, so I did not think we needed to say which one, and it really does not make sense as written to do so.}  satisfies condition (i)(b).  %\commentgs{this is still confusing to me. What about ``
%  The preservation of these orderings implies 
%that  %\sout{in this case,} 
Thus, swapping the order  of   $a_i,a_{i+1}$ in   $[u, \hat{1}]_r$  while otherwise leaving a chain-atom ordering unchanged cannot
  cannot  whether (i)(b) is satisfied. % for our  given chain-atom ordering for  $P$.
  % (would then need to take out ``in this case" in the next sentence. } 
%  \commentph{It's really important to keep the phrase "in this case" in the next sentence, since otherwise readers are likely to think we're claiming to have proven more than we have.}
    % Thus, the fact that 
Since $P$   satisfies condition (i)(b) with respect to $\Gamma $,
$P$   therefore also satisfies condition (i)(b) with respect to $\Lambda $.  This completes case (3).

%\commentplh{In this next paragraph,  it was  confusing that we had two cases instead of three, so I edited to clear up how we split into cases.}
%\commentplh{It  still might be confusing how this fits our setup with regard to what we are supposed to be checking.}
 Next we handle   case (2),  i.e., the case  where  $u$  is an atom of $P$.   %We subdivide this case into  two subcases.  
 First we check condition (i)(b) for   %any  atom $a_j\in P $ with $a_j\ne u$ and  
 each  interval $[a_j,v]$  %within $P$ 
  %\commentgs{using both $a$ and $a_j$ here} 
 with $a_j\ne u $ where $a_j$ is  an atom of $P$  %such that $a_j\ne u$ where  
 and  $v$ satisfies 
    %} \sout{such that  }
 $a_j \lessdot x \lessdot v$ for some $x\in P$; then we will separately  handle the intervals with $a_j=u$.
 %  \commentgs{Why are we calling these atoms $a$ and not $a_j$?}
  Since 
 swapping the order of  $a_i$ and $a_{i+1}$ in  $[u,\hat{1}]_{\hat{0}\lessdot u}$ has no impact on the order of the atoms of $P$ or  on the order of the atoms of 
 $[a_j, v]$ for  $a_j \neq u$, it follows from the fact that $[a_j,v]$ satisfied condition (i)(b)   before the swap that it also satisfies condition (i)(b)  after the swap.
 Next  we  prove condition (i)(b) % in case (2) 
  for those intervals $[a_j,v]$ with  $a_j=u$    
 %$  for  each interval   $[u,v]_{\hat{0}\lessdot u}$  
 such that  
 there exists $x$ with $u\lessdot x\lessdot v$, 
 splitting this in two parts based  on whether or not
 $a_i$ and $a_{i+1}$ are both elements of $[u,v]_{\hat{0}\lessdot u}$. 
%  \sout{Thus, we now focus on intervals of the form $[u,v]_{\hat{0}\lessdot u}$ having some $x$ with  $u\lessdot x \lessdot v$.  }
 %, broken into two cases depending whether or not $a_i,a_{i+1}$ are both elements of $[u,v]_{\hat{0}\lessdot u}$.  F
 First consider % the subcase given by 
 any  such  interval   $[u,v]$  %with  our fixed $u$ 
that  contains at most one of $a_i$ and $a_{i+1}$.  % where $u\lessdot x\lessdot v$ for some $x\in P$. 
% and also  contains some  $x\in P$ with  $u\lessdot x \lessdot w$.  
%Since $u <  a_i$ and $u <  a_{i+1}$ by our set-up,  %we must have 
%$a_i\not\le v$ or $a_{i+1}\not\le v$. % in this case.  
%Therefore  
Notice   that  $[\hat{0},v]$  contains at most one of $a_i,a_{i+1}$  by virtue of our  having  
$u\lessdot a_i$ and $u \lessdot a_{i+1}$ with $a_i,a_{i+1}$ not both in $[u,v]$. %, since  $a_i$ and $a_{i+1}$ both cover $u$.  
%\commentph{We should be speaking about the GRAO on all of $P$ in what follows.}  
But then   $\Gamma|_{[\hat{0},v]} $  has the property  that
swapping the order of $a_i,a_{i+1}$ in $[u,\hat{1}]$ % $\Gamma $ 
 does not change the chain-atom ordering  $\Gamma|_{[\hat{0},v]}$ which therefore is  still 
a GRAO on $[\hat{0},v]$ after the swap. 
 %\commentgs{could use a bit more explanation here, this is a comparatively large leap} 
 % In particular, t
Therefore, %  This shows that   %GRAO structure on $[\hat{0},v]$ after the swap  
 condition (i)(b) of a GRAO holds  for  $\Lambda|_{[\hat{0},v]}$, implying  that  %such 
%interval 
$[u,v]$   with at most one of $a_i,a_{i+1}$ in $[u,v]$ % having some $x$ with $u\lessdot x \lessdot v$  
 satisfies condition (i)(b)  with respect to $\Lambda $.   
%\commentgs{I don't understand how $\Lambda$ satisfying (ii) of the defn of GRAO implies condition (i)(b) is satisfied.} 
%our chain-atom ordering on $[\hat{0},w]$ % after the swap 
%still  satisfies condition (i)(b) after the swap.    \commentph{Need to explain why we get (i)(b) for all of $P$ after the swap.}
%Thus, $P$ satisfies condition (i)(b) with respect to $\Lambda $ in the case that $u$ is an atom.  
%\commentph{What about  when $u$ is an atom but $a_j$ is a different atom?}  
%
% 
%
%To complete our proof of condition (i-b-int)  for  intervals $[u,v]_{\hat{0}\lessdot u}$  falling under case (2),  %the case where $u$ is an atom of $P$, 
 Next consider  % any  the subcase that handles 
 the   intervals
$[u,v]_{\hat{0}\lessdot u}$ containing both $a_i$ and $a_{i+1}$. % where $u$ is an atom of $P$ and
%such that  there exists  
% some $x\in P $ with $u\lessdot x\lessdot v$.
%By our assumptions \commentph{I want to keep this next phrase because we can't be vague with so many cases which all have different assumptions -- it makes the job of the reader too difficult.} % \sout{
%about $a_i$ and $a_{i+1}$ in the statement of this result, \commentgs{How about something like, ``Recall that we have assumed in the statement of this result that 
Neither $a_i$ nor $a_{i+1}$ may  be the first atom with respect to $\Gamma $   of any interval to which both belong, implying neither is the first atom  in 
$[u, v]$ with respect to $\Gamma$. %" I personally find it onerous to have to flip back to the statement of the lemma mid-sentence.}  %in this case that 
%neither $a_i$ nor $a_{i+1}$ is allowed to  be the first atom of $[u,v]_{\hat{0}\lessdot u}$ with respect to $\Gamma $.  
Therefore,
 swapping the order of  $a_i$ and $a_{i+1}$ has no impact on whether  the first atom  of $[u,v]$ is greater  than  an earlier atom of $P$ 
 %(with respect to our chain-atom ordering  \commentph{this is a statement about chain-atom orderings in general, not a particular one, and that is important here} \commentgs{which?}) 
%(resp. not being greater than an earlier atom) 
 than $u$.    This swap also has no impact on whether there exists an atom of $[u,v]$ that is greater than  %\sout{(with respect to $\le_P $) }
 an earlier atom of $P$  
  %(with respect to our chain-atom ordering \commentph{again, this is not a statement about a single one}  \commentgs{which?}) 
    than $u$.  Thus, we   use that 
 condition (i)(b) holds for  each such interval 
  $[u,v]$   with respect to $\Gamma $ to deduce that it also holds  for each such interval
 with respect to  $\Lambda $.  
  % is preserved under the swap of $a_i$ and $a_{i+1}$ (in both cases).  
% Thus, the fact that property (i)(b) holds for $\Gamma $
%implies that (i)(b) also holds for $\Lambda $.
% the desired condition given  by (i)(b)  
%on the earliest atom  of  $[u,w]_{\hat{0}\lessdot u}$  holds after the swap of $a_i,a_{i+1}$ by virtue of its holding 
%before the swap. 

%\commentplh{I edited the next paragraph substantially, combining two paragraphs into one and reducing repetitiveness and unnecessary arguments, while explaining some things more fully.}
Finally, we turn to case (1),  the case in which  $u=\hat{0}$.
%, working our way through  the various types of intervals $[a_j,v]_{\hat{0}\lessdot a_j}$ for $a_j$ an atom of the entire poset $P$ and $a_j\lessdot x\lessdot v$ for some $x\in P$. 
%
 % In the $u=\hat{0}$ case,
 Here we subdivide  the task of checking condition (i)(b) for  the pertinent  rooted intervals $[a_j,v]$   
 % for all of the  rooted  intervals $[a_j,v]_{\hat{0}\lessdot a_j}$ within $P$  
% having some $x$ satisfying  $a_j\lessdot x\lessdot v$, 
%splitting up the task   
based on
% \sout{the nature of the rooted interval $[a_i,v]_{\hat{0} \lessdot a_j}$ \gs{$[a_j,v]_{\hat{0} \lessdot a_j}$?}.  Our cases are split up based on} 
whether  we have % the 
%rooted interval  
%$[a_j,v]_{\hat{0}\lessdot a_j} $  %under consideration  
%has  
%\begin{itemize}
%\item
(a)  $a_j\in \{ a_i, a_{i+1}\} $ with   $a_i, a_{i+1}\in [\hat{0},v]$, 
%\item
(b)  
$a_j\in \{ a_i ,a_{i+1}\} $ with   exactly one of   $a_i, a_{i+1}$  in $[\hat{0},v]$, or 
%\item
 (c)     $a_j\not \in \{ a_i,a_{i+1}\} $. 
%\end{itemize}
%\commentgs{The changes to the previous sentence are intended to make it as easy as possible to see the differences between the cases and to see that they cover all possibilities.}  
It is not possible to  have 
$a_j\in \{ a_i,a_{i+1} \} $ with neither $a_i$ nor $a_{i+1}$ in $[\hat{0},v]$ since % \sout{our set-up ensures}
 $a_j\in [\hat{0},v]$. %, so this covers all possible scenarios.  

 First  consider  for $u=\hat{0}$ any rooted interval $[a_j,v]$  of type (a).
 % \commentph{it's not a subcase, rather there are three things to check in order to check this case} 
  %\gs{sub}case (a), 
 % namely suppose % and %,  and suppose also 
 %$a_j\in \{ a_i,a_{i+1}\} $ with  % $a_i$ and  $a_{i+1}$ both in 
 %$a_i,a_{i+1}\in [\hat{0}, v]$.
% \sout{$a_i <_P  v$,  and $a_{i+1} <_P v$. }% in $P$. % where $a_j\lessdot x\lessdot v$ for some $x\in P$. % and $r$ is any root for $[u,v]$. 
%Having $u=\hat{0}$ implies that  $a_i$ and $a_{i+1}$ are atoms of $P$.  
Since neither % \sout{Our requirements on } 
$a_i$  %\gs{
nor $a_{i+1}$  is allowed to be the first atom with respect to $\Gamma $  in any rooted 
 interval containing both $a_i$ and $a_{i+1}$,  %} \sout{in the statement of the theorem guarantee that} 
there exists  an atom $a$  
of $[\hat{0},v]$ that comes earlier than both $a_i$ and $a_{i+1}$  with respect to  $\Gamma $.  %, by virtue of our assumption about   $a_i$ and $a_{i+1}$.  % set-up.  
%\commentph{This next sentence isn't really stated quite properly, in that we don't exactly speak of condition (ii) for pairs of elements in the definition.}  
%\commentph{I guess our jargon we introduce next might be a surprising or disjarring choice  in light of other conditions I've introduced elsewhere.}
%Let us say   that an ordered pair of atoms $(a_l,a_m)$ satisfies condition (ii) of a GRAO with respect to $\Gamma $ where $a_l$ is an earlier atom than $a_m$ in
%$\Gamma $ if for each $[\hat{0},y]$ having $a_l,a_m\in [\hat{0},y]$ there exists $z\in [\hat{0},y]$ and an atom $a_{l'}$ coming earlier than $a_m$ in $\Gamma $ with 
%$a_{l'}  < z$ and $a_m \lessdot z$. 
% \commentgs{Why do we need the previous sentence?}\commentph{Because  the next sentence will not make any sense to people without it -- because condition (ii) is a condition on a poset, not on an ordered pair.}  
Since $\Gamma $ is 
a GRAO, % it satisfies  condition (ii) from the definition of GRAO for 
the ordered  pairs  of atoms $(a,a_i)$ and  $(a,a_{i+1})$ each  satisfy condition (ii) 
%\commentph{the old version seems a lot clearer to me, so I did not follow your suggestion here}  
%from the definition 
of a  %\gs{a} 
GRAO % \gs{with respect to $\Gamma$}  %\sout{
when using the chain-atom ordering $\Gamma $, in the following sense: for each $[\hat{0},y]$ having $a,a_i\in [\hat{0},y]$ (resp. $a,a_{i+1}\in [\hat{0},y]$) there exists 
$z\in [\hat{0},y]$ and an atom $a_{l'}$ coming earlier than $a_i$ (resp. $a_{i+1}$) in $\Gamma $ with $a_{l'}<z$ and $a_i\lessdot z$ (resp. $a_{i+1}\lessdot z$).  
Therefore,  there exists an atom $z$  %(resp. $b_{i+1}$)  
 of  $[a_i,v]$ (resp.   % $[a_{i+1},v]_{\hat{0}\lessdot a_i}$ \gs{
$[a_{i+1},v]$) 
 that is above an earlier atom $a_{l'}$  of $[\hat{0},v]$ than $a_i$ (resp. $a_{i+1}$)  % (resp. $a_{i+1}$)
  with respect to  $\Gamma $.  
This together with the fact  that 
$\Gamma $ is a GRAO  allows us 
to deduce that the earliest atom of $[a_i,v]$ (resp. $[a_{i+1},v]$) is above an earlier atom than $a_i$ (resp. $a_{i+1}$) with respect to $\Gamma $.  %\commentph{I think I disagree with this proposed change, since the phrase helps the reader keep straight the logic here.}
% \sout{But then %.  We may directly deduce that 
%the} The
But then the  earliest atom $b'$  of $[a_i,v]$ is still above an earlier atom than $a_i$ after swapping the order of $a_i$ and $a_{i+1}$ 
since every atom that came before $a_i$ with respect to 
$\Gamma $  %the GRAO before the swap 
still comes before $a_i$ after the swap. % (due to the swap moving $a_i$ later).  %This confirms % \commentph{next notation?} 
% (IB) in the case that $a_j=a_i$.  % \commentgs{This makes it seem like we somewhere assumed $a_j=a_i$. I think we mean something more like "t
 This confirms condition  (i)(b)  for  $[a_j,v]$ when $a_j=a_i$. % \commentgs{When $a_j=a_{i+1}$, we must also consider...} 
%The case of $a_j=a_{i+1}$ is similar,  % we use the same argument along with  one additional thing we need to check in that case.
% rule out in that  case % , simply ruling out just as before 
%with one further thing we need  to rule out  in the $a_j = a_{i+1}$  case.  
%\commentph{still cleaning up the rest of this paragraph.} 
Next consider the  possibility that  $a_j=a_{i+1}$.  We have already  shown above 
that the earliest atom $b'$  of $[a_{i+1},v]$ is above an earlier atom than $a_{i+1}$ with respect to $\Gamma $.  
%When $a_j=a_{i+1}$, w
What remains  is to rule out 
the possibility  that $a_i$ is the only earlier  atom than $a_{i+1}$ with respect to $\Gamma $ that is below $b'$. % of $[a_{i+1},v]_{\hat{0}\lessdot a_{i+1}}$. 
 But this would imply $a_i<b'$ and $a_{i+1}<b'$,  which by our hypotheses implies 
that %} \sout{implying that} %which ensures  that  %To rule  this out,  we use the fact that
neither $a_i$ nor $a_{i+1}$ is % allowed to  be 
the earliest atom of $[\hat{0},b']$. 
This  gives a contradiction to our assumption that 
$a_i$ was the only earlier atom than $a_{i+1}$ below $b'$, completing our treatment of intervals of type  (a). % for $b'$ satisfying $a_i<b'$ and $a_{i+1}<b'$.
%This completes the part of the  $u=\hat{0}$  case where we also have  $a_j=a_i$ or $a_j=a_{i+1}$ with both $a_i$ and $a_{i+1}$ contained in $[\hat{0},v]$.  

%\commentph{Maybe be more clear below what is $v$?}
%\commentplh{We may want to make it more clear how the cases fit together.}
%\commentph{This  next proposed revision of yours is problematic since it isn't a subcase.  I was trying to find a way to help readers keep track of the proof structure here in spite of it not being a subcase.  So I kept my version.}   %\sout{
Continuing  the  $u=\hat{0}$ case, next  % \gs{Next} 
we handle  the intervals of type % \gs{sub}case
 (b). % \gs{of the case where $u=\hat{0}$}.  
That is,  we
consider intervals  $[a_j,v]$   with % the further requirements that
%Next consider the case with  $u=\hat{0}$  and with 
$a_j = a_i$  % (resp. $a_j=a_{i+1}$) 
but $a_{i+1}\not\le v$ (resp.  $a_j=a_{i+1}$ but $a_i\not\le v$) where $a_j\lessdot x\lessdot v$ for some $x\in P$.
%Let $a_j = a_{i+1}$ (resp. $a_j = a_i$). 
 %Starting with $\Gamma$ and  s
 In each of these two  cases, swapping the order of the atoms  $a_i, a_{i+1}$  of   $P$ %, this  swap  
 % \sout{from what it is in  %  our chain-atom ordering
%  $\Gamma $}
%for $P$ 
leaves  $\Gamma|_{[\hat{0},v]}$ unchanged.  Thus,   we may use the fact that condition (i)(b) holds for $\Gamma $ to deduce condition (i)(b) 
for intervals of  type (b)    with respect to $\Lambda $.
%  follows from the fact that each  
% $[a_j,v]_{\hat{0}\lessdot a_j} $ % $a_j\lessdot x\lessdot v$ 
%satisfied   %  the requirement of 
%condition (i)(b) within $P$ 
%with respect to   $\Gamma $ by virtue of  $\Gamma $ being  a GRAO. %, implying  $\Lambda $ satisfies condition (i)(b) from the definition of GRAO.
%to $[\hat{0},v]$ also satisfies condition (1B).  % after the swap.  

%\commentph{Again it's not a subcase, so the suggestion doesn't work.}  \gs{We complete the $u=\hat{0}$ case by considering subcase (c).} % \sout{
We now  complete   the $u=\hat{0}$ case by  handling %inally, consider  %for  $u=\hat{0}$ 
the intervals  % falling under case
of type  (c).   That is, consider 
 the intervals $[a_j,v]$ where  $a_j\not\in \{ a_i,a_{i+1}\}$ such that  there exists $x\in P$ with 
 $a_j\lessdot x\lessdot v$. % for some $x\in P$ 
%with $a_j\not\in \{ a_i,a_{i+1} \} $. 
 %The key observation for such intervals  is that
 Since $a_i$ and $a_{i+1}$ are consecutive atoms in $\Gamma $,  $a_j$ must either come  before both $a_i$ and $a_{i+1}$ in $\Gamma $ or come 
after both $a_i$ and $a_{i+1}$ in 
$\Gamma $.  In either event, the following three  facts may  easily be  checked:  
 that $\Gamma $ satisfies condition (i)(b) due to $\Gamma $ 
being a GRAO, that swapping 
$a_i$ and $a_{i+1}$ preserves $\Gamma|_{[a_j,v]} $,   % \commentgs{why is (i)(a) relevant?} 
%\commentph{the next bit sounded like the wrong reasoning here, so I just replaced with correct reasoning}
and  that the set of atoms coming earlier than $a_j$ with respect to $\Gamma $ equals the set of atoms coming earlier than $a_j$ in the chain-atom ordering obtained 
from $\Gamma $ by  swapping the order of 
$a_i$ and $a_{i+1}$. % while otherwise leaving $\Gamma $ unchanged.  
%swapping $a_i$ and $a_{i+1}$ does  not impact  which atoms of $P$ come earlier than $a_j$. 
 These facts combine to imply % together allow us 
% nor does it change 
%$\Gamma|_{[a_j,v]_{\hat{0}\lessdot a_j}}$ 
% the restriction of our chain-atom ordering to $[\hat{0},v]$
% to deduce
that  % the % swap of $a_i,a_{i+1}$  preserves    the  fact that the 
%interval 
$[a_j,v]$  satisfies condition (i)(b) after the swap, namely with respect to $\Lambda $. 
%have  its earliest atom above an earlier atom than $a_j$ if any atom of  $[a_j,v]_{\hat{0}\lessdot a_j}$ 
%is above an earlier atom than $a_j$, completing this final case.
  %this interval is above an earlier atom than $a_j$.  
%
%
\end{proof}

\begin{lem}\label{2-flippable}
Let $\Gamma $ be  a GRAO  %$\Gamma $
 for a finite bounded poset $P$.  Suppose a pair of consecutive atoms $a_i,a_{i+1}$ 
 in the induced GRAO for $[u,\hat{1}]_r$ has the property that each rooted 
interval $[u,w]_r$ containing both $a_i$ and $a_{i+1}$ has neither $a_i$ nor $a_{i+1}$ as its earliest atom with respect to % the restriction of 
$\Gamma|_{[u,w]_r}$.
% \commentgs{ (should this be $[u, w]$?)}\commentph{Either way would be okay, though we have to be careful what we have defined by this point in
%the paper}.  
Then the chain-atom ordering 
$\Lambda $ obtained from $\Gamma $ by swapping the order of $a_i$ and $a_{i+1}$ in $[u,\hat{1}]_r$ satisfies condition (ii) from the definition of GRAO. 
%\commentgs{Is the induced GRAO for $[u, \hat{1}]_r$ the same as the restriction of $\Gamma$ to $[u, \hat{1}]_r$?}\commentph{Yes, and we should be careful which we say and which has been defined by now in the paper.}
% on all of $P$. 
\end{lem}

\begin{proof}
If $u \ne \hat{0}$, swapping  the order of the atoms $a_i$ and $a_{i+1}$ of $[u,\hat{1}]_r$
 does not impact the ordering of the atoms of $P$.  Since condition  (ii) held before the swap, % and since  the ordering of the atoms of $P$ is unaffected by the swap, 
 condition  (ii) therefore  also holds after the swap. 
Thus, we may assume $u = \hat{0}$ for the remainder of the proof.
%, we must take more care due to the fact that
%First consider the case of  $u\in P$ satisfying  $u \ne  \hat{0}$. 
% \commentph{I think deletions  like the next proposed one
%will make the proof a lot harder for readers, by not 
%guiding people how it is organized.  This is another place I'm trying hard to make sure the referee does not get off track.} 
%\commentph{I want to keep the next sentence, as I think it will help readers.}
 % \sout{This case is much easier than the $u=\hat{0}$ case, but w
%  We do this  quite simple case  first to  clarify for readers  
%what it is that needs to be verified in each case.} 
%Since $\Gamma $ is a  GRAO,  $\Gamma $ satisfies condition (ii) of a GRAO.  
%Our  $u\ne \hat{0}$ assumption ensures that swapping  the order of two atoms of $[u,\hat{1}]_r$ %namely swapping $a_i$ and $a_{i+1}$, 
% does not impact the ordering of the atoms of $P$.   
 %\sout{Thus, swapping the order of the atoms  $a_i$ and $a_{i+1}$ in $[u,\hat{1}]_r$ for $u>\hat{0}$  has no 
%impact on whether condition (ii) holds for  pairs of atoms of  the entire poset $P$.  % in $[u,\hat{1}]_r$ with $u\ne \hat{0}$.  
%Thus, $P$ will still satisfy condition (ii) after the swap, % the pair of atoms of $[u,\hat{1}]_r$ will satisfy condition (ii) after the swap, 
%completing the proof of this case.} \gs{
%Since (ii) held before the swap and since  the ordering of the atoms of $P$ is unaffected by the swap, (ii) also holds after the swap. 
%
%Now we turn to the  
%case where $u=\hat{0}$.  This case will require much more care  due to the fact that
  %$a_i$ and $a_{i+1}$ are atoms of $P$ in this case. 

For each 
interval $[\hat{0},y]$ in $P$, we 
need to show the following: for any pair of atoms $a_l,a_m$ of $[\hat{0},y]$ with $a_l$ coming before $a_m$ with respect to $\Lambda $,  %\sout{that}
there exists an atom $a_k$ coming earlier than $a_m$ in $\Lambda $ and an element $z\in [\hat{0},y]$ such that  $a_k < z$ 
and $a_m\lessdot z$.  
%We find it convenient w
Within this proof we  call  what must be checked  for any fixed choice of $y\in P$ 
 ``condition (ii-int) for the interval $[\hat{0},y]$ within $P$''.  We do not call this  condition (ii)  because we wish   to 
emphasize  the fact which  will be  important to our  proof 
that this condition on $[\hat{0},y]$  is strictly easier to check than 
%checking 
%condition (ii-int)  %condition (ii) on the interval $[\hat{0},y]$ 
%is strictly easier than 
 regarding $[\hat{0},y]$ as a poset with a chain-atom ordering and checking condition (ii) from the definition of  GRAO on that poset.
First we show for each  $y\in P$ having  $a_i,a_{i+1}\in [\hat{0},y]$  that condition (ii-int)  holds for $[\hat{0},y]$  with respect to $\Lambda $.  Consider any pair
$a_l,a_m$ of atoms of $[\hat{0},y]$ where $a_l$ comes earlier than $a_m$ with respect to $\Lambda $.  
%any pair $a_j,a_k$ of atoms of $[\hat{0},y]$ 
%
%
%%%\commentph{maybe we should do any pair of atoms in this interval as this case, then split based on whether or not both $a_i$ and $a_{i+1}$ are less than $z$?} 
% any pair $a_l,a_m$ of atoms of $[\hat{0},y]$  with $a_l$ coming before $a_m$  with respect to % \commentph{Be careful here, $\Gamma $ or $\Lambda $}
%   $\Lambda  $ 
%satisfies \commentph{Name this like in the last lemma?} condition (ii) of the definition of GRAO with respect to $\Lambda $; in other words, we prove that 
%%%%  which in this case  means 
%%% \commentph{phrase the next bit more generally to correspond with new way of splitting cases?}  that % for any  $y\in P$ having $a_i<y$ and $a_{i+1}<y$ that 
%there exists an atom $a_k$ coming earlier than $a_m$ in $\Lambda $ and an element $z\in [\hat{0},y]$ such that 
%$a_m \lessdot z $ and $a_k < z $.  
%%% Our assumption that neither $a_i$ nor $a_{i+1}$ may be the first atom of any interval having them both as atoms  
%%%implies that swapping their order cannot impact which atom is the earliest atom of 
%%%$[\hat{0},y]$.  
%%%\commentph{Possible new text (which with tweaking should handle any pair of atoms both below $y$): 
Since $a_m$ is not the first atom of $[\hat{0},y]$ with respect to $\Lambda $,  we claim that it  also 
%\commentph{you had crossed out "also" next, but I think it is important to keep it to make this part understandable, so I restored it} also 
cannot be the first atom of $[\hat{0}, y]$ with respect to $\Gamma $;  we confirm this claim  by noting  that the  swap of 
$a_i, a_{i+1}$ converting  $\Lambda $ back  to $\Gamma $ does not impact which atom is first in $[\hat{0},y]$ since neither $a_i$ nor $a_{i+1}$ is allowed to  be the 
first atom of $[\hat{0},y]$ with respect to $\Gamma $ which implies they also  cannot be first with respect to $\Lambda $.
% \commentgs{this is because neither $a_i$ nor $a_{i+1}$ can be first in $[\hat{0}, y]$ and I think should be mentioned}.  
%\commentph{awkward language in the next sentence now}
Since $\Gamma $ is a GRAO, 
%\sout{ the fact that  $a_m$ is  not the first atom of $[\hat{0},y]$ with respect to $\Gamma $ implies  that} 
  $[\hat{0},y]$ satisfies condition (ii-int) with 
respect to $\Gamma $. 
% \commentph{What you say next is not correct, it is a condition on pairs of atoms, one of which comes later than the other which does ensure that the later of the two atoms is not the overall first atom.  Perhaps I need to say something more clearly.}  \commentgs{If $\Gamma$ is a GRAO, ii should be satisfied with no restriction on $a_m$, so the wording in the previous sentence is potentially problematic}. 
Since $a_m$ is not the first atom of $[\hat{0},y]$ with respect to $\Gamma $, 
% then we may use the fact that 
%$\Gamma $ is a GRAO to  guarantee that  before the swap of $a_i$ and $a_{i+1}$  yielding $\Lambda $ 
there  exists an atom in  $[\hat{0},y]$  coming earlier than $a_m$ with respect to $\Gamma $.  But then condition (ii) of a GRAO implies there exists an atom 
$a_k\in [\hat{0},y]$ and an element $z\in [\hat{0},y]$ such that $a_k$ comes earlier than $a_m$ with respect to $\Gamma $, $a_k<z$ and $a_m\lessdot z$. %, provided that   $a_m$  is not  the  first atom of $[\hat{0},y]$ with respect to $\Gamma $. 
%But swapping $a_i$ with $a_{i+1}$ cannot change which atom is the first atom of $[\hat{0},y]$, since  neither $a_i$ nor $a_{i+1}$ is allowed to be 
% our assumption about $a_i,a_{i+1}$ implies that 
%neither of these is 
%the first atom of $[\hat{0},y]$.  % Since  $a_m$ is not the first atom after the swap,  it also cannot be the first atom before the swap. 
% Thus, there is such a $z$ and such
%an $a_k$ before the swap.  
%completing the proof of our claim.  
%Thus, we must have an atom $a_k$ that comes earlier than $a_m$ with respect to $\Gamma $ and an element $z$ satisfying 
%$a_m\lessdot z$ and $a_k < z$.  
We will show that we may  use this same $z$ and  $a_k$ after the swap of $a_i$ and $a_{i+1}$  to % verify condition (II) for $[\hat{0},y]$ with respect to $\Lambda $ 
demonstrate the existence of  both an atom  $a_k'\in [\hat{0},y]$ that comes  earlier than $a_m$ with respect to $\Lambda $ and  an element %the existence  
$z' \in [\hat{0},y]$ that covers $a_m$ and is greater than $a_k'$. %;  to show we may use $a_k'=a_k $ and $z'=z$, 
What needs to be checked in order to justify letting $a_k'=a_k$ and $z'=z$  is that swapping $a_i$ and $a_{i+1}$ cannot cause   
$a_k$ to come after $a_m$ with respect to $\Lambda $. 
% The only possible way  that could happen
%$a_k$ could switch from coming before $a_m$ to coming 
%after $a_m$  is
%would be  if we had 
% unless swapping 
%the  order of $a_i$ and $a_{i+1}$ causes $a_m$ to come earlier than $a_k$ in our chain-atom ordering.   
%But  this would imply % $a_m$ could only come after $a_k$ before the swap and earlier than $a_k$ after the swap
  %if 
This could only potentially  happen if  $a_k=a_{i}$ and $a_m=a_{i+1}$. 
% In the event that we do have  $a_k=a_i$ and $a_m=a_{i+1}$, we use  the fact %our assumption \sout{in this case}
But then we could use the fact    that % \sout{
   $a_i,a_{i+1}\in [\hat{0},z]$ (due to having $a_k,a_m\in [\hat{0},z]$)  to conclude that 
   %\commentph{this should be fixed now.  I had put the word `assumption' in the
  % wrong place}  \commentgs{we don't assume $a_i, a_{i+1} \in [\hat{0}, z]$}
 %our hypothesis  
  neither  $a_i$ nor $a_{i+1}$ could be   %\sout{allowed to   be} 
 the first atom of  $[\hat{0},z]$.  This would  imply the existence of   some  atom  $a \in [\hat{0},z]$ coming  earlier than both $a_i$ and $a_{i+1}$.  But then we could   use  this element $a$ to serve as our  desired atom  $a_k$ coming earlier than $a_m$,  and we could  use the same $z$ as before the swap. %, completing this case.  
 %\commentph{Language?  This completes our proof of condition (II) for all intervals...}  
 This confirms condition (ii-int) for all intervals $[\hat{0},y]$ having  %both 
 %$a_i$ and $a_{i+1}$ in 
 $a_i,a_{i+1}\in [\hat{0},y]$. 
Finally, we  verify  % for $u=\hat{0}$ 
that condition (ii-int) holds with respect to $\Lambda $ 
for all intervals $[\hat{0},y]$   such that $a_i,a_{i+1}$ are not both elements of $[\hat{0},y]$.
%  that  condition (II) holds for such $[\hat{0},y]$ with respect to $\Lambda $.
% any pair of  atoms $a_l,a_m$  of $[\hat{0},y]$ 
%%%%with $\{ a_k,a_l\} \ne \{ a_i,a_{i+1}\} $ 
%satisfies  the requirement given by 
%condition (ii) with respect to $\Lambda $. 
Again, consider any pair of atoms   $a_l,a_m\in [\hat{0},y]$  such that $a_l$ comes earlier than $a_m$ with respect to  %the chain-atom ordering 
$\Gamma|_{[\hat{0},y]}$.   Since $\Gamma $ is a GRAO, condition (ii) of a GRAO ensures that there exists $z\in [\hat{0},y]$ and an atom $a_k$  
of $[\hat{0},y]$ which comes earlier than $a_m$ with respect to $\Lambda $ such that we also have $a_k < z$ and $a_m\lessdot z$.  
%the pair $a_l,a_m$  meets the requirement needed for  
%condition (II) for $[\hat{0},y]$  with respect to $\Gamma $. %, since $\Gamma $ is a GRAO.  
Since at most one of the elements $a_i,a_{i+1}$ is an atom of $[\hat{0},y]$ in this case, the swap  of $a_i,a_{i+1}$ cannot change the relative order of the 
atoms of $[\hat{0},y]$. %, since we are in the case  
%where at most one of the elements $a_i,a_{i+1}$ is an atom of $[\hat{0},y]$. 
%  other than $a_i,a_{i+1}$, since $a_i$ and 
%$a_{i+1}$   are consecutive atoms in 
%$\Gamma $.  \commentph{I just added the next sentence, but this needs more tweaking by looking at interval up to element covering $a_l$.}  Moreover, if $a_i,a_{i+1}$ both belong to $[\hat{0},y]$ then there is also an atom $a\not\in \{ a_i,a_{i+1}\}$ that is the earliest atom
%of $[\hat{0},y]$ ensuring that $a_i$ is not the only atom of $[\hat{0},y]$ coming earlier than $a_l$.
 %\commentph{Need to be careful e.g.  that the order of $a_i$ and $a_l$ might change if $a_l = a_{i+1}$.  This situation is covered by the same argument as in the first paragraph, but we need to say that.} 
Thus, we may use this same $a_k$ and $z$ after the swap of $a_i$ and $a_{i+1}$  to demonstrate the existence of an element $z\in [\hat{0},y]$ and an atom $a_k$ coming earlier than $a_m$  such that
$a_k < z$ and $a_m\lessdot z$.  This confirms condition (ii-int) with respect to $\Lambda $ for all intervals $[\hat{0},y]$  not having both  $a_i$ and $a_{i+1}$ as elements  of $[\hat{0},y]$. %, thereby completing the $u=\hat{0}$ case.
% meets the requirement for condition (II).
%satisfies  \commentph{Language here?}  condition (ii).
%, as one may observe by using the fact that  $a_i,a_{i+1}$ are consecutive atoms with respect to $\Gamma $.  
%This completes the  proof. % now   follows. % from Remark \ref{2-rephrased}.
\end{proof}

While the statement of the next result may seem somewhat technical, it captures in a precise and seemingly useful 
way how certain types of localized moves are guaranteed to  transform a GRAO into a new 
GRAO.  Thus, this theorem  pins down a certain type of  flexibility one has  in choosing a chain-atom ordering that will be a GRAO.   % \commentph{Add: 
We will use 
this result in our proof later in the paper that any GRAO may be transformed into an RAO. %} 
% This result   is the 
%\commentph{Not sure next two words are quite right} culmination point of a    series of lemmas that we have just proven.

% \commentplh{New version:}
\begin{thm}\label{switch-thm}
Let $\Gamma $ be  a GRAO   for a finite bounded poset $P$. % with $a_1,\dots ,a_t$ the induced GRAO for the  rooted interval $[u,\hat{1}]_r$.  
Suppose
a pair of consecutive atoms $a_i,a_{i+1}$ in the % induced 
GRAO
 %\commentgs{We use the phrase ``induced GRAO" a lot; should we make a note of what we mean by this near the definition of RAO/GRAO if we don't already?} 
  for $[u,\hat{1}]_r$ induced by $\Gamma $ 
  has  the property that each rooted interval $[u,w]_r$ containing both $a_i$ and $a_{i+1}$ has neither $a_i$ nor $a_{i+1}$ as its earliest atom with respect
to $\Gamma $.
Then the chain-atom ordering $\Lambda $  obtained from $\Gamma $  % the GRAO for $P$ 
by swapping the order of $a_i$ and $a_{i+1}$ in $[u,\hat{1}]_r$ 
  is itself a GRAO for $P$.
\end{thm}

\begin{proof}
 Our proof is by induction on the length of the longest saturated chain in $P$.  
It will suffice to show  that each of the requirements of a GRAO is preserved under swapping 
the order of two consecutive atoms   $a_i$ and $a_{i+1}$ in a rooted interval $[u,\hat{1}]_r$, provided that neither $a_i$ nor $a_{i+1}$ is the earliest atom
of any rooted  interval $[u,w]_r$ that contains both $a_i$ and $a_{i+1}$.  
 
   To show  that 
condition (i)(a) in the 
definition of GRAO still holds after the swap,  it suffices to show for each of the atoms $a$ in $P$ that % the restriction of 
$\Lambda|_{[a,\hat{1}]}$ is a GRAO.
%\commentph{This is a run-on sentence next.  I  just split it up into two sentences, but might still  need to explain better...}
 First consider the case with  $u\in [a,\hat{1}]$.  One may easily observe in this case that
 %the restriction of  
 $\Lambda|_{[a,\hat{1}]}$ is the same chain-atom ordering one obtains % may alternatively be  constructed 
by restricting $\Gamma $ to $[a,\hat{1}]$ and then performing the swap on this restricted GRAO.
 %\sout{Our inductive hypothesis then  allows us to deduce from  the fact that}
  %\gs{
  Since  $\Gamma|_{[a,\hat{1}]}$ is a GRAO, our inductive hypothesis allows us to deduce % \sout{the fact} 
  that 
 $\Lambda|_{[a,\hat{1}]}$ is also a GRAO.  %  This completes the proof of (i)(a) in this case.  %; here we are using the fact 
%  \commentph{next phrase is too vague perhaps...full result or (i)(a)?}
 % the result by induction on the length of the longest saturated chain of our poset as we explain next, using that $[a,\hat{1}]_{\hat{0}\lessdot a}$ has strictly smaller
%such length than $P$.   
%Notice 
%, which by our inductive hypothesis is 
%itself is a GRAO.
Next consider the case with   $u\not\in [a,\hat{1}]$.  Then the swap of $a_i,a_{i+1}$  
leaves  $\Gamma|_{[a,\hat{1}]}$ unchanged;  %\sout{, implying} 
since this chain-atom ordering on $[a,\hat{1}]$  is a GRAO before the swap,  it  remains a GRAO 
after the swap.  This proves  condition (i)(a) for $\Lambda $.

Since $\Gamma $ is a GRAO, we may apply Lemma \ref{ib-preserved} to deduce that $\Lambda $ also  satisfies 
%By Lemma  \ref{ib-preserved},  %we prove  % \ref{F-and-G-preserved}
  %that  
  condition (i)(b) from the definition for GRAO. % holds for $\Lambda $ because of the fact that   $\Gamma $ is a GRAO.
  Likewise we may use that $\Gamma $ is a GRAO and 
  apply  Lemma  \ref{2-flippable} to deduce that condition (ii) from the definition of GRAO holds for $\Lambda $.
  %, again by virtue of  $\Gamma $ being  a GRAO.
 \end{proof}

\section{Equivalence of admitting a generalized recursive atom ordering to admitting a  recursive atom ordering}
\label{GRAOiffRAO}

In this section  we  prove that a  finite bounded poset admits a generalized recursive atom ordering (see Definition \ref{grao})  if and only if it admits a recursive atom ordering.   
We accomplish  the  more challenging half  of  this result  constructively  by a procedure that transforms any generalized recursive atom ordering into a recursive atom ordering, a process we call ``atom reordering.'' 
 Along the way, we will develop several properties of such reorderings.

First we carry out   the  other much  easier direction  of the result.

%\commentph{I commented out the next remark, at the suggestion of the referee.}

%\begin{rmk}
%In the proof of the next result and throughout this section, 
%we think of an RAO and a GRAO each as a type of chain-atom ordering rather than just as an atom ordering, following the conventions 
%explained already in Remark \ref{chain-atom-vs-atom}.
%\end{rmk}

\begin{lem} Every recursive atom ordering is a generalized recursive atom ordering. 
\label{raothengrao}
\end{lem}

\begin{proof}  
We first  verify that  conditions (i)(a) and (i)(b) of Definition \ref{grao} hold  for any given  RAO,  doing so by  way of 
a proof by 
induction on the length  $l$ of the longest saturated chain of $P$. Any ordering of the atoms 
of a finite bounded poset whose longest saturated chain is of length 1 or 2   is  a generalized recursive atom ordering, giving the base case. 
Let $a_1, a_2, \ldots, a_t$ be an RAO  for  a finite bounded poset $P$ having  $l>2$.
By  induction, we may assume for each atom $a_j\in P$  that  the % recursive atom ordering
RAO on $[a_j, \hat{1}]$ guaranteed to exist  by condition  (i)(a)  of the definition of RAO  % Definition \ref{rao} 
is a GRAO, giving (i)(a) from the definition of GRAO. % for each atom $a_j$. %  $j=1, 2, \ldots t$. 

Now we turn to %the proof of 
%%% the task of proving 
condition   
(i)(b) of  Definition \ref{grao}.
By the definition of  
 RAO, %a recursive atom ordering, % part (ii) of Definition \ref{rao},
  each  atom in $F_{\hat{0}\lessdot a_j} (a_j)$ (see Definition \ref{F-u}) must come earlier in our given  RAO for $[a_j,\hat{1}]$  %recursive atom ordering
  than every  atom in $G_{\hat{0}\lessdot a_j}(a_j)$.
 %Let $A_w$ be the set of atoms of $[a_j, w]$, where $w$ satisfies $a_j \lessdot x \lessdot w$ for some $x$. 
 This implies that if  any  atom of $[a_j, \hat{1}]$ is greater than   $a_i$ for  some $i< j$, then % $F(a_j) \cap A_w \neq \emptyset$. In this case,
  the first atom of $[a_j, \hat{1}]$ is greater than  $a_{i'}$ for some $i' < j$.  
  Once we check that any RAO for $[a_j,\hat{1}]$ restricts to an RAO for $[a_j,w]$ for each  $w>a_j$ in $P$, we will likewise be able to deduce the following implication:
  if any atom of $[a_j,w]$ is greater than $a_i$ for some $i<j$, then the first atom of 
  $[a_j,w]$ is greater  than $a_{i'}$ for some $i'<j$. 
  % \commentgs{This next sentence is very long and therefore difficult to follow. At the very least, we should change the semi-colon to a period, but I think it would be even better to break it up even more.} 
  Now we verify the desired claim about restriction of an RAO.
  %To see that the desired claim about restriction of an RAO holds, 
  Given an RAO,  we first apply the construction of Bj\"oner and Wachs from 
  \cite{bw} that produces a CL-labeling from any RAO.  Then we note that this CL-labeling restricts to a CL-labeling for the interval $[a_j,w]$.
   Finally  we apply the construction of 
  Bj\"orner and Wachs from \cite{bw} which produces an RAO from any CL-labeling to get an RAO for $[a_j,w]$.  It is easy to 
  see that the chain-atom ordering obtained this 
  way is exactly the restriction of our given RAO to $[a_j,w]$.  
  %;  to see this, simply use the fact that 
 % any atom of $[a_j, w]$ that is in $F(a_j)$ comes before any atom of $[a_j, w]$ that is instead  in $G(a_j)$.  % with respect to the RAO. 
 Having verified this claim, we have completed the   confirmation of  %ation of  %condition (i)  %confirms  condition 
 condition (i)(b) 
  of Definition \ref{grao}.
  
 Condition (ii) in the definition of GRAO is the same as condition (ii) in the definition of RAO, hence is guaranteed to hold for any RAO.
 \end{proof}

\begin{rmk}
 It is not true that every generalized recursive atom ordering is a recursive atom ordering.  See Figure \ref{graoex} for an example illustrating this.  
 \end{rmk}
 
%\gs{
%\commentph{Next few paragraphs have been edited/expanded significantly.  This eliminated the big empty spaces  we had before the algorithm.}
Next, in Algorithm \ref{reorder-def}, we describe the atom reordering process that will allow us to transform any generalized recursive atom ordering % $\Lambda $
 into a recursive atom ordering. % that we denote by $\Lambda^{re}$. 
    This algorithm takes as its input any chain-atom ordering $\Lambda $, and it outputs a chain-atom ordering that we denote by $\Lambda^{re}$.  
One may think of  this superscript 
$re$ as shorthand for ``reordered."   %It may be  useful to keep in 
%mind %while reading the description of this  reordering process is that it 
%that t
The atom reordering process %  applied to any chain-atom ordering 
is designed to output  a chain-atom ordering that will satisfy condition (i)(b) from the definition of recursive atom ordering. Moreover, it is set up 
to do so in such a way  % we will prove 
%it will  do 
%so in such a way 
that when  %the atom reordering process 
%is 
applied to a GRAO, it 
%will do so in a way that 
preserves  %\gs{some of the} 
useful structure  that is present in a GRAO,   including  preserving 
condition  (ii) from the definition of GRAO.

 Broadly, the algorithm starts at the bottom of the poset $P$ and works its way to the top, reordering the atoms of each rooted interval in  a way that takes into account the reordering that has already occurred lower in the poset.   Each of these reordering steps moves those atoms of a rooted interval $[u,\hat{1}]_r$ that are above an earlier atom than $u$ in the rooted interval $[u^-,\hat{1}]_{r^-}$  ahead of those that are not, otherwise preserving the ordering on atoms.  
 %
 %
 %}
%\commentph{Perhaps we need a remark to make sure readers understand we are talking about a full chain-atom ordering at each step, not just how to order that set of atoms.} \commentgs{I don't think this is necessary.} 
%
%\commentph{Regarding the header for this algorithm -- it may confuse people to call this the ``reordering of $P$'' since it suggests changing the ordering on the elements of $P$.  What about something like ``Transforming a chain-atom ordering into a better behaved one'' or ``Atom reordering process''}
%
%\commentph{We may need to say more in English words what is going on below with calculating $F^{\Lambda^{pr}}(u_i)$ and with having these $\Lambda^{pr} $ components be atom orderings that get put together to form the chain-atom ordering $\Lambda^{re}$.}  
%
%\commentph{Possible  text to go before the algorithm:}
% -- text  I hope will help us come to a common understanding as to what type of object $\Lambda^{pr}$ is so that we can define $F^{\Lambda^{pr}}$ in a well-defined way): 
The  algorithm progressively  builds up a chain-atom ordering $\Lambda^{pr}$ by defining
%specify more and more of the data in the output chain-atom ordering $\Lambda^{re}$.  That is, we progressively
%define 
$\Lambda^{pr}([u,\hat{1}]_r)$ for more and more choices of $u\in P$ and of  root $r$ for $[u,\hat{1}]$.% as this ``reordering'' is done.  
The superscript $pr$ in $\Lambda^{pr}$ 
is shorthand for ``partially reordered''.
% To this end, we progressively specify $\Lambda^{pr}([u,\hat{1}]_r)$ for more and more %values of $u$ and $r$, and then o
Once  the algorithm  has made $\Lambda^{pr}$ into  an entire chain-atom ordering, it  outputs this chain-atom ordering %which it calls
and calls it
% sets $\Lambda^{re}$ equal to $\Lambda^{pr}$, letting
$\Lambda^{re}$.  Readers  may find it helpful to refer to  %the notation
  %presented in 
Definition \ref{F-u}  for an explanation of the notations  $F_r^{\Lambda^{pr}}(u,v)$ and $G_r^{\Lambda^{pr}}(u,v)$ 
as they read Algorithm \ref{reorder-def}. %as well as the notation $[u^-,\hat{1}]_{r^-}$ above.  

% be  the output 
%of the algorithm. % Here ``pr'' is shorthand for partially reordered. 
 
% \commentph{I still need to tighten up this next paragraph a bit more.}
 %In the proof of Lemma \ref{reordering-as-swaps}, we will need
 It is sometimes necessary  (e.g. in the proof of Lemma \ref{reordering-as-swaps})  
 to take a different viewpoint on this  algorithm, % at that point 
 %modifying the  data 
 keeping track of more  data   at the intermediate stages  
 in the algorithm in a way that does not impact  the output of the algorithm or the essentials of how the algorithm proceeds.  
 In this enriched  version of  the algorithm, 
$\Lambda^{pr}$ will denote   an entire  chain-atom ordering at each step of 
  the atom reordering process.    % algorithm; % e atom reordering process; 
This is accomplished by initializing   $\Lambda^{pr}$   to equal $\Lambda $ and then otherwise leaving  the algorithm unchanged.  The effect is that
progressively more and more of the values $\Lambda^{pr}([u,\hat{1}]_r)$ are re-set from what they equal in $\Lambda $ to what they will equal in 
$\Lambda^{re}$.
% data in $\Lambda^{pr}$ is modified  
%by setting new values for  $\Lambda^{pr}([u,\hat{1}]_r)$ for more and more choices of $u$ and $r$.  This setting of values for 
%$\Lambda^{pr}([u,\hat{1}]_r)$ for different choices of  $u$ and $v$ is done in exactly the manner described  in the algorithm below.  
%Giving  $\Lambda^{pr}$ this more enriched structure  at the intermediate steps  
This allows one to think of   $\Lambda $ as evolving  into $\Lambda^{re}$ over the course of the algorithm.
%; in this modified algorithm, we set  $\Lambda^{re}$ equal to $\Lambda^{pr}$ once all of the updating of $\Lambda^{pr}$ is done, and the output $\Lambda^{re}$ is the same in either case.
  Our upcoming proofs of  various  properties of the atom reordering process will  all  hold regardless of which of these two  viewpoints one takes, namely regarding $\Lambda^{pr}$ as growing or as evolving, as one may easily check by noting how  %since 
  these two versions of the algorithm  really only differ   in  terms of notation, not in substance. 
  % that in no way impacts those  proofs.  

The reason we take  the former viewpoint within %the algorithm itself
Algorithm \ref{reorder-def} itself   is so that  the sets % quantities
 $F^{\Lambda^{pr}}_r(u,v)$ and $G^{\Lambda^{pr}}_r(u,v)$ 
  %that are used both in the atom reordering process and in the proof of Lemma ~\ref{alwaysfirst} 
  will be  defined in an unambiguous way when they are used  later (e.g. within the proof of Lemma \ref{alwaysfirst}). 
     When we speak of the atom reordering process 
  transforming $\Lambda $ into $\Lambda^{re}$,  both in the introduction of the paper 
  and in the proof of Lemma \ref{reordering-as-swaps}, we are taking the latter viewpoint.  
% Keep in mind that we are building up a chain-atom ordering...}

%\commentph{A possible  alternative approach to describing  the following algorithm would be to initialize $\Lambda^{pr}$ to $\Lambda $, then keep updating 
%$\Lambda^{pr}$ as described in the algorithm and then at the end set $\Lambda^{re}$ equal to $\Lambda^{pr}$.  This would make all the intermediate objects chain-atom orderings, which would allow us to prove Lemma 4.10 without needing the first paragraph changing the definition of the algorithm for purpose of that lemma.  I don't think it would cause any problems anywhere else, since this is exactly how the original algorithm worked, and it would still allow us to speak of $F^{\Lambda^{pr}}$.    It would also make it somewhat more clear what we mean in the introduction when we say that the atom reordering process transforms a GRAO and an RAO, though I think people will figure out what is meant there anyway.  I'm not saying we need to do this, but trying to talk through how we could do it with minimal changes  
%if we wanted to do it.}

\begin{algorithm}[H]

\SetKwInput{Input}{input}
\SetKwInput{Output}{output}
%\SetKwOutput{Output}{output}

\Input{A finite bounded poset $P$ equipped with a chain-atom ordering $\Lambda$}
\Output{A chain-atom ordering of $P$, denoted $\Lambda^{re}$, also called the 
\textbf{atom reordering of $P$}.
%\sout{,  %   $\Omega^{re}$ of $P$,
% also called the \textbf{reordering of $\Lambda$}. }
}

\Begin{ 
	\begin{enumerate} 
	\item Choose any linear extension $u_0, u_1, u_2, \ldots, u_n$ of $P$\; % \commentph{next is wrong:} \sout{$\Lambda$}\;
	\item % \commentplh{Should we include this step at all?  Probably.}   
	%\commentph{
	 %\commentgs{I'm not as familiar with this language; would we be initializing $\Lambda^{re}$ or $\Lambda^{pr}$?}
	%\commentph{We should be initializing $\Lambda^{pr}$ not $\Lambda^{re}$ here.} 
	%\commentph{This is not a correct usage of the term "initialize".}
	  %It would be correct to "initialize $\Lambda^{pr}$" to equal $\Lambda $ and to make it a full chain-atom ordering, and then for each step of the reordering process to modify this chain atom-ordering so it evolves.}
	%   \commentph{I think it would be more clear and more readable to say something like: Notice that $u_0$ always equals $\hat{0}$ and begin the construction of our new chain-atom ordering by letting
	%$\Lambda^{pr}([\hat{0},\hat{1}] := \Lambda([\hat{0},\hat{1}])$\;} 
	%\sout{ Initialize % $\Lambda^{re} $ \gs{
	%$\Lambda^{pr} $ by setting} Set %} Let
	 %$\Lambda^{re}([\hat{0}, \hat{1}]) := \Lambda([\hat{0}, \hat{1}])=\Lambda([u_0, \hat{1}])$ \gs{
	 %Notice that $u_0=\hat{0}$ and s
	 Set  $\Lambda^{pr}([\hat{0}, \hat{1}]) := \Lambda([\hat{0}, \hat{1}])$, thereby determining 
	 $\Lambda^{pr}([u_0,\hat{1}])$\; %=\Lambda([u_0, \hat{1}])}$\;
	%\item  %Set $i:=1$\; %\tcc{note $u_0 = \hat{0}$}
%\end{enumerate}
\item
\For{ $i=1$ to $n$}
%\end{enumerate}
{    
\begin{enumerate}\setcounter{enumi}{2} 
	\item 
%	\commentph{You should have a number (4) in front of  ``for $i\le n$'' and not use the same numbering within this loop, as it is very misleading and not correct numbering as the loop is the fourth step of the algorithm, not the things inside the loop; it would be more appropriate  to use (4a),(4b),(4c) where you now have (4), (5), (6) as these are both part of step (4), namely the loop} 
	Choose an ordering of the roots $r_1, r_2 \ldots r_{t_i}$ % \commentph{subscript should probably be $t_i$ not $t$} 
	for $[u_i, \hat{1}]$\; 
	\item 
	%Set $j:=1$\; 
%
\For{ $ j=1$ to $t_i$ }%\commentph{should be: for $j=1$ to $t_i$, or at least that was the convention for pseudo-code when I took lots of programming classes}
	{
		\begin{enumerate}[label=\roman*.]
	\item Calculate $F^{\Lambda^{pr}}_{r_j}(u_i)$ and $G^{\Lambda^{pr}}_{r_j}(u_i)$\; % \commentph{
	%with respect to %} using
	% $\Omega^{re}([u_i^-, \hat{1}]_{r_j^-})$\;
	\item Order the elements of  $F^{\Lambda^{pr}}_{r_j}(u_i)$  in the same relative order % \commentph{ 
	as \\  in   %} given by
	 $\Lambda([u_i, \hat{1}]_{r_j})$\;
	\item Order the elements of  $G^{\Lambda^{pr}}_{r_j}(u_i)$  in the same relative order  % \commentph{
	as \\  in   %} given by 
	$\Lambda([u_i, \hat{1}]_{r_j})$\;
	\item %  \commentph{Should we change "obtain" to something like "set" if this is a full chain-atom ordering since part of it will change later} Obtain
	Determine  $\Lambda^{pr}([u_i, \hat{1}]_{r_j})$ as follows: %\sout{Modify $\Lambda^{re}$ by reordering the atoms of $[u_i,\hat{1}]_{r_j}$ as follows:}
	 put %\commentph{"put"} \\%\sout{\sout{Set   $\Omega^{re}([u_i, \hat{1}]_{r_j})$}  by taking} 
	all elements of  \\ $F^{\Lambda^{pr}}_{r_j}(u_i)$
	 in   the order from step ii, followed by all  elements \\ of  $G^{\Lambda^{pr}}_{r_j}(u_i)$  in the order from   step iii\;
	\item Increase $j$ by 1\;
	\end{enumerate}
	}
	\item Increase $i$ by 1\;
	\end{enumerate}
	}
	\end{enumerate}	
	\begin{enumerate}\setcounter{enumi}{3} 
\item  
	 Set $\Lambda^{re}$ to be the chain-atom ordering that has $\Lambda^{re}([u_i, \hat{1}]_{r_j})=\Lambda^{pr}([u_i, \hat{1}]_{r_j})$ for \\
	  every $u_i \in P$ and every choice of root $r_j$ for $[u_i, \hat{1}]$\;
	\end{enumerate}
}				
%\end{enumerate}
\begin{caption} {The atom reordering process applied to $\Lambda$ which outputs $\Lambda^{re}$.}
\label{reorder-def}
\end{caption}
\end{algorithm}

\begin{ex} 
 Figure 3 shows a poset $P$ with a GRAO on the left and a RAO on the right. The RAO on the right is obtained by applying the atom reordering process described in Algorithm \ref{reorder-def} to the GRAO on the left. In this case, the atom reordering process changes the order of the atoms above the element $a$, in particular swapping the 2nd and 3rd atoms above $a$ (highlighted in red). 
Note that this example was chosen to have the further property in both the ordering on the left and the reordering on the right
that the atom ordering for  any interval $[u, \hat{1}]_r$ is independent of  choice of root $r$.  This makes  the figures more understandable, but is  a very special case. 
\end{ex}

We next deduce several  fundamental  properties of atom reorderings.   

\begin{prop} \label{restricts-consistently}
Let $P$ be a finite bounded poset with a chain-atom ordering $\Lambda $. 
%\sout{ Let  $\Lambda^{re}$ be the chain-atom ordering on $P$ obtained by applying the 
%atom reordering process to $\Lambda $.}  
% given by a GRAO.  
Let  $\Lambda|_{[\hat{0}, v]}$ (resp.   $\Lambda^{re}|_{[\hat{0}, v]}$)  be the chain-atom ordering for $[\hat{0},v]$ obtained by restricting $\Lambda $ (resp.  $\Lambda^{re}$) 
to $[\hat{0},v]$.  Then  $\Lambda^{re}|_{[\hat{0}, v]}$ equals the 
chain-atom ordering for $[\hat{0},v]$ obtained by applying the atom reordering process to  $\Lambda|_{[\hat{0}, v]}$.
\end{prop}

\begin{proof}   
%\commentph{
%This proof  (and other  proofs  in our paper too, as I understood them) needs to use the chain-atom ordering I describe in the following possible preface
%to our proof:  
%At the step in the reordering process when we are about to reorder the atoms of 
%$[u,v]_r$ for a fixed $u$ and each choice of root $r$, denote by $\Lambda^{pr} $ the chain-atom ordering we have derived from  $\Lambda $ by applying all  
%of the reordering steps for $u_1,\dots ,u_{i-1}$ where $u=u_1$ in the linear order $u_1,u_2,\dots ,u_t$ on the elements of $P$ used in the reordering process.  
%Specifically, we will make extensive use of $F^{\Lambda^{pr}}_r(u,v)$ and $G^{\Lambda^{pr}}_r(u,v)$ in our proof below.}
%Let $\Lambda^{pr} $ be
%
%\sout{Throughout this proof, let  $F_r(u,\hat{1})$ and $G_r(u,\hat{1})$ be the sets of atoms of $[u,\hat{1}]_r$ each % sub both 
%calculated within the atom reordering process for all of $P$ at the stage just prior to reordering the atoms of $[u,\hat{1}]_r$;  let  
%$F_r(u,v)$ and $G_r(u,v)$ be the sets of atoms of $[u,v]_r$ 
%calculated within the atom reordering process on $[\hat{0},v]$ at the stage just prior to reordering the atoms of $[u,v]_r$.  }
%
Let $r=\hat{0} \lessdot t_1 \lessdot t_2 \ldots \lessdot u $ be a root for the interval $[u, v]$ in $P$. Let $r^-$ be the root obtained  by eliminating $u$ from $r$, and let $u^-$ be the highest element of $r^-$. Recall from Definition \ref{F-u} that % the set 
$F^{\Lambda^{re}}_{r}(u, v)$ refers to the set of atoms of $[u, v]$ that cover an earlier atom of $[u^-, v]_{r^-}$ than $u$ in  $\Lambda^{re}|_{[u^-, v]_{r^-}}$. Similarly, $G^{\Lambda^{re}}_{r}(u, v)$   refers to the  set of atoms of $[u, v]$ that are not in  $F^{\Lambda^{re}}_{r}(u, v)$ 

Our  main task will be to prove that $F^{\Lambda^{re}}_r(u,\hat{1}) \cap [u,v]_r = F^{\Lambda^{re}}_r(u,v) $  and that 
$G^{\Lambda^{re}}_r(u,\hat{1})\cap [u,v]_r = G^{\Lambda^{re}}_r(u,v)$.
%Once this  is accomplished,  the desired result follows easily  by observing that the atom reordering process both for $P$ and for $[\hat{0},v]$
%will preserve the order of atoms within each of the sets \old{$F_r(u,v)$ and $G_r(u,v)$} $F^{\Lambda^{re}}_r(u,v)$ and $G^{\Lambda^{re}}_r(u,v)$ and will put all elements of \old{$F_r(u,v)$ earlier than
%all elements of $G_r(u,v)$} $F^{\Lambda^{re}}_r(u,v)$ earlier than
%all elements of $G^{\Lambda^{re}}_r(u,v)$.
%
%\commentplh{This was p. 20 l-17 in submitted version, so this is where referees' suggestion should be addressed.  I made an attempt at doing that.}
%If we can prove
%\old{We claim that $F_r(u,\hat{1})\cap [u,v]_r \subseteq  F_r(u,v)$   
%and  $G_r(u,\hat{1})\cap [u,v]_r \subseteq G_r(u,v)$, and we call this Claim (I).  Notice that Claim (I)  would    imply that  
%the union of sets $(F_r(u,\hat{0})\cap [u,v]_r)\cup (G_r(u,\hat{1})\cap [u,v]_r)$ is 
%also contained in the union of sets $F_r(u,v)\cup G_r(u,v)$.  }
We claim that $F^{\Lambda^{re}}_r(u,\hat{1})\cap [u,v]_r \subseteq  F^{\Lambda^{re}}_r(u,v)$   
and  $G^{\Lambda^{re}}_r(u,\hat{1})\cap [u,v]_r \subseteq G^{\Lambda^{re}}_r(u,v)$, and we call this Claim (I).  Notice that Claim (I)  would    imply that  
the union of sets $(F^{\Lambda^{re}}_r(u,\hat{1})\cap [u,v]_r)\cup (G^{\Lambda^{re}}_r(u,\hat{1})\cap [u,v]_r)$ is 
also contained in the union of sets $F^{\Lambda^{re}}_r(u,v)\cup G^{\Lambda^{re}}_r(u,v)$.  
But this last  set containment would actually be  an equality of sets   
by virtue of both of the sets in the containment equalling  the set of all  atoms in $[u,v]$.  This set equality  would imply 
that each of the  component  set containments would  also   be a  set equality, which would complete the proof.

What remains is to prove  Claim (I).  %the 
%set containments $F_r(u,\hat{1})\cap [u,v]_r \subseteq F_r(u,v)$ and $G_r(u,\hat{1})\cap [u,v]_r \subseteq G_r(u,v)$.  
We  %now prove these two set containments, doing 
do so  by induction on the length of the longest saturated chain from 
$\hat{0}$ to $u$.  % results.  
 %\sout{Let $r^-$ be the root  obtained  by eliminating $u$ from $r$, and let $u^-$ be the highest element of $r^-$.} \commentgs{moved earlier} 
The point is that  $a\in F^{\Lambda^{re}}_r(u,\hat{1})$ (resp. $a\in G^{\Lambda^{re}}_r(u,\hat{1})$) 
 implies that  %$a$ is greater 
there exists (resp. does not exist) an atom $a'$ of $[u^-,\hat{1}]_{r^-}$ with $a'<  a$ in $P$  such that $a'$ is earlier than $u$ in $\Lambda^{re}|_{[u^-, \hat{1}]_{r^-}}$. %\sout{$a'$  is an earlier atom of $[u^-,\hat{1}]_{r^-}$  
%than $u$ in  $\Lambda^{re}$.}  
But any such 
$a'$ is also  an atom of $[u^-,v]_{r^-}$ since $u^-\lessdot  a' <  a\le  v$ in $P$.  Moreover, our inductive hypothesis ensures that $a'$ is an earlier atom than $u$ in
$[u^-,v]_{r^-}$ regardless of whether we are 
 reordering within $P$ or within  $[\hat{0},v]$.
 This % together  
is  exactly what is needed to show that  % which is enough 
%\old{$a\in F_r(u,v)$.}
 $a\in F^{\Lambda^{re}}_r(u,v)$.
On the other hand, when such $a'$ does not exist within $P$, this implies  that no % the nonexistence of 
such $a'$ exists within $[\hat{0},v]$ since  $[\hat{0},v]$ is a 
subset of $P$, yielding
the desired claim that % \old{$G_r(u,\hat{1}) \cap [u,v]_r \subseteq  G_r(u,v)$.}
  $G^{\Lambda^{re}}_r(u,\hat{1}) \cap [u,v]_r \subseteq  G^{\Lambda^{re}}_r(u,v)$. 
\end{proof}

In light of Proposition \ref{restricts-consistently}, we may henceforth speak interchangeably of
 the restriction to $[u,v]_r$ of
the  atom reordering $\Lambda^{re}$  of  % $\Lambda $ 
 a  chain-atom ordering 
$\Lambda $ of a finite bounded poset $P$ % 
%restricted to $[u, v]_r$ 
and of  
the restriction to $[u,v]_r$ of 
the  atom reordering of $\Lambda|_{[\hat{0},v]}$   % \gs{$[\hat{0}, v]$} 
%restricted to $[u, v]_r$  % \sout{within $[\hat{0}, v]$} 
 for any given  $u < v$  in $P$ and any  root $r$ for $[u,v]$. %  in $P$. 

Next we give a variation  on condition (i)(b) from the definition of GRAO that will be useful in upcoming inductive arguments.

\begin{prop} 
Let $P$ be a finite bounded poset with  a GRAO.  Consider $t, u, v\in P $  such that $t \lessdot u < v$ and consider a choice of root $r$ for $[t,v]$. Then either the first atom 
of  $[u, v]_{r\cup u}$ in the GRAO is above an earlier atom of $[t, v]_r$  than $u$ or else  $u$ is the first atom % in the GRAO 
of $[t, v]_r$. 
\label{firstatomprop}
\end{prop}
\begin{proof}  %If $P$ admits a GRAO, then 
Lemma \ref{graoibequiv} immediately  implies that either  the first atom % in the GRAO 
of  $[u, v]_{r\cup u}$  is above an earlier atom % in the GRAO 
of $[t, v]_r$ than $u$ or no atom of $[u, v]_{r\cup u}$  is above an earlier atom of $[t, v]_r$ than $u$.   It suffices to consider the case where    %  by way of contradiction  
no atom of $[u, v]_{r\cup u}$ is above an earlier atom than $u$ in the GRAO of $[t, v]_r$, and to show in this case  that $u$ is the first atom in the GRAO of 
$[t,v]_r$.  Suppose by way of contradiction that  $u'$ is  the first atom in the GRAO of $[t, v]_r$ for some
$u'\ne u$. Since both $u'$ and $u$ are below $v$, Definition \ref{grao}, part (ii),  
implies there must exist some atom $u''$ of $[t, v]_r$ that comes before $u$ in the GRAO and some element $x'$ such that $u \lessdot x'$ and $u'' < x' < v$. The existence of such $x'$  contradicts the fact  that no atom of $[u, v]_{r\cup u}$  is above an atom earlier than $u$ in the GRAO of $[t, v]_r$.  This completes the proof.
\end{proof}

%\commentph{Word was in wrong place.} %The next two sentences say almost the same thing as each other.}
Next is  % \sout{perhaps} 
 a  result that  plays an important role in explaining  (and in justifying) the fact that our atom reordering process  transforms  
any GRAO into an RAO.  

 \begin{lem} 
 Let $P$ be a finite, bounded poset with $\Lambda $  a GRAO for $P$. 
 %\sout{Let $\Lambda^{re}$ be the chain-atom ordering obtained by applying the atom reordering process to $\Lambda$.}  
 Then for any $u < v$ in $P$  and  any root $r$ for $[u, v]$, the first atom of $[u, v]_r$ in $\Lambda$
 % \sout{the GRAO }
 is
% remains
 the first atom 
  of $[u, v]_r$ in $\Lambda^{re}$, namely in the atom reordering of $\Lambda $. 
  % after applying the atom reordering process  %from Definition \ref{reorder-def}
  %to $\Sigma $. % the GRAO for $P$.% obtained from the GRAO. % $[u, \hat{1}]_r$. 
\label{alwaysfirst}
\end{lem}

%\commentph{I think we need to replace $[u,v]$ by $[u,v]_r$ at lots of places in this proof, as the restricted GRAO definitely depends on the choice of root.  Likewise
%e.g. for $[t_n,v]$.}

%\commentph{We need to be careful in that it is not precise to say we are calculating $F_r(u,v)$ of $\Lambda^{pr}$ 
%at the step when we have done reordering for all poset elements before 
%$u$ but not yet for those after $u$ -- we need to specify whether or not we have yet done the reordering of atoms of $[u,\hat{1}]_r$ for various $r$, since we need to be 
%working here with respect to a chain-atom ordering, and we have one just before this step and a different one just after this step.}

\begin{proof} 
Proposition \ref{GRAOrestricts} ensures that     % allows us to restrict 
$\Lambda$ restricts 
 to a GRAO for $[u,v]_r$.
%\commentph{Next sentence is related to issue later in proof.  Issue also in Proposition 4.5.}  \commentgs{See proposed fix in light blue (I crossed out the problematic sentence} 
%\sout{ Proposition \ref{restricts-consistently} allows us to speak interchangeably 
%below about  the atom reordering process applied to $\Lambda|_{[u, v]}$, denoted $(\Lambda|_{[u, v]})^{re}$, 
%\sout{the restriction of $\Sigma $ to  $[u,v]$}
%and the restriction to $[u,v]$ of the atom reordering process applied to $\Lambda$, denoted $\Lambda^{re}|_{[u, v]} $. }
Proposition \ref{restricts-consistently} allows us to speak interchangeably 
%\sout{below} 
about % \gs{$(\Lambda|_{[\hat{0},v]})^{re}$}  
 the atom reordering of ${[\hat{0}, v]}$  restricted to $[u, v]_r$ 
%\sout{the restriction of $\Sigma $ to  $[u,v]$}
and % \gs{ $\Lambda^{re}$} 
 the atom reordering of $P$  restricted to $[u,v]_r$. 
%\sout{calculated on all of  $P$.  }}
Let $r=\hat{0} \lessdot t_1 \lessdot t_2 \ldots \lessdot t_n \lessdot u $ be a root for the interval $[u, v]$ in $P$ and let 
$r^-$ denote the root  for $[t_n,v]$  
obtained from $r$ by removing  $u$.  Recall  from Definition \ref{F-u} that  $F^{\Lambda}_{r}(u, v)$ refers to the set of atoms of $[u, v]_r$ that cover an earlier atom of $[t_n, v]_{r^-}$ than $u$ in $\Lambda([t_n, v]_{r^-})$. Also recall that 
$G^{\Lambda}_{r}(u, v)$ refers to the set of  atoms of $[u, v]$ that are not in $F^{\Lambda}_{r}(u, v)$.  
%\commentph{I deleted the rest of this paragraph.} %, thinking it is no longer needed and will help readers to get to the heard of the proof faster.}
%We let $\Lambda^{pr}([t_n, v]_{r^-})$ denote the reordering of atoms of $[t_n, v]_{r^-}$ that results when the reordering process (Algorithm \ref{reorder-def}) is applied to 
%$\Lambda|_{[\hat{0},v]}$.
%\sout{ Note that this reordering of atoms of $[t_n, v]_{r^-}$ is determined at the point in Algorithm \ref{reorder-def} when the atoms of  $[x, v]_{r'}$ have been reordered for all $x$ that come before $t_n$ in the linear extension of   $[\hat{0}, v]$ and all roots  $r'$ of $[x, v]$, but for no subsequent  $x$. }
%\commentph{I wondered if the next bit is now redundant.  I was trying to figure out if we could get to the heart of the proof a bit faster, not making readers work quite so hard to absorb lots of notation first.}   \sout{The set $F^{\Lambda^{pr}}_{r}(u, v)$ refers to the  atoms of $[u, v]_r$ that cover an earlier atom of $[t_n, v]_{r^-}$ than $u$ in $\Lambda^{pr}([t_n, v]_{r^-})$. The set $G^{\Lambda^{pr}}_{r}(u, v)$ refers to the atoms of $[u, v]_r$ that are not in $F^{\Lambda^{pr}}_{r}(u, v)$. } 

We will prove  that the first atom of $\Lambda([u, v]_r)$ is the first atom of 
%\commentph{the next term has not been defined since we only defined restrictions for full chain-atom orderings.  This could be corrected by defining such restriction sufficiently generally, but I really think it would be a lot easier on readers and on the referee  just to make $\Lambda^{pr}$ a full chain-atom ordering everywhere}  
$\Lambda^{re}([u,v]_r)$
%$\Lambda^{pr}([u, v]_r)$ 
for all $u \in P$  and all choices of root $r$.  
We will do so by induction on  the length $l$ of the root $r$ from $\hat{0}$ to $u$. 
%Since Algorithm \ref{reorder-def} sets \commentph{This next is incorrect (it has the wrong meaning)  and should say $\Lambda^{re}[u,\hat{1}]_r := \lambda^{pr}[u,\hat{1}]_r$ for all $u$ and $r$ (and hence $\Lambda^{re}[u,v]_r:= \Lambda^{pr}[u,v]_r$ for all $u,v$ and $r$), not the other way around and not using general $v$ (though we can easily deduce general $v$ from the $v=\hat{1}$ case) -- when you set one thing equal to another, you always list the one you are changing to equal the other first, and it is good form to put $:=$ to emphasize you are setting one equal to the other} 
% $\Lambda^{pr}([u, v]_r)=\Lambda^{re}([u, v]_r)$ for all $u \in P$ and all roots $r$, this will prove our result.   %We begin with the base case with $l=1$.
For the base case, 
let  %\sout{ $l=1$; in other words, let}
$u$ be an atom of $[\hat{0},v]$. Let $x$ be the first atom in $\Lambda([u, v]_{\hat{0}\lessdot u})$.
% \sout{the GRAO of $[u, v]_r$  }
%where  $r$  is the root $\hat{0} \lessdot u$.  %. In this case  $u$ covers $ \hat{0}$ so $r$ is %the saturated chain 
If $x$ is in $F^{\Lambda}_r(u, v)$, then $x$ is in $F^{\Lambda^{re}}_r(u, v)$ %for $u$ an atom 
% \sout{this is still true in the  reordering of $[\hat{0},v]$}, % (see Definition \ref{reorder-def}), 
since $u$ is an atom of $P$ and the atom reordering process does not change the order of the  elements covering $\hat{0}$.
%\commentph{I am concerned that in the next sentence we seem to be needing the output of our reordering algorithm in order to conduct a step in the middle of the algorithm -- I realize this is not really a problem, but I think as it is written, it will be very confusing to readers and make them think something is wrong -- at one point I proposed having a notation for a partially carried out reordering process  and making a convention that it refers to the chain-atom ordering we have at a certain stage in the reordering process.  Perhaps something like that would work better here?} \commentplh{Regarding my next comment, I think it would be more accurate to say something along the lines of ``when the reordering process for $\Lambda $ gets to the stage of reordering $[u,v]_r$'' -- what I just wrote may be clunky, but hopefully it is a good starting point for the correction I think is needed here.}
%  \commentph{I'm also worried it is not accurate to say we apply the reordering process
%to $\Lambda|_{[u,v]_r}$ since doing that would leave the order of the atoms of $[u,v]_r$ unchanged.  There may be places like this where it is tempting to make a 
%compact notation but where it is not possible to make a compact notation that is accurate and will not be misinterpreted.  Not everything necessarily needs a compact
%notation.} 
Since we put the elements of $F_r^{\Lambda^{pr}}(u, v)$ before all other atoms of $[u, v]_r$ during the atom reordering of $\Lambda|_{[\hat{0},v]}$,
%\sout{$[\hat{0}, v]$}, 
$x$ remains first among the atoms of $[u,v]_r$ in $\Lambda^{pr}([u, v]_r)$ in this $x\in F^{\Lambda }_r(u,v)$ case.  %$[u, v]_r$. % by definition of the reordering. 
Now suppose  $x$ is not in $F^{\Lambda}_r(u, v)$.
%\sout{, i.e. suppose $x$ is in $G^{\Sigma}_r(u, v)$. \sout{above an earlier atom than $u$ in the GRAO of $[\hat{0},v]$.  }}
By Proposition \ref{firstatomprop},
$u$ is the first atom in $\Lambda([\hat{0}, v])$.
%Again, s
Since the atom reordering process does not change the order of the elements covering $\hat{0}$, $u$ is also the first atom in $\Lambda^{re}([\hat{0}, v])$. 
%\sout{|the reordering of $[\hat{0}, v]$, because  the atom reordering process preserves the order of the atoms of $P$.}
The fact that  $u$ is first in $\Lambda^{re}([\hat{0}, v])$ implies that 
every atom in $[u, v]_r$ must be in $G^{\Lambda^{pr}}_r(u,v)$. 
 Since the atom reordering process preserves the relative ordering of the 
elements of $G^{\Lambda^{pr}}_r(u,v)$  given by $\Lambda([u, v]_r)$, % given by the GRAO, 
the atom reordering process preserves the ordering of all atoms of $[u, v]_r$.
 Thus,  $x$ remains  first in $\Lambda^{re}([u, v]_r)$. % the atom reordering of $[u, v]_r$, when $l=1$. % in this  case as well. 
 This completes the  base case. %, namely the case  with  $u$  an atom and $l=1$.  

For the inductive step, assume that the first atom in $\Lambda([u,v]_r)$ is also the first atom in $\Lambda^{re}([u,v]_r)$ for any $u < v$ and any root $r$ of length $l$, where %$1 \leq 
$l \leq n$. We will   % we assume by induction that
prove that  the first atom in $\Lambda([u,v]_r)$ is also the first atom in
%\commentph{we haven't  defined this next notation for partial chain-atom orderings such as $\Lambda^{pr}$ for general $u$ and $v$ -- I think it would be far easier on readers never to use partial chain-orderings but rather make it so $\Lambda^{pr}$ is a full chain-atom ordering}
 $\Lambda^{re}([u,v]_r)$ for any $u<v$ and any root $r$ of length $n+1$.  
Let  $x$ be  the first atom in $\Lambda([u, v]_r)$  where $r$ is a  root of length $n+1$.   %\commentph{next is too repetitive and makes sentence too long}
%\sout{where $r=\hat{0} \lessdot t_1 \lessdot t_2 \ldots \lessdot t_n \lessdot u $}.  %Let $t =  t_{n+1}$. 
Suppose by way of contradiction that $x$ is not first in $\Lambda^{re}([u,v]_r)$. 
%Since t
This implies $x$  is moved to a later position % shifted later 
by the atom reordering process applied to $\Lambda $,  %it must be the case that 
%it must be the case
hence it implies  $x\in G^{\Lambda^{pr}}_r(u,v)$. % calculated just prior to reordering the elements covering $u$.  
Since $x$ is the first atom  in $\Lambda([u, v]_r)$, Lemma \ref{graoibequiv}  %Proposition \ref{firstatomprop} 
gives us  that either (1)  $x$ is greater than  an atom that comes before $u$ in $\Lambda([t_n, v]_{r^-})$  %where 
% $r^-$ denotes  $\hat{0} \lessdot t_1 \lessdot t_2 \ldots \lessdot t_n$ 
or  (2) no atom of $[u, v]_r$ is greater than  an atom that comes before $u$ in $\Lambda([t_n, v]_{r^-})$.
We consider these two cases separately. %, obtaining a contradiction  in each case. 

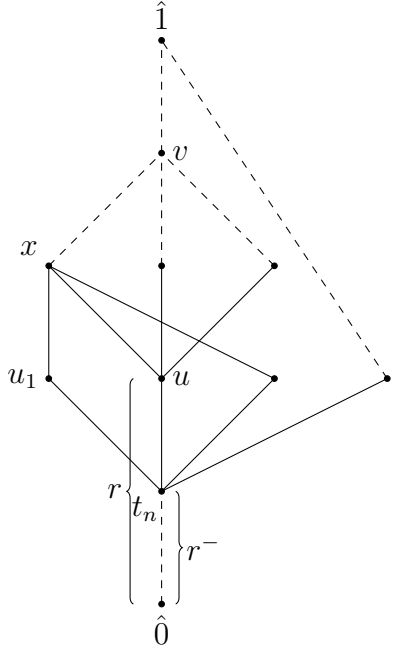
\begin{figure}
\begin{center}
\begin{tikzpicture}[scale=0.75]
\node [below] at (0, -2){$\hat{0}$};
\draw [fill] (0,-2) circle [radius=0.05];
\node [below left] at (0.13,0.1){$t_n$};
\draw [fill] (0,0) circle [radius=0.05];
\node [right] at (0,2){$u$};
\draw [fill] (0,2) circle [radius=0.05];

\node [right] at (0,6){$v$};
\draw [fill] (0,6) circle [radius=0.05];
\node [above] at (0, 8) {$\hat{1}$};
\draw [fill] (0,8) circle [radius=0.05];
\node [right] at (2,2){};
\draw [fill] (2,2) circle [radius=0.05];
\node [right] at (4,2){};
\draw [fill] (4,2) circle [radius=0.05];
\node [above left] at (-2,4){$x$};
\draw [fill] (-2,4) circle [radius=0.05];
\node [left] at (-2,2){$u_1$};
\draw [fill] (-2,2) circle [radius=0.05];
\node [right] at (0,4){};
\draw [fill] (0,4) circle [radius=0.05];
\node [right] at (2,4){};
\draw [fill] (2,4) circle [radius=0.05];
\draw[decoration={brace, mirror}, decorate] (0.25,-2) --node [right] {$r^-$} (0.25,0); 
\draw[decoration={brace}, decorate] (-0.5,-2) --node [left] {$r$} (-0.5,2); 

\draw (0,0)--(4,2);
\draw (0,0)--(0,2);
\draw (0,2)--(-2,4);
\draw [dashed] (4, 2)--(0, 8); 
\draw [dashed] (-2,4)--(0,6);
\draw (0,0)--(2,2)--(-2,4);
\draw (0,0)--(-2,2)--(-2,4);
\draw [dashed] (0, -2)--(0,0); 
\draw [dashed](0,6)--(0, 8);
\draw (0,2)--(0,4);
\draw [dashed] (0,4)--(0,6); 
\draw (0,2)--(2,4);
\draw [dashed] (2,4)--(0,6);
\end{tikzpicture}
\end{center}
\caption{Case 1 in the proof of Lemma \ref{alwaysfirst}.
}
\label{case1alwaysfirst}
\end{figure}

\underline{Case 1:}  This is the case in which  $x$ is above an earlier atom than $u$ in $\Lambda([t_n, v]_{r^-})$. See Figure \ref{case1alwaysfirst} for an illustration of this case.
Let $u_1$ be % the earliest such atom, meaning  $u_1$ is
 the first atom in $\Lambda([t_n, x]_{r^-})$. By our assumption about $x$, 
 %\commentph{next two words inserted so sentence will not be mis-parsed} 
 note that $u_1 \neq u$; this implies that $u_1$ comes earlier than $u$ in $\Lambda([t_n, v]_{r^-})$. 
Since we already showed  $x\in G^{\Lambda^{pr}}_r(u,v)$,  % as calculated just prior to reordering the elements covering $u$, 
 $u_1$ must come later than $u$ in $\Lambda^{re}([t_n, v]_{r^-})$. 
 %\commentgs{Does this next sentence need more justification?} \commentph{Yes, in one way in particular: you need to explain that $\Lambda^{re}$ coincides with $\Lambda^{re}$ whereever $\Lambda^{pr}$ is defined and  that this allows you to apply the 
 %proposition.}  
 Because of Proposition \ref{restricts-consistently}, $u_1$ must come later than $u$ in 
 % \commentph{you haven't defined $\Lambda^{pr}([u,v])$ for $v\ne \hat{1}$ though this issue would go away if $\Lambda^{pr}$ were a full chain-atom ordering -- I suspect that would remove many annoyances and complications like this.} 
   $\Lambda^{re}([t_n, x]_{r^-})$ as well. In particular, $u_1$ is not first in $\Lambda^{re}([t_n, x]_{r^-})$. % , since otherwise $x$ would be in $F^{\Sigma^{re}}_{r}(u,v)$ rather than $G^{\Sigma^{re}}_{r}(u,v)$. \sout{the atom reordering of $t_n$}.  
 But the root $r^-$ has length $n$, which contradicts our 
 inductive hypothesis that requires $u_1$ to remain the earliest atom in $\Lambda^{re}([t_n, x]_{r^-})$. 
 %\sout{ the atom reordering of $[t_n,v]_{r-}$.}  
 Thus, we rule out this case. 
 %, since otherwise $x$ would be in $F_r(u,v)$ rather than $G_r(u,v)$.  %ereas we already have shown $x\in G_r(u)$.  
% Since  $[t_n, x]_{r^-}$ has root of length $n$ and $u_1$ is first in the GRAO of $[t_n, x]_{r^-}$, the induction hypothesis says $u_1$ is the first atom in the reordering of $[t_n, x]_{r^-}$. In particular,  $u_1$ is before $u$ in the reordering, giving a contradiction that rules out this case.  
% of $u_1$ coming later than $u$ in the reordering. Hence $x$ must be first in the the atom reordering of $[u, v]_r$.  

%
%
\begin{figure}
\begin{center}
\begin{tikzpicture}[scale=0.75]
\node [below] at (0, -2){$\hat{0}$};
\draw [fill] (0,-2) circle [radius=0.05];
\node [below left] at (0,0){$t_n$};
\draw [fill] (0,0) circle [radius=0.05];
\node [left] at (0,2){};
\draw [fill] (0,2) circle [radius=0.05];
\node [right] at (0,6){$v$};
\draw [fill] (0,6) circle [radius=0.05];
\node [above] at (0, 8) {$\hat{1}$};
\draw [fill] (0,8) circle [radius=0.05];
\node [right] at (2,2){};
\draw [fill] (2,2) circle [radius=0.05];
\node [right] at (4,2){};
\draw [fill] (4,2) circle [radius=0.05];
\node [above left] at (-2,4){$x$};
\draw [fill] (-2,4) circle [radius=0.05];
\node [above left] at (-2,2){$u$};
\draw [fill] (-2,2) circle [radius=0.05];
\node [right] at (0,4){};
\draw [fill] (0,4) circle [radius=0.05];
\node [right] at (2,4){};
\draw [fill] (2,4) circle [radius=0.05];
\draw[decoration={brace, mirror}, decorate] (0.25,-2) --node [right] {$r^-$} (0.25,0); 
\draw[decoration={brace}, decorate] (-0.25,-2) --node [left] {$r$} (-2.25,2); 
\draw (0,0)--(4,2);
\draw (-2, 2)--(0,4); 
\draw (-2, 2)--(2,4); 
\draw (0,0)--(0,2);
\draw (0,2)--(-2,4);
\draw (0,4)--(4,2);
\draw [dashed] (-2,4)--(0,6);
\draw (0,0)--(2,2)--(-2,4);
\draw (0,0)--(-2,2)--(-2,4);
\draw [dashed] (0, -2)--(0,0); 
\draw [dashed](0,6)--(0, 8);
\draw (0,2)--(0,4);
\draw [dashed] (0,4)--(0,6); 
\draw (0,2)--(2,4);
\draw [dashed] (2,4)--(0,6);
\end{tikzpicture}
\end{center}
\caption{Case 2 in the proof of Lemma \ref{alwaysfirst}.  
}
\label{case2alwaysfirst}
\end{figure}
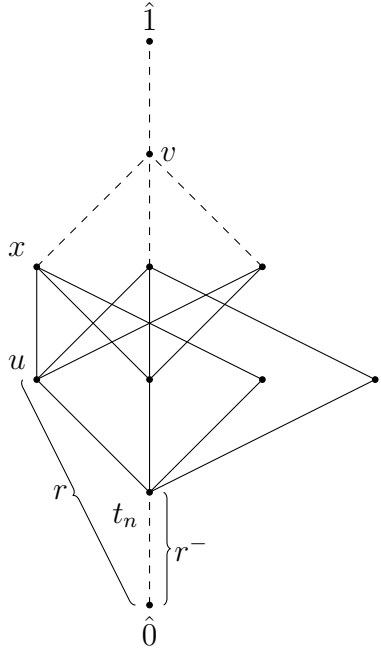

\underline{Case 2:}  %On the other hand, if 
This is the case in which no atom of $[u, v]_r$ is above an atom that is  earlier than $u$ in $\Lambda([t_n, v]_{r^-})$.  See Figure \ref{case2alwaysfirst} for an illustration of this case.
By Proposition \ref{firstatomprop}, $u$ must then be first in $\Lambda([t_n, v]_{r^-})$.
Since $r^-$ has length $n$, our inductive hypothesis ensures that  $u$ must also be  first in $\Lambda^{re}([t_n, v]_{r^-})$. This means that all atoms of $[u, v]_r$ are in $G^{\Lambda^{pr}}_r(u,v)$. %calculated for the appropriate partial reordering.
This implies that  the ordering of all atoms of $[u,v]_r$ is preserved by the atom reordering process. In particular,  $x$ remains first among the atoms of $[u, v]_r$,
  contradicting our assumption  that $x$ is not the first atom of $[u,v]_r$  
 in $\Lambda^{re}([u,v]_r)$. % induced by the reordering of  $P$.  % $[\hat{0},v]$.  
 This completes case 2. %completing case 2. 
\end{proof}

Next is a  lemma %we will use  later.  These
that  sheds light on what the reordering process % applied to  a chain-atom ordering  for  a finite bounded poset $P$ 
does to each  rooted 
interval $[u,\hat{1}]_r$ in a finite bounded poset  $P$ endowed with a chain-atom ordering.  This will be helpful for proving that applying the atom reordering process to a GRAO
yields a chain-atom ordering that is still a GRAO.

%\commentph{In the next result, we are using different notation for restriction to intervals than what I have just been proposing in the background section.}

\begin{lem}\label{equiv-of-chain-atom-orderings}
%\commentph{Below do we want to restrict like this or do a chain-atom ordering on all of $P$?}
Let $\Sigma $ be a chain-atom ordering % $\Sigma $
on a finite bounded poset $P$.  Given any $u\in P$ and any root $r$ for $[u,\hat{1}]$, consider the chain-atom ordering 
% \commentph{consistency whether parentheses like this}
 $\Sigma^{re}|_{[u,\hat{1}]_r}$
%$\Gamma $ 
on $[u,\hat{1}]_r$ 
that is obtained by applying the atom  reordering process to $\Sigma $   
and then restricting the resulting chain-atom ordering for $P$ to $[u,\hat{1}]_r$.   % Let $\mathcal{A}$ be the set of atoms of $[u,\hat{1}]$.  
\begin{enumerate} [label=\alph*)]
\item
The  chain-atom ordering $\Sigma^{re}|_{[u,\hat{1}]_r}$
% $\Gamma $
 may
alternatively be obtained  from $\Sigma|_{[u,\hat{1}]_r} $ by  first  % restricting $\Sigma $  to $[u,\hat{1}]_r$, then 
 permuting the 
atoms of $[u,\hat{1}]_r$   by a permutation $\pi $ % on the set of atoms of $[u,\hat{1}]$ 
to obtain a chain-atom ordering
%\commentph{Change in notation just in this lemma -- I thought the old notation was misleading:} $\pi (\Sigma|_{[u,\hat{1}]_r})$ 
 on $[u,\hat{1}]_r$,  and then   applying the atom  reordering process  %directly 
 to $\pi (\Sigma|_{ [u,\hat{1}]_r})$. % this chain-atom ordering. 
\item
 If  %the chain-atom ordering 
  $\Sigma $   % from Lemma \ref{equiv-of-chain-atom-orderings} for a finite bounded poset $P$ 
  is a GRAO for $P$, then the 
permutation $\pi $ from part (a) % Lemma \ref{equiv-of-chain-atom-orderings} on the atoms of $[u,\hat{1}]_r$  
is a product of adjacent transpositions in 
which each of these adjacent transpositions  fixes % applied in sequence 
 %preserves \commentph{fixes}  which 
%\commentph{preserves} \sout{fixes} 
 which  atom of $[u,v]_r$ is first for every $v>u$ in $P$.  
\item
%\commentph{statement had been written inconsistently with what we actually prove, and I think the stronger statement will help people understand the next proof}
%\commentph{Is this what we want, as opposed to a GRAO for $P$?}
If  %the chain-atom ordering 
 $\Sigma $ is a GRAO for $P$, then the chain-atom ordering   for $P$ 
%$\Sigma_{\pi [u,\hat{1}]_r}$  
obtained from $\Sigma $ by applying the permutation $\pi $ from (a) to permute the order of the atoms of 
$[u,\hat{1}]_r$ while otherwise leaving $\Sigma $ unchanged  
 % Lemma \ref{equiv-of-chain-atom-orderings}, part (b), %\ref{interval-reorder-via-red-exp} 
is a GRAO for $P$.  % $[u,\hat{1}]_r$.
\end{enumerate}
\end{lem}

\begin{proof}
%\gs{Do we need superscripts on $F_r(u,\hat{1})$, etc. in this proof?}
First we prove (a).  
Applying the reordering process to all of $P$, the point will be to observe what happens to $[u,\hat{1}]_r$.  The earliest  step in which the chain-atom ordering for 
$[u,\hat{1}]_r$  
%\commentph{check correctness of next statement} %\commentph{next is wrong if $\Lambda^{pr}$ is only a portion of a chain-atom ordering} 
gets modified by  the reordering process is the step where the atoms of 
$[u,\hat{1}]_r$ get permuted by moving the elements of $F^{\Sigma }_r(u,\hat{1})$ ahead of the elements of $G^{\Sigma }_r(u,\hat{1})$,  preserving the order of the elements  within $F^{\Sigma }_r(u,\hat{1})$ 
and within $G^{\Sigma }_r(u,\hat{1})$.  This gives the desired permutation $\pi $ on the ordering of the atoms of $[u,\hat{1}]_r$.  Now one may observe that
 the subsequent  reordering  of the atoms of each rooted  interval $[u',\hat{1}]_{r'}$ for $u'\in (u,\hat{1}]$ and each root 
$r'$ containing $r$  depends only on $\pi $ and  on the  reordering of  $[u,\hat{1}]_r$ that takes place prior to reaching $u'$ and $r'$. More precisely, 
one  may confirm by 
induction on the length of the 
longest saturated chain from $u$ to $u'$  
that the reordering  of the atoms  of  $[u',\hat{1}]_{r'}$ accomplished within the course of reordering   $P$  %after $\pi $ is applied 
gives  the same atom ordering for $[u',\hat{1}]_{r'}$ as we get by first restricting
 %the original chain-atom ordering 
 $\Sigma $  to  %the rooted interval 
$[u,\hat{1}]_r$, then applying 
the permutation $\pi $ to the atoms of $[u,\hat{1}]_r$,  and finally  applying the reordering process just  to the resulting chain-atom ordering on  $[u,\hat{1}]_r$ (i.e. not working our  way up from $\hat{0}$).    This 
completes the proof of (a)

%\commentph{Switch to notation of Bj\"orner and Wachs for reduced expressions, to be consistent with treatment of Bruhat order later?}
Next we  prove (b).   Lemma \ref{alwaysfirst} implies    that the permutation $\pi $ from (a)  has the property for 
each $v>u$ and each root $r$  that 
%\commentph{Grace proposed change (which is not correct as stated, but Grace wonders if I can say something like this): $\pi $ fixes the first  atom of $[u,v]_r$.  I made an attempt in this direction with the second half of the next sentence.}
  the earliest atom of $[u,v]_r$ before applying  $\pi $  is still the earliest 
atom of $[u,v]_r$ after applying $\pi $; equivalently, the permutation that $\pi $ induces on % those atoms of $[u,\hat{1}]$ that belong to
the atoms of  $[u,v]_r$ fixes the first atom of 
$[u,v]_r$.   Let  $s_{1}\cdots s_{k}$ be any reduced expression for $\pi $, i.e.  any expression for $\pi $ as a product  of adjacent transpositions $(i,i+1)$  with $k$ as 
small as possible.  %\commentph{delete? In what follows, we will use this  same reduced expression for $\pi $  to handle  
%every  possible $v$ satisfying $v>u$ for our fixed choice of $u$ and $r$.  }
The natural correspondence between reduced expressions 
for  any fixed  $\pi \in S_n$ and saturated chains from $e$ to $\pi $  in left % \commentph{left or right?} 
weak order for $S_n$ (which is discussed extensively  e.g. in \cite{BB}) may be combined  with Corollary 3.1.4 in \cite{BB} to  give us the well-known  fact that for any $1\le l < m\le n$ satisfying $\pi (l) < \pi (m) $, we also must have  
$s_{j}\cdots s_{k}(l) < s_{j}\cdots s_{k}(m)$ for  $1\le j\le k$. % such that $1\le j\le k$.  

%\commentph{I just cleaned up the next paragraph, since it was written in a confusing way.}
We will apply  this property of reduced expressions to % this translates to saying % our usage of a reduced expression implies 
certain pairs  of  atoms % \commentph{use $[u,\hat{1}]_r$ here for clarity?}
 $a,a' \in [u,v]_r$ such that  $a$ comes  before $a'$  both in the ordering $\mathcal{A}$ on the atoms of $[u,\hat{1}]_r$ given by % according to
$\Sigma $
and in  the ordering $\pi (\mathcal{A}) $ on the atoms of $[u,\hat{1}]_r$  obtained by  applying the permutation  $\pi $ to $\mathcal{A}$. 
%\commentph{
%Given  $s_{1}\cdots s_k$ which is  a reduced expression for $\pi $, we may  first apply % obtain $\pi (\mathcal{A})$ by first applying 
% $s_k$ to 
%$\mathcal{A}$ to swap the order of two atoms, then apply $s_{k-1}$ to the resulting atom ordering, and  continue in this manner from right to left through 
%$s_1\cdots s_k$ to obtain $\pi (\mathcal{A})$ as the end result. 
%  It % The above  property of reduced expressions  %discussed at the end of the previous
%paragraph 
%guarantees for any pair of atoms $a,a'\in [u,\hat{1}]_r$ with $a$ coming before $a'$ both in $\mathcal{A} $ and in $\pi (\mathcal{A})$  that 
% The above property of permutations implies 
%that  $\pi $ 
%\commentgs{Can you check that this next sentence says what you've intended? I am having a hard time following it (e.g. do we need to repeat ``from right to left?")} %\sout{Applied to our setting,  %t
%this property of reduced expressions
%guarantees for any reduced expression $s_{1}\dots s_{k}$  for $\pi $ and any   such  %\commentph{should be pair, not pairs, next}  
%pair  $a,a'$ of atoms that}  
%$a$ must come earlier than $a'$ in each atom ordering $s_{j}\cdots s_{k}(\mathcal{A}) $ for $1\le j\le k$.
%\sout{  obtained by applying to $\mathcal{A}$  
%an initial segment (proceeding from right to left)  of the series of adjacent 
%transpositions  $s_{1}\cdots s_{k}$ which are applied to $\mathcal{A}$  from right to left.}  % to $\mathcal{A}$. % whose product is $\pi $.   
%\commentph{
For any  two atoms $a,a'$ of  an interval $[u,v]_r$ such that  $a$ is the first atom of $[u,v]_r$ both  in the  ordering for the atoms of 
 $[u,v]_r$  inherited from  
 $\mathcal{A}$  %(by restricting $\mathcal{A}$ to the subset of atoms of $[u,\hat{1}]_r$ belonging to $[u,v]_r$) 
  and in the  ordering  for the atoms of  $[u,v]_r$   inherited from  
$\pi (\mathcal{A})$,  $a$ comes before $a'$ both in $\mathcal{A}$ and in $\pi (\mathcal{A})$.  Thus, our  above 
property of reduced expressions implies that  
$a$ must come before $a'$ in $s_j\cdots s_k (\mathcal{A})$ for each $j\le k$.   This implies that $a$ must be the first atom in the ordering on the atoms 
of  $[u,v]_r$
% in the ordering on the   and hence in  
%the ordering of  the atoms of $[u,v]_r$ 
that is inherited from the 
ordering $s_{j}\cdots s_{k}(\mathcal{A})$ on the atoms of $[u,\hat{1}]_r$  for each $j\le k$.
%In particular, t
Thus, %is shows  that % \commentph{vague next few words}
each adjacent transposition in $s_1\dots s_k$ applied to $\mathcal{A}$  from right to left  preserves which atom is first in every rooted interval of $P$.
% \commentph{also ``fixes'' vs ''preserves'' in statement of (b)}

%\commentph{Here we prove something stronger than the above statement.  Also this is repetitive with the proof of the next theorem.}
Finally we prove (c). 
%Lemma \ref{equiv-of-chain-atom-orderings}, part (b), % \ref{interval-reorder-via-red-exp} 
Part (b)  tells us that each of the adjacent transpositions in $s_{1}\cdots s_{k}$  has the property when it is applied in turn to $\mathcal{A}$  that  it 
preserves which atom is first in each $[u,v]_r$. 
This enables us to 
use Theorem \ref{switch-thm} repeatedly, once for each of the adjacent transpositions in  $s_{1}\cdots s_{k}$,  %the reduced expression, 
to deduce that the chain-atom ordering obtained from 
$\Sigma $ by % restricting to $[u,\hat{1}]_r$ and then  r
reordering the atoms of $[u,\hat{1}]_r$ by $\pi $   is a GRAO for $P$. % $[u,\hat{1}]_r$.  
\end{proof}

We are now prepared   to prove that 
applying the atom reordering process to a GRAO yields a GRAO  and then to go on from there to prove that  the resulting GRAO is in fact an RAO.

\begin{lem}\label{reordering-as-swaps}
The atom reordering process applied to a GRAO for a finite bounded poset  $P$ yields  a GRAO for $P$.    
\end{lem}

\begin{proof}
%\commentph{This paragraph and some of the language later in this proof are new as I am trying to figure out how to work with the new version of the atom
%reordering process.}
% -- changes necessitated by the recent change in what data structures are used at the intermediate steps of the atom reordering process.   We still should think about whether it is really okay to have the same thing mean two different things in two different parts of the paper --  I am worried it will confuse readers to define $\Lambda^{pr}$ to mean two different things but have been working hard at trying to make this as palatable as possible.  I am still somewhat concerned that this inconsistency/awkwardness in how we do things  could  increase the chance of  the paper getting rejected -- this might also require a comment near the algorithm itself when we are discussing in words what it does.}
Throughout this %e following  
proof, we  will work with % a  minor variation on
 the more 
 enriched variation  on the atom reordering process  % with a slight change to the data structure we use in the intermediate steps in the algorithm, as 
 discussed just prior to 
 Algorithm \ref{reorder-def}.   That is, % Specifically, 
 %in which 
%here 
 we regard $\Lambda^{pr}$ at each step in the 
process as being a full chain-atom ordering, % doing so as described shortly.
doing so as follows. We initialize $\Lambda^{pr}$ to equal $\Lambda $, and then otherwise run the algorithm just as in Algorithm \ref{reorder-def}.  
In other words, 
at  each  of the steps of the algorithm in %atom reordering process in 
which  
%$[u,\hat{1}]_r$ such that 
 $\Lambda^{pr}([u,\hat{1}]_r)$ has not yet been defined for a given $u\in P$ and a given choice of root $r$ in the usual algorithm,
  we set $\Lambda^{pr}([u,\hat{1}]_r)$  equal to  $\Lambda ([u,\hat{1}]_r)$ in our modified atom reordering process, 
  otherwise leaving  the process entirely unchanged. 
  %
%  \commentph{this next paragraph is repetitive as things now stand.}
Since  this is merely an  enrichment of the data  contained in $\Lambda^{pr}$ at the intermediate stages
of the algorithm, one may easily see that this  % is  a change in viewpoint and in notation  that  % in a way  that 
does not impact  the output of the algorithm or 
%that does not impact 
the applicability of the proofs of Lemmas \ref{restricts-consistently},  \ref{alwaysfirst} and  \ref{equiv-of-chain-atom-orderings}
to this enriched version of 
 %regarding 
 the atom reordering process.
%mainly a in viewpoint, and 
%\
%
%
  %Thus,  at each step of the atom reordering process we now treat  $\Lambda^{pr}$ as being a chain-atom ordering. 
   With these conventions,  $\Lambda^{pr}$ progressively 
   evolves from being the chain-atom ordering $\Lambda $ at the start of the algorithm   to equalling %and progressively 
   %evolves into 
    the chain-atom ordering
   $\Lambda^{re}$ at the end of the algorithm, as will be necessary for the argument below to apply.   
   %This is essentially a change in notation, though also somewhat of a change in viewpoint:  rather than treating 
   %$\Lambda^{pr}$ as being  built up over the course of the algorithm into a chain-atom ordering with the end-result being $\Lambda^{re}$, 
   %here we  instead regard it as evolving from the
   %existing chain-atom ordering  $\Lambda $ into the output chain-atom ordering  $\Lambda^{re}$.
 %   Some readers might find it convenient to  think of  $\Lambda^{pr}$ in this way at various points in the paper.

By Lemma \ref{alwaysfirst}, the atom reordering process for $\Lambda $ 
takes $[u,\hat{1}]_r$ for each  $u\in P$ and each root $r$ for $[u,\hat{1}]$ and permutes  the atoms $a_1,\dots ,a_m$  of $[u,\hat{1}]_r$ in a way that preserves which atom comes first in
% \commentph{I crossed out a statement that was mathematically incorrect here
%and replaced with a mathematically correct statement} \sout{fixes 
%the earliest atom of} 
every rooted  interval $[u,v]_r$ for $v>u$ with  our fixed root $r$.  
% Since the output of the modified atom reordering process is the same as the output of the atom reordering process, 
%Lemma \ref{alwaysfirst} likewise applies to the modified process.  
%  as the first atom of $[u,v]_r$  after this permutation of the atoms of $[u,\hat{1}]_r$.    
%\commentph{Possible replacement for next sentence (which I think is much more clear):  
Thus, we may 
apply Lemma \ref{equiv-of-chain-atom-orderings}, part (b), to deduce that this permutation on the atoms of $[u, \hat{1}]_r$  is expressible as a reduced expression comprised of a series of adjacent transpositions, each of which preserves which atom comes first in   $[u,v]_r$ 
%as the first atom of $[u,v]_r$ 
for each $v > u$.
%Thus, Lemma \ref{equiv-of-chain-atom-orderings}, part (b), implies that  
%this permutation on the atoms of $[u,\hat{1}]_r$  is expressible as a reduced expression comprised of 
%a series of adjacent transpositions, each of which keeps the earliest atom of $[u,v]_r$ as the first atom of $[u,v]_r$ for each $v>u$.
%, as is proven in Lemma \ref{equiv-of-chain-atom-orderings}, part (b).
%\commentph{The next sentence is not correct as stated now that $\Lambda^{pr}$ is not a chain-atom ordering at the intermediate steps in the algorithm.  I worry there could be other statements throughout the paper that I have not managed to locate yet that became incorrect when the reordering process was changed due to changing what the objects are that we have at the intermediate steps.}  
But this means that  we may express this enriched version of  the  atom reordering process  %discussed above 
% (in which $\Lambda^{pr}$ is an entire  chain-atom ordering at each step of the process)
as a series of such steps by considering each $u$ and each $r$ in turn, applying such  
a series of adjacent transpositions as in Lemma \ref{equiv-of-chain-atom-orderings}, part (b), %\ref{interval-reorder-via-red-exp} 
for each $u$ and each $r$.   %But then Lemma \ref{switch-thm} ensures that i
If we start with a GRAO, then Lemma \ref{switch-thm} tells us that  after each of these steps, i.e. after each of these adjacent transpositions, % of this process 
 we still have a GRAO.  In particular, % this means that
the end-result  of this enriched version of  the  atom reordering process  applied to a GRAO  for $P$ %, the resulting chain-atom ordering 
is a GRAO for $P$.  Since by design  this enriched version of the atom reordering process  has % is constructed so as to  have 
the same output $\Lambda^{re}$  as the atom reordering process given in Algorithm \ref{reorder-def}, this completes the 
proof.
\end{proof}

%\commentph{double check this next corollary is in the right place (and check the sentence before it) -- yes it is}

Before getting to our main result of this section, we mention a consequence of what we have just proven that could give  
useful insight into the atom reordering process.

\begin{cor}
The reordering
 process applied to a GRAO of a finite bounded poset $P$  preserves the sets $F_r(u,v)$ and $G_r(u,v)$ for each $u<v$ in $P$ and each root $r$ for $[u,v]$.
\end{cor}

\begin{proof}
This follows from Lemma \ref{F-and-G-preserved} and the proof of Lemma  \ref{reordering-as-swaps}.
\end{proof}

Now we are ready to prove the main result of this section. 

%\commentph{I think ''Moreover''  is confusing in the next statement in the sense  that the second sentence only speaks about one direction of the result in the first 
%sentence.   I'm trying to think of something better to put here -- this is such an important result within the paper, I want to make sure it is well stated.}

%\commentph{In the proof of the next result, we do not currently put a root subscript on $[a_j,\hat{1}]$ throughout, whereas elsewhere we do.  I'm of two minds which is better, since in this case the subscript clutters things without adding any new information, though I suppose it may help people see how things do depend on root as we propagate upward.  Anyway, we might want to add the subscripts here.}
%\commentgs{Since the root here is always $\hat{0}\lessdot a_j$, and so there's no choice of root, I'm leaning towards not adding it}

\begin{thm} A finite bounded  %,  graded 
poset admits a generalized recursive atom ordering (GRAO) if and only if it admits a recursive atom ordering (RAO).  
%Moreover,  the atom reordering procedure given  in Algorithm ~\ref{reorder-def}  transforms any GRAO into an RAO.  \commentph{One proposed alternative to prior sentence: 
Moreover,  every RAO is a GRAO while  every GRAO may be transformed into an RAO by the atom reordering process.  
\label{cccl}
\end{thm}

\begin{proof}  
By Lemma \ref{raothengrao}, any RAO is a GRAO.
What remains is to prove that applying the reodering process from Algorithm \ref{reorder-def}  to any GRAO for a finite bounded poset $P$ yields an RAO.
We will do this  by induction on the length $l$  of the longest saturated chain in $P$.  
For $l=2$, any ordering of the atoms of $P$ is a recursive atom ordering, giving the base case for our  proof by   %\commentph{Insert: proof by}
 induction.

Now assuming the result for $l\le n$, it will suffice to show  that this implies the result for $l=n+1$.  % \commentph{Insert something like: 
Let $\Lambda $ be a GRAO for $P$, regarded as a chain-atom ordering. 
Let $a_1,a_2,\dots ,a_q$ be the ordering of the atoms of $P$ with $l=n+1$ in  $\Lambda $. % a GRAO for $P$.  
In this case, for each atom $a_j\in P$, we have that 
 $[a_j, \hat{1}]$ %_{s\cup u}$ 
 is an interval whose longest  saturated chain is of length  at most  $n$.   
 By condition (i)(a) in the definition of GRAO,  %\commentph{notation for restriction here?  and elsewhere in this proof? -- agreed} 
 %\commentph{replace next bit  with:
  $\Lambda|_{[a_j,\hat{1}]}$
 % the restriction of our GRAO for $P$ to 
 %$[a_j,\hat{1}]$ 
 is a GRAO for $[a_j,\hat{1}]$.  
 By Lemma \ref{equiv-of-chain-atom-orderings}, parts (a) and (b),   %  and \ref{interval-stays-GRAO}, 
 applying the atom reordering process to  $\Lambda $ 
 and 
 %\commentplh{here we think we should not switch to notation, as it is important here to emphasize the sequence of operations}
  then restricting the resulting chain-atom ordering  
  $\Lambda^{re} $  to
 $[a_j,\hat{1}]$ yields the same chain-atom ordering for $[a_j,\hat{1}]$ that 
 we get by instead permuting the atoms of $[a_j,\hat{1}]$ in a way that  yields a new GRAO for $[a_j,\hat{1}]$  that is denoted by
  $\pi(\Lambda|_{[a_j,\hat{1}]})$
 and then applying the atom reordering process directly to  $\pi (\Lambda|_{[a_j,\hat{1}]})$.
 % this 
 %new GRAO on $[a_j,\hat{1}]$.  
 By our inductive hypothesis, the result of this  latter  series of operations 
 is an RAO for $[a_j,\hat{1}]$.  This implies   that   %Thus %,  we assume 
  condition (i)(a) of Definition \ref{rao} holds for $\Lambda^{re}$. % the chain-atom ordering resulting from the atom reordering process. 
 By definition of the atom reordering process, 
 % \commentph{do we really need this next lemma here or is it extraneous?  I checked, and it is extraneous, so I commented
 %it out}  
 %and by
 %Lemma \ref{F-and-G-preserved},  %\commentph{what I crossed out could be misinterpreted, so I suggest something like: 
 %an initial segment of atoms  
%\sout{ the atoms that come first} 
%in the atom reordering of $[a_j,\hat{1}]$ %]_{s\cup u}$  
 %\sout{are} \commentph{consists of} 
  those atoms of $[a_j,\hat{1}]$  that are above an  earlier atom of $P$  than $a_j$  all come before all of the other atoms of $[a_j,\hat{1}]$  
  in $\Lambda^{re}$,
  % the atom
  %reordering of $P$, % $[a_j,\hat{1}]$, %of $[u,\hat{1}]_s$,
  so condition  (i)(b) of Definition \ref{rao} is also satisfied for  $\Lambda^{re}$.
  %the chain-atom ordering produced by  the atom reordering process. 

What remains is to prove  that condition (ii) of Definition \ref{rao} holds for $\Lambda^{re}$.  
% the chain-atom ordering obtained by applying the atom  reordering process to 
%$\Lambda $. %our GRAO for 
%$P$.  % $[u, \hat{1}]_s$. 
Here we may use Lemma \ref{reordering-as-swaps} which tells us that applying the atom reordering process to $\Lambda $ 
%\commentph{Replace ``showed'' earlier in this sentence  with  ``tells us that'' or ``implies'' or something like this} 
%our GRAO for $P$ 
yields a GRAO for $P$.  Therefore $\Lambda^{re}$  satisfies condition (ii) in the definition of GRAO.  But condition (ii)  in the definition of GRAO  is  
the same % (albeit phrased slightly differently) 
%exactly  \commentph{is it equivalent or is it the same.  We want it to be the same, so should adjust as needed.} 
%equivalent to
as  condition (ii) in the definition of RAO, albeit phrased slightly differently, completing our proof.
\end{proof}

\section{Self-consistency, the UE property and TCL-shellability}\label{UE-section}

In this section, we introduce a condition that a CC-labeling may have called self-consistency.  We also introduce a fairly natural and readily checkable property called the UE property that will imply self-consistency.  We prove that all CL-labelings have the UE property.

These notions are introduced in preparation for a result later in the paper where we will  prove %clarify the relationship between CC-shellability and CL-shellability, doing so by proving 
that a finite bounded poset is CL-shellable if and only if it is CC-shellable by way of  a  self-consistent CC-labeling.  To  help us further clarify the relationship between the established notions of  CC-shellability and CL-shellability,   %established notions of lexicographic shellability 
we also  introduce a variation on CC-shellability that  we call TCL-shellability.

\begin{defn}\label{SC-def}
 Consider a chain-edge labeling $\lambda$ such that each rooted interval has a unique lexicographically earliest saturated chain.   We define such $\lambda $  to be  \textbf{self-consistent} if for any rooted interval $[u, v]_r$ we have the following condition: if $a$ is the atom in the lexicographically first saturated chain of $[u, v]_r$ and $b \neq a$ is also an atom of $[u, v]_r$, then for any  $[u, v']_r$ containing $a$ and $b$, all saturated chains of $[u,v']_r$ containing $b$ come lexicographically later than all saturated chains  of $[u,v']_r$ containing $a$.
 If a chain-edge labeling is not self-consistent,  then it is said to be  \textbf{self-inconsistent}.
\end{defn}

Next we introduce a property that will imply  self-consistency that is more readily checkable.  

\begin{defn} A chain-edge labeling $\lambda$ of a finite bounded poset $P$ has the \textbf{unique earliest (UE) property} if for each rooted interval $[u, v]_r$ in $P$, the smallest label occurring on any cover relation upward from $u$ only occurs on one such cover relation. 
\label{ueprop}
\end{defn}

\begin{lem}\label{UEimpliesSC}
If a CC-labeling has the UE property, then it is self-consistent.
\end{lem}

\begin{proof}
Let $a$ be the unique atom of $[u, v]_r$ for which $\lambda(u, a)$ is smallest among all labels upward from $u$. 
Consider any other atom $b\in [u,v]_r$ and any $v'$ satisfying $u < v'$ with $a,b\in [u,v']$. The label sequences on saturated chains of $[u,v']_r$ containing $a$ must be lexicographically smaller than those containing $b$ by virtue of  $\lambda (u,a)$ being smaller than $\lambda (u,b)$ with respect to root $r$.
\end{proof}

\begin{cor}\label{RAOimpliesSCCC}
Any recursive atom ordering for a finite bounded poset gives rise to a CC-labeling with the UE property, hence to a self-consistent CC-labeling.
\end{cor}

\begin{proof}
Simply observe that the CC-labeling constructed in the proof of Proposition  ~\ref{RAOimpliesCC}  has the UE property, then apply Lemma \ref{UEimpliesSC}.
\end{proof}

The next result gives some evidence that the UE property is not an unduly burdensome  condition to impose.

\begin{lem}\label{CLhaveUE}
Every CL-labeling has the UE property.
\end{lem}

\begin{proof}
Consider a pair of  elements $u < v$ in a finite bounded poset $P$ with a CL-labeling $\lambda $, and consider any root $r$ of $[u, v]$ from $\hat{0}$ to $u$.    Suppose that there are distinct atoms $a,a' \in [u,v]$ such that $\lambda_r (u,a) = \lambda_r (u,a')$. Further suppose $\lambda_r (u,a) \le \lambda_r (u,a'')$ for all other atoms $a''\in [u,v]_r$.    We may choose $a$ to belong to the lexicographically earliest saturated chain $M$ from $u$ to $v$ in $[u,v]_r$ since this saturated chain will begin with the smallest possible first label.  Let 
$u\lessdot a\lessdot x_2\lessdot x_3\lessdot \cdots \lessdot x_k\lessdot v$ be this lexicographically first saturated chain of $[u,v]_r$.  Let 
$u\lessdot a'\lessdot y_2\lessdot y_3\lessdot \cdots \lessdot y_l\lessdot v$ be the lexicographically first saturated chain of $[u,v]_r$ that contains $a'$.   Denote by $M'$ this 
latter saturated chain from $u$ to $v$.   
We must have  % \commentph{Shouldn't next $<$ be $\le $?  The same question applies a couple  sentences later when we have the same inequality...Need to be consistent throughout the paper whether ascending chains are weakly ascending or strictly ascending.} 
%\gs{Root should be $r\cup a$ not $r\cup u$.} 
$\lambda_r (u,a) < \lambda_{r\cup a} (a,x_2)$  since $M$ is an ascending chain.  But we also have $\lambda_{r\cup a} (a,x_2) \le \lambda_{r\cup a'}  (a',y_2)$ since the label sequences for $M$ and $M'$ both start with the same label and %have these as their second labels with 
$M$ has a lexicographically smaller label sequence than $M'$.  Thus, we have 
$$\lambda_r (u,a') = \lambda_r (u,a) < \lambda_{r\cup a}  (a,x_2) \le \lambda_{r\cup a'}  (a',y_2)$$ implying that $M'$ has an ascent at $a'$.  But $M'$ is also ascending from $a'$ to $v$ since $M'$ is lexicographically first in $[a',v]_{r\cup a'}$.  Thus, $M$ and $M'$ are both  ascending chains on $[u,v]_r$, giving a contradiction.
\end{proof}

\begin{ex}\label{varying-properties}
Figure \ref{nonue} shows a poset $P$ and three different CC-labelings (all of which are in fact EC-labelings). The leftmost  CC-labeling  is self-inconsistent  because 
both $[\hat{0}, x]$ and $[\hat{0}, y]$ contain $a$ and $b$ as atoms, but $a$ is first in $[\hat{0}, x]$ while $b$ is first in $[\hat{0}, y]$. 
Given that  this labeling is self-inconsistent, it also cannot have the UE property.   The CC-labeling in the middle of Figure \ref{nonue} is self-consistent but does not have the UE property because both cover relations up from $\hat{0}$ have  label 1.  The CC-labeling on the right has the UE property and is self-consistent.
\end{ex}

\begin{figure}

\begin{tikzpicture}[scale=1.25]
\node [below] at (0,0) {\tiny $\hat{0}$}; 
\draw [fill] (0,0) circle [radius=0.05]; 
\node [left] at (-1,1) {\tiny $a$}; 
\draw [fill] (-1,1) circle [radius=0.05]; 
\node [right] at (1,1) {\tiny $b$}; 
\draw [fill] (1,1) circle [radius=0.05]; 
\node [left] at (-1,2) {\tiny $x$}; 
\draw [fill] (-1,2) circle [radius=0.05]; 
\node [right] at (1,2) {\tiny $y$}; 
\draw [fill] (1,2) circle [radius=0.05]; 
\node [above] at (0,3) {\tiny $\hat{1}$}; 
\draw [fill] (0,3) circle [radius=0.05];
\draw (1,1)--(-1,2);
\draw (0,0)--node [midway, circle, fill=white, inner sep=0.5pt,minimum size=1pt] {\tiny 1} (-1,1)-- node [midway, circle, fill=white, inner sep=1pt,minimum size=1pt] {\tiny 1} (-1,2)-- node [midway, circle, fill=white, inner sep=0.5pt,minimum size=1pt] {\tiny 1} (0,3);
\draw (0,0)--node [midway, circle, fill=white, inner sep=0.5pt,minimum size=1pt] {\tiny 1}(1,1)--node [midway, circle, fill=white, inner sep=0.5pt,minimum size=1pt] {\tiny 3}(1,2)-- node [midway, circle, fill=white, inner sep=0.5pt,minimum size=1pt] {\tiny 2}(0,3);
\draw (-1,1)--node [near end, circle, fill=white, inner sep=0.5pt,minimum size=1pt] {\tiny 4}(1,2);
\draw (1,1)--node [near end, circle, fill=white, inner sep=0.5pt,minimum size=1pt] {\tiny 2} (-1,2);
%\node at (0,-1.6) {}; 
\end{tikzpicture}
\hskip 1cm
\begin{tikzpicture}[scale=1.25]
\node [below] at (0,0) {\tiny $\hat{0}$}; 
\draw [fill] (0,0) circle [radius=0.05]; 
\node [left] at (-1,1) {\tiny $a$}; 
\draw [fill] (-1,1) circle [radius=0.05]; 
\node [right] at (1,1) {\tiny $b$}; 
\draw [fill] (1,1) circle [radius=0.05]; 
\node [left] at (-1,2) {\tiny $x$}; 
\draw [fill] (-1,2) circle [radius=0.05]; 
\node [right] at (1,2) {\tiny $y$}; 
\draw [fill] (1,2) circle [radius=0.05]; 
\node [above] at (0,3) {\tiny $\hat{1}$}; 
\draw [fill] (0,3) circle [radius=0.05];
\draw (1,1)--(-1,2);
\draw (0,0)--node [midway, circle, fill=white, inner sep=0.5pt,minimum size=1pt] {\tiny 1} (-1,1)-- node [midway, circle, fill=white, inner sep=1pt,minimum size=1pt] {\tiny 1} (-1,2)-- node [midway, circle, fill=white, inner sep=0.5pt,minimum size=1pt] {\tiny 1} (0,3);
\draw (0,0)--node [midway, circle, fill=white, inner sep=0.5pt,minimum size=1pt] {\tiny 1}(1,1)--node [midway, circle, fill=white, inner sep=0.5pt,minimum size=1pt] {\tiny 4}(1,2)-- node [midway, circle, fill=white, inner sep=0.5pt,minimum size=1pt] {\tiny 2}(0,3);
\draw (-1,1)--node [near end, circle, fill=white, inner sep=0.5pt,minimum size=1pt] {\tiny 3}(1,2);
\draw (1,1)--node [near end, circle, fill=white, inner sep=0.5pt,minimum size=1pt] {\tiny 2} (-1,2);
%\node at (0,-1.6) {}; 
\end{tikzpicture}
\hskip 1cm
\begin{tikzpicture}[scale=1.25]
\node [below] at (0,0) {\tiny $\hat{0}$}; 
\draw [fill] (0,0) circle [radius=0.05]; 
\node [left] at (-1,1) {\tiny $a$}; 
\draw [fill] (-1,1) circle [radius=0.05]; 
\node [right] at (1,1) {\tiny $b$}; 
\draw [fill] (1,1) circle [radius=0.05]; 
\node [left] at (-1,2) {\tiny $x$}; 
\draw [fill] (-1,2) circle [radius=0.05]; 
\node [right] at (1,2) {\tiny $y$}; 
\draw [fill] (1,2) circle [radius=0.05]; 
\node [above] at (0,3) {\tiny $\hat{1}$}; 
\draw [fill] (0,3) circle [radius=0.05];
\draw (1,1)--(-1,2);
\draw (0,0)--node [midway, circle, fill=white, inner sep=0.5pt,minimum size=1pt] {\tiny 1} (-1,1)-- node [midway, circle, fill=white, inner sep=0.5pt,minimum size=1pt] {\tiny 2} (-1,2)-- node [midway, circle, fill=white, inner sep=0.5pt,minimum size=1pt] {\tiny 3} (0,3);
\draw (0,0)--node [midway, circle, fill=white, inner sep=0.5pt,minimum size=1pt] {\tiny 4}(1,1)--node [midway, circle, fill=white, inner sep=0.5pt,minimum size=1pt] {\tiny 6}(1,2)-- node [midway, circle, fill=white, inner sep=0.5pt,minimum size=1pt] {\tiny 7}(0,3);
\draw (-1,1)--node [near end, circle, fill=white, inner sep=0.5pt,minimum size=1pt] {\tiny 8}(1,2);
\draw (1,1)--node [near end, circle, fill=white, inner sep=0.5pt,minimum size=1pt] {\tiny 5}(-1,2);
%\node at (0,-1.6) {}; 
\end{tikzpicture}

\caption{Self-inconsistent and self-consistent CC-labelings} 
\label{nonue}
\end{figure}
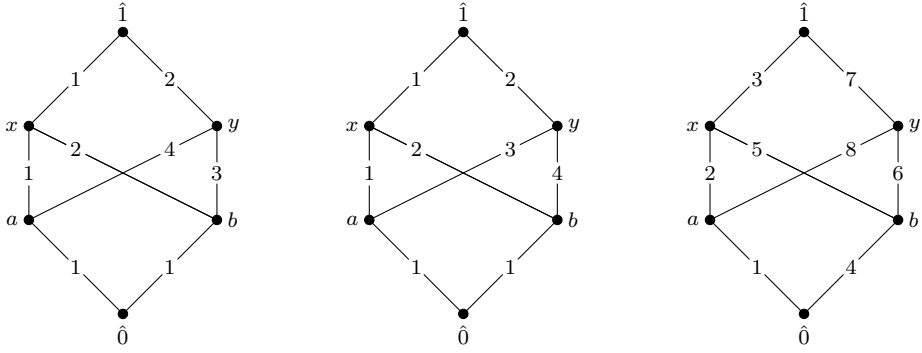

It seems plausible that many of the most interesting and the most 
natural examples of CC-labelings  will have the UE property.  % or will  at least be self-consistent. %in some manner.  % As some evidence for this, 
Section \ref{CC-UE} shows that an  EC-labeling from \cite{elect} for the  dual poset to the uncrossing order has the UE property, 
 allowing us to invoke  the upcoming 
Corollaries \ref{ConsistECimpliesCL} and \ref{co-RAO}
to deduce dual  CL-shellability for the uncrossing order. 
Section  \ref{CL-UE} shows how readily checkable the UE property  %may be, by carrying this out
is for numerous well-known EL-labelings and CL-labelings.

\subsection{TCL-shellability:  a  variation on CC-shellability}\label{topol-CL-section}
%\subsection{Topological CL-shellability: a  variation on CC-shellability}\label{topol-CL-section}

One may replace Kozlov's requirements of a CC-labeling (a) that  saturated chains have distinct label sequences and (b)  that no label sequence be a prefix of any other label sequence by simply requiring each rooted interval to have a unique lexicographically earliest saturated chain.  With just this requirement, the proof of Bj\"orner and Wachs that a CL-labeling yields a shelling of the order complex will carry over using topological descents in place of descents and using topological ascents in place of ascents.  In particular, the UE property suffices in place of these two aforementioned conditions of Kozlov to guarantee we get a shelling. 
 
  This yields a class of poset chain-edge labelings inducing shellings  that includes all CL-labelings, whether or not the labels upward from the lowest element $u$ in any rooted interval $[u,v]_r$ are all distinct from each other.
  %; in particular, unlike Kozlov's CC-shellability % in contrast to CC-shellability in its original formulation by Kozlov, 
   %this includes CL-labelings where not every interval has M\"obius function $0,1,$ or $-1$.  
   In justifying his assertion that CC-shellability is the most general possible version of lexicographic shellability, Kozlov notes in \cite{kozlov}  that any CL-labeling may be modified into a CC-labeling, implying that every CL-shellable poset is CC-shellable.

\begin{defn}\label{topol-CL}
A  \textbf{TCL-labeling}  of a finite bounded poset $P$ is a chain-edge labeling $\lambda $ for $P$ such that (i) each rooted interval $[u,v]_r$ has a unique saturated chain from 
$u$ to $v$ with lexicographically smallest label sequence, and (ii) every other saturated chain from $u$ to $v$ in $[u,v]_r$ has at least one topological descent.   If a poset has a TCL-labeling, then it is \textbf{TCL-shellable}. %  is  the property of having a TCL-labeling. 
\end{defn}

%\begin{defn}\label{topol-CL}
%A {\bf topological CL-labeling} of a finite bounded poset $P$ is a chain-edge labeling $\lambda $ for $P$ such that (i) each rooted interval $[u,v]_r$ has a unique saturated chain from 
%$u$ to $v$ with lexicographically smallest label sequence, and (ii) every other saturated chain from $u$ to $v$ in $[u,v]_r$ has at least one topological descent.   {\bf Topological CL-shellability} is  the property of having a topological CL-labeling. 
%\end{defn}

%\commentph{Question: If we delete this paragraph, do we adequately discuss elsewhere that there are self-inconsistent TCL-labelings?}
%We will sometimes refer to condition (i) in Definition \ref{topol-CL} as the {\bf UE-chain property}. \commentgs{It seems we never actually use the term UE-chain property; I suggest we remove this.}  Notice that it is implied by the UE property, but does not itself imply the UE property. 
 %In particular, there can be self-inconsistent topological CL-labelings (see, for example, Figure \ref{nonue}). 

\begin{thm}\label{CCimpliesTopolCL}
Every   %self-consistent 
CC-labeling is a TCL-labeling.  However, not every TCL-labeling is a  CC-labeling.
\end{thm}

\begin{proof}
First notice that the requirement of a CC-labeling that the saturated chains in any rooted interval  have distinct label sequences in particular implies the uniqueness of the lexicographically earliest label sequence on each rooted interval. 
Since the saturated chain from $u$ to $v$ with the  lexicographically earliest label sequence on a rooted interval  $[u,v]_r$ 
cannot have any topological descents, 
%Also notice that %condition (CC) appearing in the definition given in \cite{kozlov} for a CC-labeling  amounts exactly  to the requirement of 
condition (1) in the definition of CC-labeling  implies that every other  saturated chain of $[u,v]_r$  must have a topological 
descent.
%that we give in the introduction; that is, it is equivalent to requiring  that each rooted interval have a unique topologically ascending chain.
This shows that  every % any self-consistent 
CC-labeling is a TCL-labeling.

To see that the converse fails, simply consider the CL-labeling $\lambda $ (which is also a TCL-labeling) for a graded bounded poset of rank 2 consisting  of elements 
$\hat{0},\hat{1},x_1,x_2,x_3$ with $\hat{0}\lessdot x_i \lessdot \hat{1}$ for $i=1,2,3$ having  $\lambda (\hat{0},x_1) = \lambda (x_2,\hat{1}) = \lambda (x_3,\hat{1}) = 1$ and 
$\lambda (\hat{0},x_2) = \lambda (\hat{0},x_3) = \lambda (x_1,\hat{1}) = 2$.  This fails the requirement for CC-labelings that distinct saturated chains have distinct label sequences.
\end{proof}

\begin{rmk}
It may very well be true that a finite bounded poset is CC-shellable if and only if it admits a TCL-labeling.   The labeling given in the second half of the 
proof of Theorem  \ref{CCimpliesTopolCL} can easily be modified into a CC-labeling, and indeed it remains open whether these two notions of shellability are equivalent in terms 
of which posets admit such shellings.  Likewise, as far as we know, it is open whether CC-shellability is equivalent to CL-shellability, though this seems to be quite a tricky question.
\end{rmk}

\begin{thm}\label{RAOimpliesTopolCL}
Any recursive atom ordering  of a finite bounded poset  $P$  induces a TCL-labeling for $P$ that is self-consistent.
\end{thm}

\begin{proof}
One may simply check that the CC-labeling constructed in the proof of Theorem \ref{RAOimpliesCC} is in fact self-consistent (by virtue of having the UE property) and
  is  a TCL-labeling.  One may also invoke Theorem \ref{CCimpliesTopolCL} for the latter claim.
\end{proof}

\begin{thm}\label{tclthenshell}
If  a finite bounded poset $P$ admits a TCL-labeling, then this induces one or more lexicographic shellings for $\Delta (P)$.  Specifically, any linear order on 
the maximal chains of $P$ that is a linear extension of the partial order obtained by ordering maximal chains lexicographically is a shelling order.  
\end{thm}

\begin{proof}
The proof  of Bj\"orner and Wachs in \cite{BW82} that any CL-labeling of a finite bounded poset $P$ induces a lexicographic shelling 
for $\Delta (P)$ (for each choice of linear extension of the lexicographic order on saturated chains)   still applies unchanged for TCL-labelings
when we simply  replace 
ascents by topological ascents,  descents by topological descents, ascending chains by topologically ascending chains, and descending chains by 
topologically descending chains  throughout the proof.  
\end{proof}

\section{Equivalence of CL-shellability to self-consistent CC-shellability} % and to several other properties  a poset may possess}
\label{CLiffCCsection}

In this section, we will prove  that several different conditions on a finite bounded poset are equivalent to CL-shellability.

\begin{thm} If a finite bounded poset $P$ admits a generalized recursive atom ordering, then $P$ admits a self-consistent CC-labeling.
\label{grao2thenCC}
\end{thm}

\begin{proof}
We proved in Theorem \ref{cccl} that any finite bounded poset $P$ with a generalized recursive atom ordering also admits a recursive atom ordering. We showed in 
Corollary 
\ref{RAOimpliesSCCC} that any such $P$ admits a self-consistent CC-labeling, completing the proof.  
\end{proof}

\begin{rmk}
\label{altgraothencc}
One may alternatively prove Theorem \ref{grao2thenCC} by using any generalized recursive atom ordering for a finite bounded poset $P$ to construct a self-consistent CC-labeling 
for $P$ that is compatible  (in the sense of  Definition \ref{compatible-def})  
with the GRAO without needing to use the atom reordering process at all.  The construction of such a  CC-labeling is carried out  as follows.   
For each $u \in P$ and each root $r$ for $[u,\hat{1}]$, the GRAO gives an ordering $a_1,\dots ,a_{t(u,r)}$ of the atoms of the rooted interval 
$[u,\hat{1}]_r$.  Assign the label $i $ to the cover relation $u\lessdot a_i$ for this choice of root $r$.  One may prove that this produces a UE (and thus self-consistent) CC-labeling for $P$.  It is worth noting that this is typically a different self-consistent CC-labeling for $P$ than the one produced as in the proof above of Theorem \ref{grao2thenCC}, as we do not apply the atom reordering process to $P$ in this more direct approach.    See the proof of  Theorem \ref{second-proof} for all  of the details of this approach.
\end{rmk}

\begin{thm} If a finite bounded poset $P$ admits a self-consistent TCL-labeling, then $P$ admits a  generalized recursive atom ordering. 
\label{TopolCLthengrao}
\end{thm}

\begin{proof} 
Let $\lambda $ be a self-consistent TCL-labeling for  $P$.   By definition, $\lambda $  restricts to  
a TCL-labeling for  $[u, \hat{1}]_r$ for each $u\in P$ and each root $r$.  
Consider  the following chain-atom ordering  of $P$ that 
is  compatible with $\lambda $ in the sense 
that each rooted interval $[u,v]_r$ will have its earliest atom in the chain-atom ordering belonging to the lexicographically earliest saturated chain
of $[u,v]_r$ according to $\lambda $.
For any $u\in P$ and any choice of root $r$, define  the first atom  of $[u,\hat{1}]_r$  in our  chain-atom ordering
 to be the unique atom belonging to the lexicographically first saturated chain of $[u,\hat{1}]_r$ with respect to $\lambda $.  Denote this atom as $a_1$.
 Now assuming we have already chosen atoms $a_1,a_2,\dots ,a_i$ of $[u,\hat{1}]_r$ for some $i\ge 1$,  choose $a_{i+1}$ as follows.    Among the saturated chains of 
 $[u,\hat{1}]_r$ that do not include any of the atoms $a_1,a_2,\dots ,a_i$, choose one that is as small as possible in lexicographic order.  There may be more than one choice, but pick any such saturated chain.  Let $a_{i+1}$ be the atom of $[u,\hat{1}]_r$ belonging to this chosen saturated chain.    Continuing in this manner, 
 we % selection of $a_1,a_2,\dots ,a_i$  
 specify  a total order $\Omega = a_1,\dots ,a_t$ on the  atoms of $[u,\hat{1}]_r$. %  for each $u\in P$ and each root $r$.  
 By construction,  $\Omega$ is  compatible with % the CC-labeling 
 $\lambda$ by virtue of $\lambda $ being self-consistent.

We will prove that $\Omega$ is a  generalized recursive atom ordering (GRAO), % i.e. that $\Omega $ satisfies all of the conditions of Definition \ref{grao}, 
doing so with a proof by induction on the length of the longest saturated chain of $P$. %  We sometimes refer to this quantity as the length of $P$.  
For $P$ of length 2, our base case, any ordering on the atoms of $P$ is a GRAO. Let $P$ be a finite bounded poset whose longest  saturated 
chain is of length $l \ge  3$.  %  To show that $\Omega $ is a GRAO  for such $P$, w
We assume by induction  that the ordering $\Omega$ %restricted to  
on the atoms of $[a_j,\hat{1}]$ % gives an ordering of  the atoms of $[a_j, \hat{1}]$ that we may assume by induction  on $l$  %This CC-labeling is self-consistent because 
 is a GRAO for $[a_j,\hat{1}]$, directly yielding condition (i)(a) of Definition \ref{grao}. %$\lambda $ is self-consistent.  

Next we prove that this GRAO on $[a_j, \hat{1}]$ satisfies condition  (i)(b) of Definition \ref{grao}.  That is, we verify  % We will do so by proving
 the following assertion  for any $w$ covering an element $x$ satisfying  $a_j \lessdot x$: % \lessdot w$:  
 if there exists an atom $x'$ of $[a_j,w]$ such that $x'$ is greater than $a_{i}$ for some atom $a_{i} $ with $i< j$, then  the first atom $x_1$ of $[a_j,w]$ is greater than an atom $a_{i'}$ with $i'<j$.
Suppose this assertion is false.
By definition of $\Omega $ and by its compatibility with $\lambda $,  $x_1$ must be  in the  %lexicographically first  
saturated chain of $[a_j, \hat{1}]$ with the lexicographically smallest label sequence among those saturated chains of $[a_j, \hat{1}]$ that 
contain $w$. % (and hence among  those that contain an atom of $[a_j, w]$ as well as containing $w$). 
Since $\lambda$ is self-consistent, % Lemma \ref{consistentlem} implies 
$x_1$ is also   in the % lexicographically first 
saturated chain of $[a_j, w]$ having the lexicographically smallest label sequence. % among saturated chains of  $[a_j,w]$. 
Let  $a_j \lessdot x_1 \lessdot u_1 \lessdot u_2 \lessdot \ldots \lessdot u_n \lessdot w$ be this lexicographically first  saturated chain in $[a_j, w]$ according to $\lambda $.   
Since $\lambda$ is a TCL-labeling, %$(\lambda(a_j, x_1), \lambda(x_1, u_1))$ must be a topological ascent, and likewise 
all pairs of consecutive cover relations in $a_j\lessdot x_1\lessdot u_1\lessdot\ldots \lessdot u_n\lessdot w$ must be  topological ascents.   But our assumption that 
$x_1$ is not above an earlier atom than $a_j$ implies that $a_j$ is the first atom of $[\hat{0},x_1]$.  Since the GRAO denoted by 
$\Omega $ and the TCL-labeling $\lambda $
are compatible, 
this implies that $\hat{0}\lessdot a_j\lessdot x_1$ is a topological ascent with respect to $\lambda $. %since the GRAO is compatible with the CC-labeling. 
 Thus, we have shown that 
$M = \hat{0}\lessdot a_j\lessdot x_1 \lessdot x_2\lessdot u_1\lessdot u_2\lessdot \cdots \lessdot u_n\lessdot w$ is a topologically ascending saturated
chain from $\hat{0}$ to $w$  with respect to the labeling $\lambda $.

We now  show how our assumption   that $x_1$ is not above an earlier atom than $a_j$ while  there is some  atom $x'$ for $[a_j,w]$ that  is above an earlier atom  
will imply  the existence of another topologically ascending chain for $[\hat{0},w]$ 
 besides $M$.
 Let $a_w$ be the first atom of $[\hat{0},w]$ with respect to the atom ordering $\Omega$.  %among those that are below $w$. 
 Since  $a_i < x' $ for some $i < j$, we have
 $a_i<w$ and hence have  $a_w \neq a_j$. % \commentgs{confirm change in this next sentence:} 
 Since 
 $\lambda$ is a TCL-labeling compatible with $\Omega$, % \sout{$\lambda$},
  $(\lambda(\hat{0}, a_w), \lambda(a_w, y))$ is a topological ascent for any $y$ such that $a_w \lessdot y < w$. Let $c$ be the unique topologically ascending chain in $[a_w, w]$. Then the  saturated chain $M'$ for $[\hat{0},w]$  obtained by appending $\hat{0}$ to $c$ is a  topologically ascending chain in $[\hat{0}, w]$.  Since $M'\ne M$, this contradicts $M$ being the unique topologically ascending  chain from $\hat{0}$ to $w$, completing our proof  of condition  
 (i)(b)  of Definition \ref{grao}.

Now let us check that  condition 
(ii) of Definition \ref{grao} holds for $\Omega$. Consider any atoms $a_i,a_j$ with $i<j$ satisfying  $a_i < y$ and $a_j < y$ for some element $y$. 
Let $M$ be the lexicographically first saturated chain in $[a_j, y]$ and  let  $a_j \lessdot x$ be the lowest  %bottom is not correct since that refers to being below everything else which 
cover relation in $M$.
 Let $a_y$ be the first atom of $[\hat{0},y]$   with respect to the atom ordering  $\Omega$. % among those atoms that are  less than  $y$. 
 Since $a_i$ comes before $a_j$ in $\Omega$ and %with $a_i$  satisfying 
 $a_i < y$, we must have   $a_j \neq a_y $. Let  $M'$ be the lexicographically first saturated  chain in $[a_y, y]$ and let $a_y \lessdot z$ be the lowest  
 cover relation in $M'$.  
 %\commentgs{confirm change in this next sentence:}
  Since $\lambda$ is a TCL-labeling compatible with $\Omega$, % \sout{$\lambda$},
    $a_y$ and $z$ are both  in the lexicographically first saturated chain of $[\hat{0}, y]$. 
 Thus,   $(\lambda(\hat{0}, a_y), \lambda(a_y, z))$ is  % in the lexicographically first saturated chain from $\hat{0}$ to $
 a topological ascent, implying $M'$ is a topologically ascending chain  from $\hat{0}$ to $y$.  
  Since % there exists only one topologically ascending chain in $[\hat{0}, y]$ and since 
 $M$ is topologically ascending from $a_j$ to $y$ but is not the 
 unique topologically ascending chain in $[\hat{0},y]$, % $[a_j,y]$, 
  $M$ must have a topological descent at $(\lambda(\hat{0}, a_j), \lambda(a_j, x))$.
  This means  there must exist $\hat{0}\lessdot a_k\lessdot u$ such that  $u \leq x$  and $\hat{0}\lessdot a_k\lessdot u$ belongs to 
 the lexicographically first  saturated chain of $[\hat{0}, x]$. Since $\lambda$ is self-consistent, $a_k$ is in the lexicographically first saturated  chain of $P$ that  contains an 
atom of $[\hat{0}, x]$. As $a_j$ and  $a_k$ are both atoms of $[\hat{0}, x]$, we conclude that  $a_k$ comes before $a_j$ in $\Omega$.  
This confirms that condition  (ii) of Definition \ref{grao} holds for $\Omega$. 
\end{proof}

\begin{thm}\label{TFAE-theorem}
 Let $P$ be a finite, bounded poset. Then the  following are equivalent:
\begin{enumerate} % [label=\alph*)]
\item $P$ admits a recursive atom ordering
\item $P$ admits a generalized recursive atom ordering
\item $P$ admits a CL-labeling
\item $P$ admits a CL-labeling with the UE property
\item $P$ admits a self-consistent CC-labeling. 
\item $P$ admits a CC-labeling with the UE property
\item $P$ admits a self-consistent TCL-labeling
\item $P$ admits a TCL-labeling with the UE property
\end{enumerate}
Moreover, all of these implications are proven constructively.  That is, 
for each implication, either a construction is given showing how to construct the latter type of object from the former or 
else the former type of object is proven also to be the latter type of object.
%i.e. either by constructing one type of object from the other or in some cases by proving one type of object already is an instance of the  other.
\end{thm} 

\begin{proof}  
%\commentgs{This paragraph seems to just be a clunkier version of the last sentence of the statement of the theorem. I propose we remove it.} \commentph{We need to prove the final claim in the theorem, and this paragraph tells how we do that.  Without it many readers will be confused.}
%\sout{ For each of the implications that we will prove, we will not only show  
%that any finite bounded  poset endowed with one type of structure is also endowed 
%with another  type of structure, but also  that we have 
%an explicit construction to obtain the latter  structure from the former. }
%In our proof   below, t
The results we cite below for the various implications are all proven constructively.
%\sout{    In some cases, we do this by showing  that the given structure is always  an instance of the 
%desired type of structure, implying  that  no construction is necessary to transform the one to the other.  
%In all other cases, the results we cite below  are proven constructively, making  all of our implications below  constructive.}

We proved the % two directions of the 
equivalence of (1) and (2) %  constructively
 in Theorem %s ~\ref{raothengrao} and  
 ~\ref{cccl}.  Bj\"orner and Wachs proved 
the equivalence of (1) and (3)  % constructively
  in \cite{bw}.  We proved in Lemma \ref{CLhaveUE}  that every CL-labeling has the UE property, giving the equivalence of (3) and (4).

Next we confirm the equivalence of (2) and (5). % in a constructive manner. 
Theorem \ref{CCimpliesTopolCL} followed by Theorem \ref{TopolCLthengrao} gives  one
direction. The converse is given by Theorem \ref{grao2thenCC}.

Now we prove  the equivalence of (2) and (6). %constructively.  
 If $\lambda$ is a CC-labeling with the UE property, then  Lemma ~\ref{UEimpliesSC} implies $\lambda$ is a self-consistent CC-labeling.  This gives (5) which was already shown to 
 imply (2).   %, rendering no construction necessary.  
For the converse, Remark \ref{altgraothencc} explains how to construct a CC-labeling with the UE property from a GRAO; the justification that this indeed yields a CC-labeling
may be found within the proof of 
Theorem \ref{second-proof}.  It is self-evident from  how it is constructed that this CC-labeling has the UE property.  

Next we verify  the equivalence of (2) and (7).  %constructively.  
Theorem \ref{grao2thenCC} followed by Theorem \ref{CCimpliesTopolCL} %, each of which is proven constructively, 
gives the forward direction. The reverse direction is proven % by way of a construction 
in Theorem ~\ref{TopolCLthengrao}.

Finally we prove the equivalence of (7) and (8).  %, once again doing so constructively.  
To show that (8) implies (7), we use the result from Lemma \ref{UEimpliesSC}  that the UE property implies self-consistency. %, implying that no construction is needed.   
 For the other direction, first use the result already proven above that (7) implies (4). % constructively.  
 Then  use the fact that every CL-labeling is a TCL-labeling (by virtue of how these are defined)  
 to get that (4) implies (8). % with no construction needed.  
 Combining these implications shows that (7) implies (8). % in a constructive way, completing our proof.
\end{proof} 

%The above theorem may immediately be rephrased as follows:

%\begin{cor}
% Let $P$ be a finite, bounded poset. Then the  following are equivalent:
%\begin{enumerate} % [label=\alph*)]
%\item $P$ admits a recursive atom ordering
%\item $P$ admits a generalized recursive atom ordering
%\item $P$ is CL-shellable % admits a CL-labeling
%\item $P$ is CL-shellable by way of a CL-labeling with the UE property
%\item $P$ admits a self-consistent CC-labeling. 
%\item $P$ admits a CC-labeling with the UE property
%\item $P$ admits a self-consistent topological CL-labeling
%\item $P$ admits a topological CL-labeling with the UE property
%\end{enumerate}
%Moreover, all of these implications are proven constructively.  That is, 
%for each implication either a construction is given showing how to construct the latter type of object from the former or 
%else the former type of object is proven also to be the latter type of object.
%\end{cor}

We conclude this section with  two simple consequences of the above results. They are included because they 
will both be used in Section \ref{CC-UE} in our proof that  the uncrossing order is
dual CL-shellable.

\begin{cor}\label{ConsistECimpliesCL}
Let $P$ be a finite bounded poset with a self-consistent EC-labeling $\lambda $.  Then $\lambda $ may be transformed algorithmically into a CL-labeling for $P$.
\end{cor}

\begin{proof}
This is immediate from Theorem \ref{TFAE-theorem}, % \ref{ConsistCCimpliesCL}, 
using the fact that every EC-labeling is a CC-labeling.
\end{proof}

%\commentplh{Keep dual here.}
%added back dual statement here since we need it in the next section and indeed reference this result (just before the start of the section on TCL-shellability).}
%\commentph{Not sure if we should remove the next corollary, since we will use this one in the next section.  We could move it closer to where we use it?}

\begin{cor}\label{co-RAO}
A  finite bounded poset  $P$ admits a generalized recursive coatom ordering if and only if it admits a recursive coatom ordering. 
\end{cor}

\begin{proof}
Simply apply Theorem \ref{cccl} to $P^*$.
\end{proof}

\section{%Application: d
Dual CL-shellability of   the uncrossing order}  % and other posets known to be CC-shellable}
\label{CC-UE} %  for new classes of posets and balanced complexes}

As an application of our earlier results, in this section we  prove  that the uncrossing order  $P_n$  is dual CL-shellable.  Interest in the family of uncrossing posets stems from their 
 role as the face posets of naturally arising stratified spaces of planar electrical networks (see e.g. ~\cite{lam}).      In ~\cite{lam}, Lam proved $P_n$ to be Eulerian and
conjectured that the 
dual to the uncrossing order is lexicographically shellable.  Indeed, the uncrossing order was proven dual EC-shellable in ~\cite{elect}.  Our 
results clarifying the relationship between CL-shellability and CC-shellability (as well as EC-shellability)  will allow  us to deduce dual CL-shellability of $P_n$  from 
the dual EC-shelling of ~\cite{elect}  for $P_n$  once we  prove that the EC-labeling from ~\cite{elect}    has the UE property.

Let us begin by recalling  the definition of $P_n$.
Figure \ref{P3cover} depicts   two cover relations in $P_3$  while Figure \ref{P3} shows  all of $P_3$.  These  examples may be helpful 
 for understanding both  the definition of $P_n$ and how its cover relations are labeled. 
The poset $P_n$ 
has a unique minimal element $\hat{0}$.  The other elements of $P_n$ are  the various   complete matchings on a set of  $2n$ vertices (called boundary nodes) positioned 
around a circle and labeled clockwise $1,2,\dots ,2n$.  By convention, we will typically mark the node labeled 1 (by making it larger than the other nodes) and leave the other nodes unlabeled in our examples.   
These complete  matchings giving rise to the elements of $P_n \setminus \{ \hat{0} \}$  
may be represented by collections of $n$ strands, with each strand  connecting  a matched pair of boundary nodes.  
$P_n$ is graded with the rank of any 
element other than $\hat{0}$ being one more than 
the minimal number of strand crossings needed to represent its strand diagram in the plane.  One of these strand diagrams $u$ will be covered by
another strand diagram $v$ whenever the rank of $u$ is one less than the rank of $v$ and $u$ is obtainable from $v$ by uncrossing a single pair of strands.  
Thus, a  strand diagram $u$ will be less than a strand diagram $v$  in $P_n$ whenever  $u$ may be obtained from $v$ by a series of steps, with each step 
uncrossing a pair of strands. 
The elements of $P_n$ covering $\hat{0}$ are exactly those strand diagrams  %which are realizable in the plane 
with  no two strands crossing each other; thus, for $C_n$ the $n$-th Catalan number,  there are exactly  $C_n$ such elements covering $\hat{0}$.

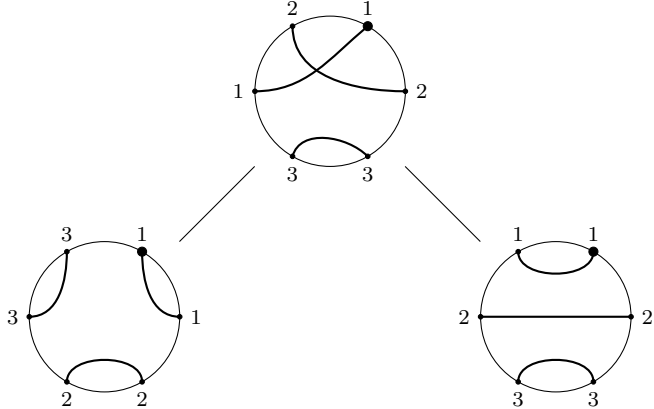
\begin{figure}
\begin{center}
\begin{tikzpicture}
 
\node (2) at (3,6){\begin{tikzpicture}
\draw (0,0) circle [radius=1];
\node [above] at (60:1){\tiny 1};
\draw [fill] (60:1) circle [radius=0.06];
\node [above] at (120:1){\tiny 2};
\draw [fill] (120:1) circle [radius=0.03];
\node [left] at (180:1){\tiny 1};
\draw [fill] (180:1) circle [radius=0.03];
\node [below] at (240:1){\tiny 3};
\draw [fill] (240:1) circle [radius=0.03];
\node [below] at (300:1) {\tiny 3};
\draw [fill] (300:1) circle [radius=0.03];
\node [right] at (360:1){\tiny 2};
\draw [fill] (360:1) circle [radius=0.03];
\draw [thick] (60:1) to [out=220, in=0] (180:1);
\draw [thick] (360:1) to [out=180, in=270] (120:1);
\draw [thick] (300:1) to [out=135, in=75] (240:1);
\end{tikzpicture}};

\node (6) at (0,3){\begin{tikzpicture}
\draw (0,0) circle [radius=1];
\node [above] at (60:1){\tiny 1};
\draw [fill] (60:1) circle [radius=0.06];
\node [above] at (120:1){\tiny 3};
\draw [fill] (120:1) circle [radius=0.03];
\node [left] at (180:1){\tiny 3};
\draw [fill] (180:1) circle [radius=0.03];
\node [below] at (240:1){\tiny 2};
\draw [fill] (240:1) circle [radius=0.03];
\node [below] at (300:1) {\tiny 2};
\draw [fill] (300:1) circle [radius=0.03];
\node [right] at (360:1){\tiny 1};
\draw [fill] (360:1) circle [radius=0.03];
\draw [thick] (60:1) to [out=270, in=180] (360:1);
\draw [thick] (240:1) to [out=90, in=90] (300:1);
\draw [thick] (120:1) to [out=270, in=0] (180:1);
\end{tikzpicture}};
 
\node (10) at (6,3){\begin{tikzpicture}
\draw (0,0) circle [radius=1];
\node [above] at (60:1){\tiny 1};
\draw [fill] (60:1) circle [radius=0.06];
\node[above] at (120:1){\tiny 1};
\draw [fill] (120:1) circle [radius=0.03];
\node [left] at (180:1){\tiny 2};
\draw [fill] (180:1) circle [radius=0.03];
\node [below] at (240:1){\tiny 3};
\draw [fill] (240:1) circle [radius=0.03];
\node [below] at (300:1) {\tiny 3};
\draw [fill] (300:1) circle [radius=0.03];
\node [right] at (360:1){\tiny 2};
\draw [fill] (360:1) circle [radius=0.03];
\draw [thick] (60:1) to [out=270, in=270] (120:1);
\draw [thick] (180:1) to [out=0, in=180] (360:1);
\draw [thick] (240:1) to [out=90, in=90] (300:1);
\end{tikzpicture}};
 
\draw (1, 4)--(2,5);
\draw (5, 4)--(4, 5);
 
\end{tikzpicture}
\caption{Cover relations uncrossing strands in two different ways in $P_3$.}
\label{P3cover}
\end{center}
\end{figure}

Next we describe an  encoding of  the elements of  $P_n\setminus \{ \hat{0} \}$ that was introduced in \cite{elect}.  This will be used  
to  define the dual $EC$-labeling for $P_n$ from \cite{elect}.  Again, Figures 7 and 8 may help with 
parsing this discussion by illustrating the case with  $n=3$.   % See Figure 6.
Start by choosing one of the 2n boundary nodes  to treat as a basepoint in each of the strand diagrams of $P_n$.  In Figures 7 and 8, this base node is enlarged and labeled with a 1.  We associate to each strand diagram a sequence of 2n numbers as follows.   %of length 2n comprised of 2 copies of each of the numbers from 1 to $n$ as follows.   
First we assign the number $1,2,\dots ,n$ to the $n$ strands by numbering the strands in the order they are  first encountered as we  proceed clockwise about the circle starting at our basepoint.  Next, we label each strand endpoint with the number that has been assigned to  its strand.  %hen (b) reading off the sequence of strands encountered clockwise from the 
Finally, we proceed clockwise about the circle from the basepoint reading off the sequence of numbers assigned to the $2n$ nodes.  
This procedure  assigns to each strand diagram  a sequence of length $2n$ comprised of two copies of each of the numbers from 1 to $n$.  Among such sequences, the ones that arise this way are exactly those in which  the first copy of $i$ occurs earlier than the first copy of $j$ for each $i<j$, and this map from $P_n\setminus \{ \hat{0}\} $ to the set of such sequences (introduced in \cite{elect})  is a bijection.

As an example, Figure \ref{P3cover} shows the strand diagram giving rise to the sequence $123312$ as well as the  two 
strand diagrams it covers that are obtained by uncrossing the pair of strands in the two different possible ways.  These two elements it covers give rise to the sequences 
$112233$ and $123321$, with the former shown on the left and the latter shown on the right in Figure \ref{P3cover}.  See Figure \ref{P3} for all of the poset $P_3$ with the strand diagrams each 
labeled numerically so that the associated sequence may be read off by proceeding clockwise from the basepoint.

Uncrossing the pair of strands labeled $i$ and $j$  for $i<j$  in a strand diagram will change a  sequence having the subsequence $ijij$ into a sequence having  either the subsequence  $iijj$ or the subsequence $ijji$ (with some of these letters possibly appearing in new positions in the sequence), depending on which of the two  possible ways the uncrossing is done.  This is discussed in more depth in \cite{elect}, but these further details are  not critical to our upcoming proof that the dual EC-labeling of \cite{elect} (which is also a dual CC-labeling)  %\sout{CC-labeling of \cite{elect}}
 may be transformed into  
a dual CL-labeling.   

  \begin{figure}[h]
  \label{uncrossing-poset-example}
        \begin{picture}(325,500)(-80,-25)
 
       \put(77,441){\line(1,0){36}}
      \put(80,462){\line(1,-1){27}}
      \put(83,434){\line(1,1){28}}

        \put(83,434){\circle*{4}}
              \put(113,441){\circle*{4}}
       \put(77,441){\circle*{4}}
       \put(108,434){\circle*{4}}
       \put(80,462){\circle*{6}}
       \put(111,462){\circle*{4}}

 \put(70,465){$1$}
 \put(115,462){$2$}
 \put(118,439){$3$}
 \put(112,425){$1$}
 \put(72,425){$2$}
 \put(66,439){$3$}

\put(40,445){$\hat{1} = $}
       \put(95,450){\circle{100}}
       
       \curve(5,392,  80,420)
       \curve(95,392,  95,420)
       \curve(180,392,  110,420)
       
       \put(5, 410){$(1,2)$}
       \put(98, 400){$(2,3)$}
       \put(155, 410){$(3,1)$}

               \put(95,365){\circle{100}}             
      \put(80,377){\line(1,-1){27}}
\curve(77, 356, 94, 362, 111,377)
\curve(83,349, 98, 356, 113,356)

        \put(83,349){\circle*{4}}
              \put(113,356){\circle*{4}}
       \put(77,356){\circle*{4}}
       \put(108,349){\circle*{4}}
       \put(80,377){\circle*{6}}
       \put(111,377){\circle*{4}}
       
       \put(70, 379){$1$}
       \put(115, 379){$2$}
       \put(117, 353){$3$}
       \put(110,340){$1$}
       \put(73,340){$3$}
       \put(65,352){$2$}

        \put(-5,365){\circle{100}}
\curve(-17, 349,  -15, 363, -20, 377)
\curve(8,349, 5, 363,  11, 377)
\curve(-23,356,  13, 356)

        \put(-17,349){\circle*{4}}
              \put(13,356){\circle*{4}}
       \put(-23,356){\circle*{4}}
       \put(8,349){\circle*{4}}
       \put(-20,377){\circle*{6}}
       \put(11,377){\circle*{4}}

 \put(-30, 379){$1$}
 \put(14, 379){$2$}
 \put(18, 353){$3$}
 \put(10, 340){$2$}
 \put(-27,340){$1$}
 \put(-34,353){$3$}
 
        \put(195,365){\circle{100}}
      \put(183,349){\line(1,1){28}}
\curve(180,377, 197, 363, 213, 356)
\curve(177,356,  192, 357, 208, 349)

        \put(183,349){\circle*{4}}
              \put(213,356){\circle*{4}}
       \put(177,356){\circle*{4}}
       \put(208,349){\circle*{4}}
       \put(180,377){\circle*{6}}
       \put(211,377){\circle*{4}}

\put(170,379){$1$}
\put(214,379){$2$}
\put(217,351){$1$}
\put(210,338){$3$}
\put(175,338){$2$}
\put(167, 350){$3$}
        
        \curve(-15,335,  -45,250)
        \curve(-5,335,  5,250)
        \curve(5,335, 60,250)
        \curve(15,335, 115,250)
        
        \put(-57,295){$(1,3)$}
        \put(-26, 290){$(2,3)$}
        \put(26, 256){$(3,2)$}
        \put(52, 273){$(3,1)$}

        \put(41, 325){$(1,2)$}
        \put(65, 301){$(1,3)$}
        \put(93, 306){$(3,1)$}
        \put(122, 328){$(2,1)$}
        
        \put(86, 286){$(1,2)$}
        \put(110, 264){$(3,2)$}
        \put(188, 290){$(2,3)$}
        \put(222, 295){$(2,1)$}

        \curve(80,335, -35, 250)
        \curve(90,335,  10, 250)
        \curve(100,335,  170, 250)
        \curve(110,335,  230, 250)
        
        \curve(175, 335, 70,250) 
        \curve(185, 335, 125, 250) 
        \curve(195, 335, 180, 250)
        \curve(205, 335, 240, 250)

\put(65,220){\circle{100}}
\curve(81,232, 78, 221, 83, 211)
\curve(50,232, 56, 218, 53,204)
\curve(47,211, 62, 213, 78,204)

        \put(53,204){\circle*{4}}
              \put(83,211){\circle*{4}}
       \put(47,211){\circle*{4}}
       \put(78,204){\circle*{4}}
       \put(50,233){\circle*{6}}
       \put(81,232){\circle*{4}}

\put(40, 235){$1$}
\put(80, 235){$2$}
\put(80,195){$3$}
\put(87,207){$2$}
\put(45,195){$1$}
\put(38,207){$3$}

\put(125,220){\circle{100}}
\curve(110,232, 127, 217, 143,211)
\curve(107,211, 114,211, 113,204)
\curve(141,232, 135, 218, 138,204)

        \put(113,204){\circle*{4}}
              \put(143,211){\circle*{4}}
       \put(107,211){\circle*{4}}
       \put(138,204){\circle*{4}}
       \put(110,233){\circle*{6}}
       \put(141,232){\circle*{4}}

\put(100, 235){$1$}
\put(140, 235){$2$}
\put(140,195){$2$}
\put(147,207){$1$}
\put(105,195){$3$}
\put(98,207){$3$}

\put(5,220){\circle{100}}
\curve(-13,211, 4, 216,  21, 232)
\curve(23,211, 17, 210, 18, 204)
\curve(-10, 232, -4, 218,  -7, 204)

        \put(-7,204){\circle*{4}}
              \put(23,211){\circle*{4}}
       \put(-13,211){\circle*{4}}
       \put(18,204){\circle*{4}}
       \put(-10,233){\circle*{6}}
       \put(21,232){\circle*{4}}

\put(-20,235){$1$}
\put(20,235){$2$}
\put(20,195){$3$}
\put(27,207){$3$}
\put(-15,195){$1$}
\put(-22,207){$2$}

\put(-55,220){\circle{100}}
\curve(-73,211, -67, 221, -70,232)
\curve(-67,204, -51, 214, -37,211)
\curve(-39,232, -45, 218, -42,204)

        \put(-67,204){\circle*{4}}
              \put(-37,211){\circle*{4}}
       \put(-73,211){\circle*{4}}
       \put(-42,204){\circle*{4}}
       \put(-70,233){\circle*{6}}
       \put(-39,232){\circle*{4}}

\put(-80,235){$1$}
\put(-40,235){$2$}
\put(-40,195){$2$}
\put(-34,207){$3$}
\put(-75,195){$3$}
\put(-82,207){$1$}

\put(185,220){\circle{100}}
\curve(170,232, 187, 217,  203,211)
\curve(173,204, 185, 207, 198,204)
\curve(201,232,  184, 217, 167,211)

        \put(173,204){\circle*{4}}
              \put(203,211){\circle*{4}}
       \put(167,211){\circle*{4}}
       \put(198,204){\circle*{4}}
       \put(170,233){\circle*{6}}
       \put(201,232){\circle*{4}}

\put(160,235){$1$}
\put(200,235){$2$}
\put(200,195){$3$}
\put(207,207){$1$}
\put(165,195){$3$}
\put(158,207){$2$}

\put(245,220){\circle{100}}
\curve(230,232, 245, 228, 261,232)
\curve(227,211, 242, 211,  258,204)
\curve(233,204,  248, 211, 263,211)

        \put(233,204){\circle*{4}}
              \put(263,211){\circle*{4}}
       \put(227,211){\circle*{4}}
       \put(258,204){\circle*{4}}
       \put(230,232){\circle*{6}}
       \put(261,232){\circle*{4}}

\put(220,235){$1$}
\put(260,235){$1$}
\put(260,195){$3$}
\put(267,207){$2$}
\put(225,195){$2$}
\put(218,207){$3$}

\curve(-55, 190,  -35, 105)
\curve(-45, 190,  25, 105)
\put(-75, 150){$(2,3)$}
\put(-34, 177){$(3,2)$}

\curve(0, 190, -25, 105)
\curve(10, 190, 90, 105)
\put(-21, 115){$(1,2)$}
\put(5, 160){$(2,1)$}

\curve(60, 190, 35, 105)
\curve(70, 190, 145, 105)
\put(16,130){$(1,3)$}
\put(65,160){$(3,1)$}

\curve(120, 190, 100, 105)
\curve(130, 190, 155, 105)
\put(91,175){$(2,1)$}
\put(142,148){$(1,2)$}

\curve(180, 190, 45, 105)
\curve(190, 190, 210, 105)
\put(67,137){$(1,2)$}
\put(192,181){$(2,1)$}

\curve(240, 190, 115, 105)
\curve(250, 190, 220, 105)
\put(161,128){$(3,2)$}
\put(233,130){$(2,3)$}

\put(95,75){\circle{100}}
\curve(80,87, 95, 83, 111,87)
\curve(77,66, 83,66,  83,59)
\curve(108,59, 108,66, 113,66)

        \put(83,59){\circle*{4}}
              \put(113,66){\circle*{4}}
       \put(77,66){\circle*{4}}
       \put(108,59){\circle*{4}}
       \put(80,88){\circle*{6}}
       \put(111,87){\circle*{4}}

\put(70,90){$1$}
\put(113,90){$1$}
\put(117,63){$2$}
\put(109,49){$2$}
\put(75,49){$3$}
\put(67,63){$3$}

\put(155,75){\circle{100}}
\curve(140,87, 168,59)
\curve(171,87, 168, 76, 173,66)
\curve(143,59, 143,66, 137,66)

        \put(143,59){\circle*{4}}
              \put(173,66){\circle*{4}}
       \put(137,66){\circle*{4}}
       \put(168,59){\circle*{4}}
       \put(140,88){\circle*{6}}
       \put(171,87){\circle*{4}}

\put(130,90){$1$}
\put(173,90){$2$}
\put(177,63){$2$}
\put(169,49){$1$}
\put(137,49){$3$}
\put(127,63){$3$}

\put(35,75){\circle{100}}
\curve(20,87, 23, 76, 17,66)
\curve(23,59, 35, 63, 48,59)
\curve(53,66, 48, 76,  51,87)

        \put(23,59){\circle*{4}}
              \put(53,66){\circle*{4}}
       \put(17,66){\circle*{4}}
       \put(48,59){\circle*{4}}
       \put(20,88){\circle*{6}}
       \put(51,87){\circle*{4}}

\put(10,90){$1$}
\put(53,90){$2$}
\put(57,63){$2$}
\put(50,49){$3$}
\put(15,49){$3$}
\put(8,63){$1$}

\put(-25,75){\circle{100}}
\curve(-43,66,  -37, 76, -40,87)
\curve(-37,59, -9,87)
\curve(-7,66, -13,66, -12,59)

        \put(-37,59){\circle*{4}}
              \put(-7,66){\circle*{4}}
       \put(-43,66){\circle*{4}}
       \put(-12,59){\circle*{4}}
       \put(-40,88){\circle*{6}}
       \put(-9,87){\circle*{4}}

\put(-50,90){$1$}
\put(-7,90){$2$}
\put(-4,63){$3$}
\put(-10,49){$3$}
\put(-45,49){$2$}
\put(-52,63){$1$}

\put(215,75){\circle{100}}
\curve(200,87, 215, 83,  231,87)
\curve(203,59, 215, 62, 228,59)
\curve(233,66, 215, 69, 197,66)

        \put(203,59){\circle*{4}}
              \put(233,66){\circle*{4}}
       \put(197,66){\circle*{4}}
       \put(228,59){\circle*{4}}
       \put(200,87){\circle*{6}}
       \put(231,87){\circle*{4}}

\put(190,89){$1$}
\put(233,89){$1$}
\put(238,62){$2$}
\put(231,50){$3$}
\put(194,50){$3$}
\put(188,60){$2$}

\curve(75, -5, -20, 47)
\curve(85, -5, 35, 47)
\curve(95,-5, 95, 47)
\curve(105,-5, 150, 47)
\curve(115, -5, 200, 47)

\put(-3,22){$L$}
\put(40,22){$L$}
\put(83,22){$L$}
\put(118,22){$L$}
\put(177,22){$L$}

\put(92,-25){$\hat{0}$}
          
        \end{picture}
        \caption{$P_3$ and its edge labeling inducing a shelling} 
        \label{P3}
\end{figure}
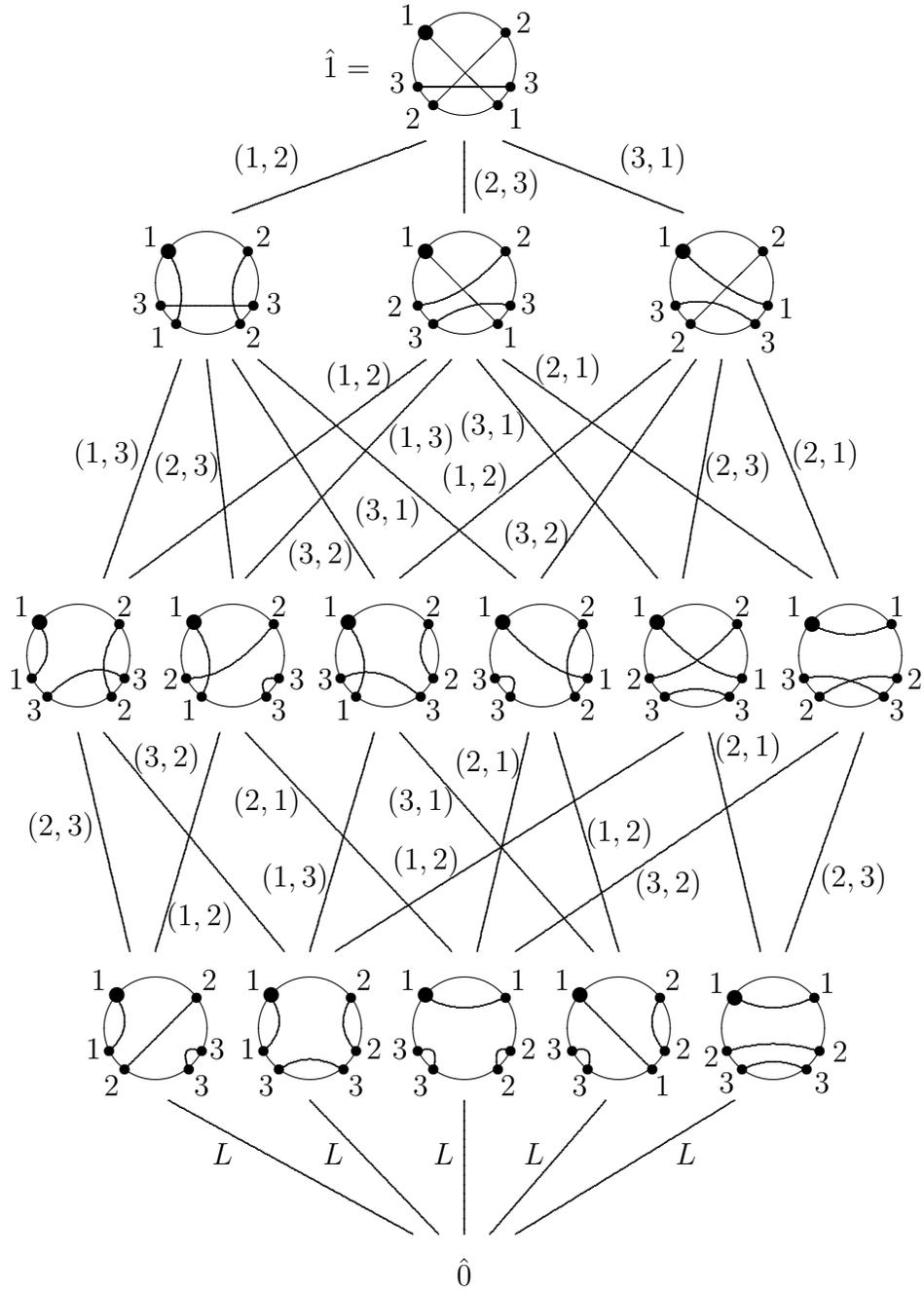

With this encoding in hand, we next describe the dual EC-labeling $\lambda $ for $P_n$ from \cite{elect}.  
Given a cover relation $u\lessdot v$ which replaces subsequence  $ijij$ in $v$ with $ijji$ in $u$, we assign the label %this the poset edge label 
$\lambda (u,v) = (i,j)$.    %We label the
Given a  cover relation $u\lessdot v$ which replaces $ijij$ in $v$ with $iijj$ in $u$, we assign the label $\lambda(u, v)=(j, i)$. % label this with poset edge label $(j,i)$.  
Label each cover relation which proceeds downward from a strand diagram not having any crossings to the element  $\hat{0}$ with the label $L$.  Totally order the label set as follows:
\begin{itemize}
\item  For $i<j$ we have $(i,j) < L < (j,i)$.
\item We make the label $(i_1,j_1)$ smaller than the label $(i_2,j_2)$ for $i_1<j_1$ and $i_2<j_2$ if and only if we either have $i_1<i_2$ or we have $i_1=i_2$ and $j_1<j_2$.  In other words, we order these labels lexicographically.
\item We make the label $(j_1,i_1)$ smaller than the label $(j_2,i_2)$ for $i_1<j_1$ and $i_2<j_2$ if and only if we either have $j_1>j_2$ or we have $j_1=j_2$ and $i_1>i_2$.  In other words, we order these labels in the reverse order to lexicographic order.
\end{itemize}
It is proven  in \cite{elect} that this edge labeling  $\lambda $ 
is a dual EC-labeling for the uncrossing order $P_n$  (i.e. with label sequences proceeding from top to bottom in $P_n$).

\begin{thm}\label{uncrossing-is-CL}
The uncrossing poset $P_n$  is  dual CL-shellable.
\end{thm}

\begin{proof}
By Corollary ~\ref{ConsistECimpliesCL}, %{CL-iff-CC}, 
it suffices to show that the dual  EC-labeling $\lambda $  for $P_n$ given in \cite{elect} is self-consistent.  To this end,
we will  show that  the labels on the cover relations downward from a given element $v\in P_n$ are all distinct.  This will immediately  imply that  the resulting edge labeling for $P_n^*$ has the UE property.  Lemma ~\ref{UEimpliesSC} shows that the UE property implies that the labeling is self-consistent.

Given any fixed element $v\ne \hat{0}$ in $P_n$ and given a label $K$ on an edge down from $v$, we may determine the unique $u$ for which $\lambda(u, v)=K$. %one may uniquely determine  from the label  $K$  on a cover relation downward from $v$ what is the unique element $u$ such that  $\lambda (u,v) = K$.  
 This holds because  we can read from the label $K$  the names of the strands being uncrossed as well as which way they are being uncrossed. %as we proceed down a cover relation from the strand diagram  $v$ within   a saturated chain in $P_n$. 
 As these two pieces of information are the only data used to determine $u$ from $v$ when $u \lessdot v$, this means the labels on the cover relations downward from $v$ uniquely determine each $u$ and thus must be  distinct. %since the label uniquely determines $u$.  This %sort of label distinctness implies that the smallest label occurs only once, hence 
This implies this dual EC-labeling for $P_n$ has the UE property.  
%By Lemma \ref{UEimpliesSC}, this implies we have a self-consistent EC-labeling for $P_n^*$.  % and hence implies that the dual EC-labeling is self-consistent.  
%Since any EC-labeling is a CC-labeling, we have a self-consistent CC-labeling for the dual uncrossing order.  By  Theorem  ~\ref{ConsistCCimpliesCL}, this  labeling may be transformed into a dual CL-labeling for  $P_n$.
\end{proof}
% a class of posets which were not known to be CL-shellable but which indeed are CL-shellable by virtue of our equivalency of CL-shellability to CC-shellability.

\section{Further results, remarks, and open questions}\label{Consequences-section}

We conclude with some further connections between our work and results in the literature, some open questions, and assorted other remarks.

\subsection{Classical EL-labelings and CL-labelings with the UE property}\label{CL-UE}

In light of the role of self-consistency % (defined in Section \ref{graosection}) 
in constructing a GRAO or a CL-labeling  from a CC-labeling and  also given the fact that self-consistency  is implied by the  %seemingly more readily checkable  
UE property (see Lemma \ref{UEimpliesSC}),  it is natural to ask how readily checkable  the UE property is.   We already proved that every CL-labeling (and hence every EL-labeling) is self-consistent in Lemma \ref{CLhaveUE}.  We now show how  easy it is % many of the most well-known CL-labelings and EL-labelings indeed have the UE property and the
 to verify the UE property directly  for many of the classic examples of EL-labelings and CL-labelings.  This together with the proof in 
 Section \ref{CC-UE} of the UE property for the EC-labeling for dual uncrossing orders from \cite{elect} may give readers some sense of how the UE property may well be readily  
 verifiable for many of 
the sorts of CC-labelings and EC-labelings that  people are likely to construct for posets of interest in the future.   

\begin{prop}
Let $L$ be a geometric lattice and let $a_1,\dots ,a_t$ be any total order on the atoms of $L$. For each $x\in L$, denote by $A(x)$ the set of atoms $a\in L$ satisfying $a\le x$.   
Then the EL-labeling $\lambda $ for  geometric lattices which  labels each cover relation $u\lessdot v$ with the smallest atom in $A(v)\setminus A(u)$ satisfies the UE property.
\end{prop}

\begin{proof}
Notice for each $u\in L$ and any cover relation $u\lessdot v$  % upward  from $u$ 
with edge label $a_i$ that we must have % this  implies 
$v = u\vee a_i$.  In particular, $v$ may be recovered from $u$ and 
$\lambda (u,v)$, implying that the labels on the cover relations upward from $u$ are all distinct from each other.  Thus, the smallest such label can only occur once, as needed.
\end{proof}

\begin{prop}
Let $L$ be a finite, distributive lattice. Consider the EL-labeling obtained by first interpreting $L$ as the poset $J(P)$ of order ideals in $P$ ordered by containment and then labeling the cover 
relation $I\lessdot I\cup \{ x\} $ with the label $x$, using any linear extension on $P$ to order the set of edge labels.  Then this EL-labeling $\lambda $ satisfies the UE property.
\end{prop}

\begin{proof}
Again the proof is based on observing for the cover relations upward from $u$  that any element $v\in L$ covering $u$  may be determined from  knowing $u$ and $\lambda (u,v)$.  This  implies that the edge labels  on cover relations upward from a fixed element $u$ are all distinct.  Since the smallest such label only occurs once, this implies  the UE property.
\end{proof}

Next we consider an example in which some %not every 
labels upward from a poset element $u$ occur %only once, 
more than once, but  where the smallest upward from $u$ % does 
occurs only once:

\begin{prop}
Let   $\Pi_n$ be the partition lattice, namely  the lattice of  set partitions of $\{ 1,2,\dots ,n\} $ ordered by refinement.  That is, $\Pi_n$ has $\pi \le \tau $ if and only if 
$\pi $ is a refinement 
of $\tau $ or, in other words, each block of $\pi $ is contained in a block of $\tau $.  
Consider the EL-labeling $\lambda $  of $\Pi_n$  given by 
$$\lambda (\pi,\tau) = \max (\min B_1,\min B_2)$$ where $B_1$ and $B_2$ are the two blocks of $\pi$ being merged to obtain $\tau$. 
%by merging blocks $B_1$ and $B_2$ of $u$.  
Then $\lambda $ has the UE property.
\end{prop}

\begin{proof}
%In this case, there may be repetition in the labels upward from $u$, but we will show that there is no repetition in the smallest such label.  To see this, o
Order % the smallest elements of the various blocks in a set partition $u$ from smallest to largest.  Now order 
the  %corresponding
 blocks $B_1,\dots ,B_r$ of  a set partition $\pi$ %in this manner, i.e. where 
 in such a way  
 that the smallest element of $B_i$ is smaller than the smallest element of $B_j$ for  $i<j$.    By definition of $\lambda $, the smallest edge label  that is achievable 
 on any cover relation upward from $u$
is $\min B_2$.  The only way this label arises is  by merging blocks $B_1$ and $B_2$.  In particular, this smallest label only occurs on one cover relation upward from 
$\pi$.
% only occurs on the cover relation $u\lessdot v$ where $v$ is obtained from $u$ by merging the blocks $B_1,B_2$ having the overall two  smallest minimal elements.
%  This cover relation has label $\min B_2$, and one may see by construction no other cover relation upward from $u$ may have this label or a smaller label. 
 This completes the proof of the UE property for this  EL-labeling of  $\Pi_n$. %in this example.
\end{proof}

%\commentplh{In response to referee, we can say that we now review the labeling before turning to the proof, so as to make this more self-contained.}  
%Below we need to either state what is the labeling of Bj\"orner and Wachs in the next proof  more precisely or follow the referee's alternative suggestion of saying the next proposition may be easily derived from ... \cite{BW82}..I think the issue is that the proof was written sloppily and would be hard to follow without going to \cite{BW82}...so I made an attempt to fix  it.}

%\commentph{Should we first describe the CL-labeling of Bj\"orner and Wachs carefully, so then the statement of the proposition is self-contained in that it does not rely on stuff in the proof of the proposition and/or in Bj\"orner-Wachs?  I'm not at all sure we should do this, but perhaps.}

Next we will verify  the UE property for  the  chain-edge labeling  $\lambda $
described next  for  the dual poset to the Bruhat order of any finite Coxeter group $W$.  
Bj\"orner and Wachs
introduced this chain-edge labeling and proved it was a CL-labeling in \cite{BW82}.   In order  to work with 
the dual poset to Bruhat order, 
we will speak of maximal chains proceeding from top to bottom in Bruhat order, allowing the label on a cover relation $u\lessdot  v$ % in Bruhat order 
to  depend on $u,v$ and the choice of saturated chain downward from $\hat{1}$ to $v$.  We call  each saturated chain  from $\hat{1}$ to $v$
a co-root.   We sometimes denote a cover relation $u\lessdot v$ by $v\rightarrow u$ below to highlight the fact that we are proceeding down each  cover
relation.  %,  since it is dual to what we  typically call a root.  

%\commentph{Perhaps it could help to use the notation $s_{i_{j_1}}\cdots s_{i_{j_{l(v)}}}$ below  
%for our subexpression of $s_{i_1}\cdots s_{i_d}$ that is a reduced expression 
%for $v$, so then the label is the index $j_r$ of the letter being deleted to obtained a reduced expression for $u$, which is not to be confused with $r$.  I'm not sure if this is clearer than our usage of $t_1\cdots t_d$ below, but it's an alternative approach to this that might be preferable.}

%Let  $w_0$ denote  the longest element of $W$.  
Choose a  reduced expression $R_{w_0}=s_{1}\cdots s_{d}$  for the longest element, denoted $w_0$, in $W$.   
 Label each cover relation $w_0\rightarrow v$ %downward from $w_0$
by the  position $j$  of the  unique letter  $s_j$ that may be deleted from 
$R_{w_0}$ %=s_{1}\cdots s_{d}$
 to obtain a reduced expression  for $v$.  Let $R_{v,r} = s_{1}\cdots s_{j-1}\hat{s}_{j}s_{j+1}\cdots s_{d}$  be this reduced expression for $v$, with $r$ denoting 
 the co-root downward from $\hat{1}$ to  $v$.
 % that is allowed to depend on the choice of co-root $r$. 
%\commentph{ Denote by $R(v,r)$ this reduced expression for $v$ where we also keep track of the original positions that the remaining letters
%each had within $R_{w_0}$; more specifically, l}
%Let $R(v,r)=t_{i_1}\cdots t_{i_d}$ where $t_{i_j}=e$ and $t_{i_k}=s_{i_k}$  for $k\ne j$.  
Continue  proceeding  down a saturated chain from $\hat{1}$  in Bruhat order,  labeling each downward cover relation 
$y \rightarrow x $  that we encounter  in turn % \commentph{
as follows.   Assume  inductively that we have chosen a co-root $r$ downward from $\hat{1}$ to  $y\in W$, and assume that  $y$ and $r$ together have been used to
specify  a subexpression $R_{y,r}=s_{i_1}\cdots s_{i_k}$ of $R_{w_0}$ that is a reduced expression for $y$. The set $\{ i_1,\dots ,i_k\} $ of indices in $R_{y,r}$ 
indicates the
positions of the letters in $R_{w_0}$ that appear  in $R_{y,r}$;   the co-root  $r$ uniquely specifies this set of indices.  Label the cover relation
downward from $y$ to  $x$ for our given choice of co-root $r$ with the position $i_l$ within $R_{w_0}$ of the unique letter $s_{i_l}$ that may 
be deleted from $R_{y,r}$ to obtain a reduced expression   for $x$. % that we denote by 
%$R_{x,r'}$.  % \commentgs{Something seems wrong with this next sentence starting at "thereby", but I'm not sure what it's supposed to say.} 
Denote by   $R_{x,r'}$ this reduced expression
$s_{i_1}\cdots s_{i_{l-1}}\hat{s}_{i_l}s_{i_{l+1}}\cdots s_{i_k}$  %be this reduced expression
 for $x$ that is determined by $x$ and $r'$ where % the data 
%consisting of  $x$ together 
%with
$r'$ is  the  co-root  %given by obtained by letting  
$r' := r\cup \{ x \} $.  
%That is, 
 %\commentph{I'm still working on the next bit here -- it probably should still be written better.}
%let $\lambda (y,x;r) = i_j$ in this case.
%  be  
% the position within  $R_{w_0}$ % $R(y,r)$
%   of the unique  letter that may be deleted from  $R_{y,r}$ to  obtain a reduced expression for $x$.
   %, specifically using as the label 
 % the position that this letter being 
% deleted had in the original word $R_{w_0}$; l
% Let $R_{x,r'}$ be this resulting  reduced expression for $x$ which may also depend on the choice of co-root $r'$; here we use $r'$  to 
% keep  track for each  letter of $R_{x,r'}$  what position it had appeared in within $R_{w_0}$.  

As an example, let  $W$ be the symmetric group $S_3$ generated by simple reflections $a=(1,2)$ and $b=(2,3)$.  The cover relation $aba \rightarrow ba$ is labeled 1 while the cover relation $ba \rightarrow a$ proceeding further down this maximal chain is labeled 2; the latter label reflects the fact that $b$ is the second letter in 
 $aba$ even though it is the first letter in $ba$.  The cover relation $a \rightarrow e$ is either labeled 1 or 3, depending on the choice of co-root 
 downward from $aba$ to $a$.

\begin{prop}\label{Bruhat-UE}
Consider the dual CL-labeling for Bruhat order for any finite Coxeter group $W$  that was given by  Bj\"orner and Wachs in \cite{BW82} and is recalled above.   This
 CL-labeling for the dual poset to Bruhat order  has the UE property.
\end{prop}

\begin{proof}
Let $s_{1}\cdots s_{d}$ be our chosen reduced expression $R_{w_0}$ for  $w_0$.  
%For any $v\in W$ and any choice of saturated chain downward from $w_0$ to $v$, let 
%$t_{1}\cdots t_{d}$ be the  expression for $v$  where $t_{j}=e$ for each $j$ such   that $s_{j}$  
%has been deleted from $R_{w_0}$ to obtain $R_v$ and where $t_{j} = s_{j} $ for each $j$ such that $s_{j}$ has not been deleted from $R_{w_0}$ to obtain 
%$R_v$.  
%\commentph{We need to be more precise about which reduced expression we are using to go down from $v$ to $u$, i.e. we pick one for the longest element and then use the fact that we have a chain-labeling and hence a series of deletions which specifies a reduced expression for $v$...}
%First we recall the dual CL-labeling from \cite{BW82}.   
%\commentph{What had been below had not been  a correct depiction of how the labeling works, since $v$ does not have all $d$ letters in it unless $v=w_0$.  
%This needed fixing, and I should still proofread this yet again.}
%Since t
The label on the cover relation  downward from an element $v$ to an element $u$ for a choice of co-root $r$   %covered by $v$
 is the position  $i_j$ within $s_1\cdots s_d$  of the unique  letter that may be  deleted from $R_{v,r} = s_{i_1}\cdots s_{i_k}$ 
 to obtain a reduced expression for $u$, as described above.
 % $t_{j}$ being changed from $s_{j}$ to $e$ within  $R(v,r) = t_{1}\dots t_{d}$,  %\commentph{Replace rest of this phrase with something more precise such as: ``the fixed reduced expression for $v$ that we obtain from our chosen reduced expression for $w_o(W)$ by the series of deletions'' carried out by the maximal chain in proceeding downward from $w_o(W)$ to $v$''} 
%a  fixed reduced expression for $v$, 
%using this same reduced expression for every cover relation downward from $v$, 
 Therefore the  label $i_j$  uniquely determines the reduced expression $R_{u,r'} = s_{i_1}\cdots s_{i_{j-1}}\hat{s}_{i_j}s_{i_{j+1}}\cdots s_{i_k}$  
 that  % = s_{i_1}\cdots s_{i_{j-1}}\hat{s}_{i_j}s_{i_{j+1}}\cdots s_{i_d}$ 
 we get by deleting $s_{i_j}$ from $R_{v,r}$, letting  
  %replacing  $s_{j}$ with $e$ in $R(v,r)$  and then deleting  all of the copies of the identity element $e$; 
 $r' = r\cup \{ u\} $.  %  denotes the new co-root obtained from
 %$r$ by appending $u$ to it. 
 One may determine  $u$ from $v, r,$ and $i_j$   as the unique Coxeter group element   for which $R_{u,r'}$ is an expression.  This shows that   $u$ is uniquely determined by  $v$, $r$ and  the label $\lambda_r (v,u)$. %  \commentph{ $\lambda (v,u;r)$}. 
  Thus, for any fixed $v$ and $r$, the value of the label on a cover relation  $v\rightarrow u$
 downward from $v$ 
 uniquely determines  this  element  $u$ that is covered by $v$.  
 Since  all cover relations downward from $v$  for fixed co-root $r$ lead to distinct elements, the labels on these cover relations  must  therefore be distinct. 
 This implies that  this  dual CL-labeling %In particular, distinctness of all labels  implies 
has the UE property.
\end{proof}

\subsection{Obtaining a self-consistent CC-labeling directly from a GRAO}\label{GRAO-SCCC-section}

Our next result  shows that one does not need to use the atom reordering process to convert a GRAO to an RAO in order to show that every finite bounded poset with a 
GRAO has a self-consistent CC-labeling. It shows that alternatively, one may obtain a self-consistent CC-labeling directly  from any GRAO.  % in a more direct manner.   

\begin{defn}
\label{compatible-def}
Any CC-labeling
$\lambda $  for a finite bounded poset $P$ gives rise to a total order on the maximal chains of $P$ by ordering the label sequences  for  the maximal chains 
lexicographically.  On the other hand, any   % \commentph{fix next bit of notation}
GRAO %  \sout{$\Gamma $}
 for this same poset  $P$ gives rise  to a total order on the maximal  chains of $P$ as we describe shortly.  Let $\Gamma $ be such
a GRAO.  
We say that $\Gamma $ is {\bf consistent} with $\lambda $ if $\Gamma $ and $\lambda $ both give rise to 
the same total order on the maximal  chains of $P$.  
  
  Now we describe  the total order on maximal chains  induced by $\Gamma $.
Given any two maximal  chains $m_1,m_2$  of $P$, let  $x$ (resp. $x'$)  be the lowest 
element in $m_1$ (resp. $m_2$)  that is not in 
$m_2$ (resp. $m_1$).   Let $u$ be the element in both $m_1$ and $m_2$ that is covered by both $x$ and $x'$.
% where both  $m_1$ and $m_2$.  
Denote by $r$ the unique  saturated chain from $\hat{0}$ to % some element 
$u$  %that is covered by both  $x$ and $x'$ 
with the property that $r$ is contained in both $m_1$ and $m_2$.  
In this case, $m_1$ comes earlier than $m_2$ in the total order on maximal chains induced by $\Gamma $ if and only if $x$ is an earlier atom than $x'$ in the GRAO for 
$[u,\hat{1}]_r$ induced by $\Gamma $.  
\end{defn}

\begin{thm}\label{second-proof} 
Any generalized recursive atom ordering of a finite bounded poset $P$ gives rise to a self-consistent CC-labeling which orders the maximal 
chains of $P$ consistently with the GRAO.
\end{thm}

\begin{proof}
Given a generalized recursive atom ordering (GRAO)  % (resp. GRAO)  
for  the poset $P$,  construct the following 
chain-edge labeling $\lambda $ for $P$.   %that is compatible % consistent (compatible????) 
%with the GRAO chain-atom ordering
In the rooted interval $[u,v]_r$, label the  cover relation  $u\lessdot a_i  $  %upward from $u$
 with the   %with labels $1,2,\dots ,n$, where $n$ is the number of atoms of $[u,v]_r$, as follows.  %, doing so 
%Assign the 
label $i$, % is assigned to the cover relation $u\lessdot a_i$ 
where $a_i$ is  %upward from $u$ to 
the $i$th atom in the GRAO % atom ordering  
of $[u,v]_r$ that is induced by the GRAO of $P$ as justified in Lemma ~\ref{GRAOrestricts}.% given by the GRAO.  

By construction, $\lambda $ has the property that distinct cover relations upward from $u$ within $[u,v]_r$ will have distinct labels on them, ensuring that this 
chain-edge labeling will have the UE property and hence be self-consistent.
Also by construction, $\lambda $ has the  following  pair of properties required for a CC-labeling (see Definition \ref{cc}): 

\begin{enumerate}
\item The saturated chains on a rooted interval $[u,v]_r$ have distinct label sequences.
\item  No saturated chain  in $[u,v]_r$
has as its label sequence a prefix of the label sequence of another saturated chain in $[u,v]_r$.
\end{enumerate}

%\commentgs{Should the C's in this next paragraph be M's? }\commentph{Not necessarily. 
%$M$ has a particular meaning which  is 
%different than what we need here.  However, it is not problematic to make this $M$, so I will do that}
For any rooted interval $[u,v]_r$, consider any saturated chain  $M$ 
from $u$ to $v$ other than the one % saturated chain from $u$ to $v$ 
 which comes earliest with respect to $\lambda$ %the GRAO 
 for the choice of root $r$.  
We will show that $M$ must have a topological descent with respect to the labeling $\lambda $,  %and the root $r$, 
thereby proving that $\lambda $ is  indeed a self-consistent CC-labeling.

Suppose there exists some $u<v$ in $P$ and some root $r$ with   a saturated chain  $M = u\lessdot u_1\lessdot u_2\lessdot \cdots \lessdot u_s \lessdot v$ in $[u,v]_r$  that is not the lexicographically earliest saturated chain of $[u,v]_r$ but  where $M$ nonetheless does not have any topological descents.  Among all such triples $u,v, r$, choose one for which the length of the longest saturated chain from $u$ to $v$ is as small as possible.  This choice 
ensures  % for every $r$ 
that   $u\lessdot u_1$ does not belong to the lexicographically smallest saturated chain in $[u,v]_r$, since 
otherwise, $M$ %restricted to $[u_1,v]_{r\cup u}$ would give a contradiction to our minimality assumption for the length of the longest saturated chain from $u$ to $v$ in choosing 
restricted to $[u_1,v]_{r\cup u}$ would have the same properties within an interval with strictly shorter longest  saturated chain, contradicting the sort of minimality $[u,v]_r$ was 
chosen to have.

We may deduce that $u_1$ does not come first amongst the atoms of $[u,v]_r$ in the %atom ordering given by the 
GRAO, as $M$ is not the lexicographically earliest saturated chain in $[u,v]_r$.
Consider first  the case where $u_2$ is the  first atom in the GRAO of $[u_1,v]_{r\cup u_1}$. % in the GRAO. %(resp. GRAO) 
 Condition (ii) in Definition \ref{grao} % \ref{graot} % (resp. GRAO)
then ensures that there  exists an atom $u_1'$ of $[u, v]_r$ %with $u\lessdot u_1'  < v$ 
where $u_1'$ comes earlier than $u_1$ in our GRAO  %(resp. GRAO)  
  and  that there exists some element $z$ such that $u_1 \lessdot z < v$  and 
$u_1'  <  z < v$.  Lemma %s \ref{GRAOrestricts}  
\ref{graoibequiv}  then  % condition (i) in the definition of GRAO-2 %  (but not in the definition of GRAO) 
guarantees  %for  the earliest atom above $u_1$ in the rooted  interval $[u_1,v]_{r\cup u}$, namely for $u_2$, 
that $u_2$ must satisfy $u_1''  <  u_2$ % <  v$ %y???$  
for some atom  $u_1''$ in $[u,v]_r$  with $u_1''$ coming earlier than $u_1$ in the GRAO.
This implies that  $(\lambda(u, u_1), \lambda(u_1, u_2))$ is a topological descent, contradicting our assumption that $M$ does not have any topological descents. 

Next consider the case where $u_2$ is not the smallest atom in the GRAO %-2 % (resp. GRAO) 
 of $[u_1,v]_{r\cup u_1}$.  Our minimality assumption 
 in choosing  $[u, v]_r$ above implies that 
 $M$ restricted to  $[u_1,v]_{r\cup u_1}$ must have a topological descent. % by virtue of  % this restricted $M$
  But this topological descent also gives a topological descent in $M$ itself, again giving  a contradiction.
  This completes the proof that  % we have justified that % the chain-labeling 
  $\lambda $ % derived from the GRAO
   is % indeed % meets the requirement needed to be 
a self-consistent CC-labeling.
\end{proof} 

\begin{rmk}
The key difference between this construction and our use of a GRAO to construct an RAO (and thereby a CL-labeling and hence a self-consistent CC-labeling) is that this construction we have just given goes directly from a 
GRAO to a self-consistent CC-labeling without needing to invoke the atom reordering process or construct an RAO as an intermediate step.
\end{rmk}

%\commentph{I just moved the Bj\"orner-Wachs  from  here to the end of section 2.}

\subsection{Dual CC-shellability of face lattices of $d$-complexes}\label{d-complexes}

Bj\"orner proved in  \cite{Bj84}  that  shellability  of  a $d$-complex
is equivalent to  dual CL-shellability of its  face lattice.  

Recall that a  \textbf{d-complex} is a polyhedral complex whose  maximal faces  are all  of dimension $d$ for a fixed $d$.   The  \textbf{face lattice} $L(\Delta )$ of  a
$d$-complex $\Delta $   is the partial order on its cells in which one cell is less than or equal to another if the former is in the closure of the latter, with a unique minimal element
$\hat{0}$ representing the empty cell and with an additional maximal element denoted by $\hat{1} $ appended if there is not a unique maximal cell  in $\Delta $.

\begin{prop}\label{topol-d-complex}
Shellability of a $d$-complex $\Delta $  is equivalent to $L(\Delta )$ admitting a self-consistent dual CC-labeling.  This is also equivalent to 
$L(\Delta )$ admitting a self-consistent dual TCL-labeling. % for its face lattice.  
\end{prop}

\begin{proof}
We use Bj\"orner's result  from \cite{Bj84} that a $d$-complex $\Delta $ is shellable if and only if its face lattice  $L(\Delta )$
  is dual CL-shellable.  We combine this  with Theorem \ref{TFAE-theorem}  applied to $L(\Delta )^*$, % Corollary \ref{CL-iff-CC}  %applied to $L(\Delta)^*$ 
  giving the equivalence of CL-shellability to self-consistent  TCL-shellability and to self-consistent CC-shellability.
\end{proof}

As an application of % Proposition \ref{topol-d-complex} (or of 
Bj\"orner's aforementioned result from \cite{Bj84}, 
we may deduce directly  from Theorem \ref{uncrossing-is-CL} that the $d$-complexes having  the uncrossing posets  as their posets of closure relations are shellable complexes.    
%\commentph{Proposed alternative: These $d$-complexes have a unique maximal cell. 
 The recursive nature of the definition of shelling for polyhedral complexes in particular thereby yields  %means that in particular this yields 
shellability of the boundary of the unique maximal cell in  these $d$-complexes coming from uncrossing orders.
%\sout{The recursive nature of the definition of shelling for polyhedral complexes means that in particular this yields shellability of the boundary of the unique maximal cell in such a $d$-complex.  }
%\gs{The recursive nature of the definition of shelling for polyhedral complexes %means that 
%in particular 
%%%this 
%yields shellability of the boundary of the unique maximal cell in such a $d$-complex.  }

\subsection{A possible further direction}\label{further-section} 
The notions of CC-shellability and CL-shellability are generalized  in \cite{hersh}  from poset order complexes to more general balanced simplicial  complexes (and balanced 
boolean cell complexes).  % ~\cite{hersh}.  
The class of balanced  quotient complexes $\Delta (B^{2n})/S_2\wr S_n $ is proven % re are results proving specific classes of balanced complexes
 to be CC-shellable in \cite{hersh}  but is  not %without also  being 
 proven to be CL-shellable.  

\begin{qstn}
  It remains open whether  self-consistent CC-shellability is equivalent to CL-shellability in the more general setting of balanced complexes. 
  \end{qstn}
  
  This could be an interesting avenue for further study, with potential applications to quotient cell 
  complexes such as $\Delta (B^{2n})/S_2\wr S_n$ (see \cite{wreath} for non-shellability of
  $\Delta (B^{kn})/S_k\wr S_n$ for $k>2$ and see \cite{hersh} for CC-shellability when $k=2$).

\end{document}